\documentclass{amsart}

\usepackage[T1]{fontenc}
\usepackage[utf8]{inputenc}
\usepackage{lmodern}
\usepackage[english]{babel}
\usepackage[autostyle]{csquotes}

\usepackage{amsmath}
\usepackage{amsfonts}
\usepackage{mathtools}
\usepackage{graphicx}
\usepackage{mathrsfs}
\usepackage{amsthm}
\usepackage{amssymb}
\usepackage{centernot}
\usepackage{enumerate}
\usepackage{mathbbol}
\usepackage{tikz}
\usepackage{tikz-cd}
\usepackage{verbatim}
\usepackage[mathcal]{euscript}
\usepackage[hidelinks]{hyperref}
\usepackage[margin=2.5cm]{geometry}

\newtheorem{theorem}{Theorem}[section]

\newtheorem{lemma}[theorem]{Lemma}

\newtheorem{proposition}[theorem]{Proposition}

\newtheorem{definition}[theorem]{Definition}

\theoremstyle{remark}

\theoremstyle{definition}
\newtheorem{remark}[theorem]{Remark}
\newtheorem{example}[theorem]{Example}

\newsavebox{\pullback}
\sbox\pullback{%
\begin{tikzpicture}%
\draw (0,0) -- (0ex,1ex);%
\draw (0ex,1ex) -- (1ex,1ex);%
\end{tikzpicture}}

\def\C{{\mathbb C}}
\def\Q{{\mathbb Q}}
\def\Z{{\mathbb Z}}

\def\R{{\mathbb R}}

\def\Oe{{\mathcal O_E}}

\def\ul#1{{\underline{#1}}}

\def\mc#1{{\mathcal{#1}}}
\def\mb#1{{\mathbb{#1}}}
\def\mf#1{{\mathfrak{#1}}}
\def\mr#1{{\mathrm{#1}}}
\def\im{{\mr{im}}}

\renewcommand{\hom}{\mr{Hom}}

\DeclareMathOperator{\edom}{End}

\DeclareMathOperator{\spec}{Spec}

\newcommand{\btone}{BT\textsubscript{1}}

\title[Generalised \texorpdfstring{$\theta$}{Theta} operators]{Generalised \texorpdfstring{$\theta$}{Theta} operators on unitary Shimura varieties}

\author{Lorenzo La Porta}
\address{Department of Mathematics, King's College London}
\email{lorenzo.laporta@protonmail.com}

\begin{document}


\begin{abstract}
The main result of this paper is the construction of a new class of weight shifting operators, similar to the theta operators of \cite{EFGMM}, \cite{deshagor19} and others, which are defined on the lower Ekedahl-Oort strata of the geometric special fibre of unitary Shimura varieties of signature $(n-1, 1)$ at a good prime $p$, split in the in the reflex field $E$, which we assume to be quadratic imaginary. These operators act on certain graded sheaves which are obtained from the arithmetic structure of the EO strata, in particular the $p$-rank on each stratum. We expect these operators to have applications to the study of Hecke-eigensystems of $(\mr{mod}\, p)$ modular forms and generalisations of the weight part of Serre's conjecture.
\end{abstract}

\maketitle

\tableofcontents

\section{Introduction}

\subsection{Motivation}
The theory of the classical theta operator was instrumental in the proof of the weight part of Serre's modularity conjecture of \cite{edixweight}. Edixhoven's proof relied, in particular, on the study of the \emph{$\theta$-cycles} of Tate and Jochnowitz, introduced in \cite{joch}. Since then, much work has been devoted to extending the construction of this operator to other Shimura varieties, with an eye towards generalisations of Serre's conjecture, or to gain insight in the {Langlands programme} $(\mr{mod}\,p)$ in a broader sense. 
\par Interesting results have been obtained in the case of Hilbert modular varieties, starting with work of Katz, in \cite{katzpadicCMfields}. Following Katz's construction of \emph{partial Hasse invariants}, Andreatta \& Goren, in \cite{andreattagoren}, constructed \emph{partial theta operators} and described their kernels and effects on the \emph{weight filtration}. These results have subsequently been improved upon and generalised further, see for instance \cite{diamondgeoweightshifting2021}.
\par In \cite{yamauchi2021weight}, Yamauchi constructed theta operators for Siegel modular forms, in degree $2$, and managed to study their theta cycles.
\par In \cite{deshagor19}, de Shalit \& Goren, building on their previous work \cite{deshagor16}, constructed $(\mr{mod}\, p)$ and $p$-adic theta operators on certain unitary Shimura varieties.
\par At the same time, Eischen, Mantovan and others, in a series of papers, \cite{EFGMM, eiscmant21, eiscmant212}, constructed theta operators on PEL Shimura varieties of types A and C. Their approach is independent from that of de Shalit \& Goren and it uses geometric techniques which, unlike the more classical theory, do not rely on $q$-expansions, or more general Fourier-Jacobi and Serre-Tate expansions. 
\par Most of these works leave the problem of studying theta cycles largely open and, where results are obtained, they seem to depend on the specific context. We believe that, in order to understand theta cycles in greater generality, one could benefit from considering new theta or ``theta-like'' operators, which produce more general weight shifts. Our goal is to present the construction of a new class of \emph{generalised theta operators} that seem to produce exactly the weight shifts that one would expect from a representation-theoretic viewpoint. Our theory of generalised theta operators ties in neatly with the theory of generalised Hasse invariants of \cite{boxer} and \cite{koskgold}.

\subsection{Notations and conventions}
We fix $E$ a quadratic imaginary extension of $\Q$ and $p$ an odd, rational prime, split in $E$. We write $\mb{F}$ for a given algebraic closure of $\mb{F}_p$. We choose, once and for all, a preferred embedding $\sigma \colon E \to \C$, so that $\hom(E, \C) = \{\sigma, \overline{\sigma}\}$, and an element $i = \sqrt{-1} \in \C$. Let $\delta_{E/\Q}$ denote the unique generator of the different ideal $\mathfrak{D}_{E/\Q}$ with positive imaginary part (with respect to our choices of $\sigma$ and $i$). In particular, we have the discriminant $D = D_{E/\Q} = -\mr{N}_{E/\Q}(\delta_{E/\Q}) = -\delta_{E/\Q} \overline{\delta}_{E/\Q} = -\lvert \delta_{E/\Q} \rvert^2$. If we also fix an isomorphism $\C \cong \overline{\Q}_p$, we obtain
\[
    \hom(E, \C) \cong \hom(E, \overline{\Q}_p) \cong \hom(\Oe, \mc{O}^\mr{ur}_{\overline{\Q}_p}) \cong \hom(\Oe/(p), \mb{F}),
\]
the last two isomorphisms depending on the fact that $p$ is unramified in $\Oe$. The last identification induces on $\{\sigma, \overline{\sigma}\}$ an action of the Frobenius automorphism $\mb{F}\to \mb{F}$ by post-composition. Since $p$ is split, this action is trivial. We will also write $\mc{O}_{E, \mr{ur}} \coloneqq \Oe[1/2D]$.  We fix $n\geq 3$ an integer, which will denote throughout the paper the relative dimension of the abelian schemes parametrised by the moduli spaces under consideration.
\par Unless otherwise specified, we assume that all the schemes we work with are locally noetherian.

\subsection{The elliptic case}
\label{introellcase}
We sketch here the construction of the theta operator in the classical case of modular curves, following \cite{katzresult}. This construction is the prototype on which most of the more general definitions are based.
\par Let $N \geq 5$ be an integer prime to $p$. Let $Y_1(N)$, or simply $Y$, be the modular curve of level $\Gamma_1(N)$ over $\mb{F}$. It is a smooth, affine, connected curve over $\mb{F}$, which comes equipped with a universal elliptic curve $\pi \colon \mc{E} \to Y_1(N)$. From this universal object we obtain the invertible sheaf $\ul{\omega} = \pi_\ast \big( \Omega^1_{{\mc{E}}/Y} \big)$, the so-called \emph{Hodge sheaf}. One can consider a projective compactification $Y_1(N) \subset X_1(N)$, which we may simply denote by $X$, and extend $\ul{\omega}$ to $X$.
\par In this setting, \emph{modular forms with coefficients in $\mb{F}$} are the elements of the graded $\mb{F}$-algebra $\mr{M}(N)=\oplus_k \mr{M}_k(N)$, with $\mr{M}_k(N) \coloneqq H^0(X, \ul{\omega}^k)$. The space of \emph{cusp forms} of weight $k$ is $\mr{S}_k(N) \coloneqq H^0(X, \ul{\omega}^k(-C)) \subseteq M_k(N)$, where $C = X \setminus Y$ is the \emph{cuspidal divisor}. On this algebra one can define as usual an action of the Hecke operators $T_l$, for $l \neq p$ any prime, and hence an action of the Hecke algebra they generate.
\par Since we are working in characteristic $p$, over $Y$, we can consider the Verschiebung morphism $V \colon \mc{E}^{(p)} \to \mc{E}$, which by pullback defines $V \colon \ul{\omega} \to \ul{\omega}^{(p)} \cong \ul{\omega}^p$ and hence a section $h \in H^0(Y, \ul{\omega}^{p-1})$. This section extends to a form $h \in M_{p-1}(N)$, which is called the \emph{Hasse invariant}. It has the special property of vanishing with simple zeros precisely at the \emph{supersingular points} $Y^\mr{ss} \subset Y \subset X$ and its $q$-expansion at the cusps is identically $1$. The complement $Y^\mr{ord} = Y \setminus Y^\mr{ss}$ is a dense open subset of $Y$, which is called the \emph{ordinary locus}.
\par 
A key observation in Katz's geometric construction of the classical theta operator is that over the ordinary locus there is a natural splitting for the \emph{Hodge filtration}
\[
\label{H}
\tag{H}
    0 \longrightarrow \ul{\omega} \longrightarrow H^1_\mr{dR}(\mc{E}/Y) \longrightarrow \ul{\omega}^\vee \longrightarrow 0.
\]
Write $H = H^1_\mr{dR}(\mc{E}/Y) = R^1\pi_\ast \Omega^\bullet_{\mc{E}/Y}$ for the (relative) de Rham cohomology of $\pi \colon \mc{E} \to Y$ and let $F \colon H^{(p)} \to H$ be the morphism obtained by pulling back via the relative Frobenius $F \colon \mc{E} \to \mc{E}^{(p)}$. Katz shows that, over $Y^\mr{ord}$, a natural complement for the subsheaf $\ul{\omega} \subset H$ is given by $\mc{U} = \im(F \colon H^{(p)} \to H)$, providing
\[
    H \cong \ul{\omega} \oplus \mc{U},
\]
which is called the \emph{unit-root splitting} (of (\ref{H})). While this splitting cannot be naturally extended  to $Y$, the projection parallel to it,
\[
    p_\mr{ur} \colon H \to \ul{\omega},
\]
does extend to $Y$ upon multiplication by the Hasse invariant. This underlies Katz's construction of the theta operator. We present here a reformulation of these facts which leads naturally to the generalisations we want to discuss. Over $Y^\mr{ord}$, which can be described as the locus where $V|_{\ul{\omega}} \colon \ul{\omega} \to \ul{\omega}^{(p)}$ is an isomorphism, we may consider the composition
\[
\label{P}
\tag{P}
    H \overset{V}{\longrightarrow} \ul{\omega}^{(p)} \overset{V|_{\ul{\omega}}^{-1}}{\longrightarrow} \ul{\omega}.
\]
It is easy to see that the map $H \to \ul{\omega}$ from (\ref{P}) is precisely $p_\mr{ur}$. In particular, the morphism $h \cdot p_\mr{ur} \colon H \to \ul{\omega}^p$ can be written as
\[
\label{hp}
\tag{hP}
    H \overset{V}{\longrightarrow} \ul{\omega}^{p} \overset{V|_{\ul{\omega}}^{-1}}{\longrightarrow} \ul{\omega} \overset{h = V|_{\ul{\omega}}}{\longrightarrow} \ul{\omega}^p,
\]
which is simply the surjection $V \colon H \to \ul{\omega}^p$ and readily extends from $Y^\mr{ord}$ to the entire modular curve. More generally, one can look at 
\[
    \mr{Sym}^k(p_\mr{ur}) \colon \mr{Sym}^k(H) \longrightarrow \ul{\omega}^k,
\]
where $\mr{Sym}^k$ is the $k$-th symmetric power. The locally free sheaf $\mr{Sym}^k(H)$ admits a descending filtration
\[
    F^i(\mr{Sym}^k(H)) = \im(\ul{\omega}^i \otimes \mr{Sym}^{k-i}(H) \longrightarrow \mr{Sym}^k(H))
\]
and one can show, in the same vein, that the morphism
\[
    h \cdot \mr{Sym}^k(p_\mr{ur})|_{F^{k-1}(\mr{Sym}^k H)} \colon F^{k-1}(\mr{Sym}^k H) \to \ul{\omega}^{k+p-1}
\]
extends from $Y^\mr{ord}$ to $Y$. This is important, because on $\mr{Sym}^k(H)$ and over $Y$, one can define the {Gauss-Manin connection}
\[
    \nabla \colon \mr{Sym}^k(H) \longrightarrow \mr{Sym}^k(H) \otimes \Omega^1_{Y/\mb{F}}
\]
which satisfies a general transversality property implying, in particular, that
\[
    \nabla(\ul{\omega}^k) \subseteq F^{k-1}(\mr{Sym}^k(H)) \otimes \Omega^1_{Y/\mb{F}} \subseteq \mr{Sym}^k(H) \otimes \Omega^1_{Y/\mb{F}}.
\]
All of this is used by Katz to define the theta operator as the composition
\[
    \begin{tikzcd}
        &  \theta \colon \ul{\omega}^k \arrow[r, "{{\nabla}}"] & F^{k-1}(\mr{Sym}^k(H)) \otimes \Omega^1_{Y/\mb{F}} \arrow[rrr, "{{(h \cdot \mr{Sym}^k(p_\mr{ur})) \otimes \ul{\mr{ks}}^{-1}}}"] &&&\ul{\omega}^{k+p+1},
    \end{tikzcd}
\]
where $\ul{\mr{ks}} \colon \ul{\omega}^2 \to \Omega^1_{Y/\mb{F}} $ is the {Kodaira-Spencer} isomorphism. A more detailed inspection of $\theta$, for instance via $q$-expansions and the $q$-expansion principle, shows that it extends to an operator over $X$. Taking global sections gives rise to
\[
    \theta \colon \mr{M}_k(N) \longrightarrow \mr{S}_{k+p+1}(N) \subseteq \mr{M}_{k+p+1}(N),
\]
the \emph{theta operator on modular forms}. This is a derivation of the algebra $\mr{M}(N)$ of modular forms of degree $p+1$. By this we mean that for two modular forms $f, g$, we have $\theta(fg) = f\theta(g) + \theta(f)g$, which is a cusp form of degree $p+1+\deg(fg)$. The operator $\theta$ is $h$-linear, in the sense that $\theta(h) = 0$. Moreover, one can show that
\[
    T_l \, \theta = l\, \theta \, T_l.
\]
\par In particular, if $f$ is an eigenform, so is $\theta(f)$.

\subsection{The Picard case}

We briefly sketch the construction of ordinary and generalised theta operators in the special case of \emph{Picard modular surfaces} to illustrate the main ideas involved. 
\par First, let us set up some notation. Write $S$ for the geometric special fibre of the Picard modular surface over $\mb{F}$ with some neat, $p$-hyperspecial level. Roughly speaking, this is a moduli space of polarised abelian schemes of relative dimension $3$ endowed with an action of $\mathcal{O}_E$, the ring of integers of $E$. For more details on the moduli problem see the next section. We have a universal object $\pi \colon A \to S$ and, as before, we can consider the Hodge sheaf $\ul{\omega} = \pi_\ast(\Omega^1_{A/S})$. In this case, $\ul{\omega}$ is locally free of rank $3$ and, under the induced action of $\mc{O}_E$, it decomposes into the direct sum of two locally free sheaves $\mathcal{P} = \ul{\omega}_\sigma$ and $\mathcal{L} = \ul{\omega}_{\overline{\sigma}}$, the subscripts indicating that $\mc{O}_E$ acts on $\ul{\omega}_\tau$ via $\tau \in \{\sigma, \overline{\sigma}\}$. We assume that $\mc{P}$ and $\mc{L}$ have ranks $2$ and $1$, respectively. We will also consider the sheaf $\delta = \delta_\sigma \cong \det \mc{P} \otimes \mc{L}^{-1}$, which is torsion on $S$, by Lemma \ref{lemtorsiondelta}. We write $S^\mu \subseteq S$ for the \emph{ordinary locus}, the maximal stratum in the Ekedahl-Oort stratification of $S$, which is a dense open in $S$. The complement $S^\mr{no} = S \setminus S^\mu$ is called the \emph{non-ordinary locus} and is given by the disjoint union of  the \emph{almost ordinary locus} $S^\mr{ao}$, the EO stratum of dimension 1, open and dense in $S^\mr{no}$, and the \emph{core locus} $S^\mr{core}$, the EO stratum of dimension $0$, closed in $S$ and $S^\mr{no}$. We define an automorphic weight to be a couple $(\ul{k}, w)$, where $\ul{k}=(k_1, k_2) \in \Z^2$, with $k_1 \geq k_2$, $w \in \Z$, and the corresponding automorphic sheaf to be
\[
    \ul{\omega}^{\ul{k}, w} \coloneqq \mr{Sym}^{k_1-k_2}(\mc{P}) \otimes \det(\mc{P})^{k_2} \otimes \delta^w.
\]
\par Analogously to (\ref{H}), we have an Hodge filtration on $H = H^1_\mr{dR}(A/S)$, of which we can take components according to the action of $E$:
\begin{align*}
\label{HP}
\tag{HP}
    0 \longrightarrow \mc{P} \longrightarrow &H_\sigma \longrightarrow \mc{L}^\vee \longrightarrow 0,\\
\label{HL}
\tag{HL}
    0 \longrightarrow \mc{L} \longrightarrow &H_{\overline{\sigma}} \longrightarrow \mc{P}^\vee \longrightarrow 0.
\end{align*}
Katz's explicit construction of the unit-root splitting in the elliptic case carries over and provides a natural splitting of (\ref{HP}) and (\ref{HL}) over $S^\mu$. As we alluded to in \ref{introellcase}, we may reinterpret this unit-root splitting in terms of the Verschiebung morphism. Let $V \colon H \to H^{(p)}$ be the pullback of the Verschiebung morphism of the universal abelian scheme $A \to S$. We consider its CM components $V_\sigma \colon H_\sigma \to H_\sigma^{(p)}, V_{\overline{\sigma}} \colon H_{\overline{\sigma}} \to H_{\overline{\sigma}}^{(p)}$, whose images are $\mc{P}^{(p)}, \mc{L}^p$, respectively. Over $S^\mu$, the restrictions $V_\sigma|_\mc{P} \colon \mc{P} \to \mc{P}^{(p)}$, $V_{\overline{\sigma}}|_\mc{L} \colon \mc{L} \to \mc{L}^p$ are isomorphisms and we can use them to define
\[
    \begin{tikzcd}
     & p_{\mr{ur}, \sigma}\colon {H}_\sigma \arrow[r, two heads, "{V_\sigma}"] &\mc{P}^{(p)} \arrow[r, "{V_\sigma|_{\mc{P}}^{-1}}"] &\mc{P},\\
     & p_{\mr{ur}, {\overline{\sigma}}}\colon {H}_{\overline \sigma} \arrow[r, two heads, "{V_{\overline{\sigma}}}"] &\mc{L}^{p} \arrow[r, "{V_{\overline{\sigma}}|_{\mc{L}}^{-1}}"] &\mc{L}.
    \end{tikzcd}
\]
These morphisms give the unit-root splitting in this case. Like in the classical case, we cannot extend this splitting naturally to the non-ordinary locus $S^\mr{no} = S \setminus S^\mu$, essentially because $V \colon \ul{\omega} \to \ul{\omega}^{(p)}$ is not invertible on $S^\mr{no}$, but we can use the Hasse invariant to deal with this obstruction. In this case, the Hasse invariant is the section $h \in H^0(S, \det \ul{\omega}^{p-1})$, which is obtained as the determinant of the morphism $V\colon \ul{\omega} \to \ul{\omega}^{(p)}$. The section $h$ is nowhere-vanishing on $S^\mu$, identically $0$ on $S^\mr{no}$ and it splits as the product $h = h_\sigma \cdot h_{\overline{\sigma}}$, where 
\begin{align*}
    h_\sigma &\in H^0(S, \det \mc{P}^{p-1}), \\
    h_{\overline{\sigma}} &\in H^0(S, \mc{L}^{p-1}),
\end{align*}
are obtained from $\det V_\sigma$ and $V_{\overline{\sigma}}$, respectively. Both $h_\sigma$ and $h_{\overline{\sigma}}$ vanish with simple zeros at $S^\mr{no}$. Just like in the elliptic case, the morphism $h_{\overline{\sigma}} \cdot p_{\mr{ur}, \overline{\sigma}} \colon H_{\overline{\sigma}} \to \mc{L}^p$ can be extended from $S^\mu$ to $S$, since it simply coincides with $V\colon H_{\overline{\sigma}} \to \mc{L}^p$. We can also extend
\[
    h_\sigma \cdot p_{\mr{ur}, \sigma} \colon H_\sigma \to \mc{P} \otimes \det \mc{P}^{p-1}
\]
from $S^\mu$ to $S$, even if the rank of $\mc{P}$ is greater than $1$. In fact, the product 
\[
    h_\sigma \cdot V^{-1} \colon \mc{P}^{(p)} \to \mc{P} \otimes \det \mc{P}^{p-1}
\]
is the \emph{adjugate}, $V^\mr{adj}$, of the morphism $V \colon \mc{P} \to \mc{P}^{(p)}$ which, like $V$, is defined on the whole of $S$. As a result, the extension of $h_\sigma \cdot p_{\mr{ur}, \sigma}$ is the composition
\[
    V|_{\mc{P}}^\mr{adj} \circ V \colon H_\sigma \to \mc{P} \otimes \det \mc{P}^{p-1}.
\]
The consideration of $V^\mr{adj}$ is a core idea of \cite{EFGMM}. In Lemma \ref{lemadjext} we prove a general result for extending morphisms from $S^\mu$ to $S$, which allows us to consider the correct analogue of the above construction to define the theta operator on $\ul{\omega}^{\ul{k}, w}$, for a general weight $(\ul{k}, w)$. Lemma \ref{lemadjext} is also used crucially in our definition of generalised theta operators. In particular, in the Picard case, we can define, for $(\ul{k}, w)$ such that $k_2 \geq 0$, a morphism $h_\sigma \cdot (p_{\mr{ur}})^{\ul{k}, w} \colon H^{\ul{k}, w} \to \ul{\omega}^{\ul{k}, w} \otimes \det \mc{P}^{p-1}$, where
\begin{align*}
     h_\sigma \cdot (p_{\mr{ur}})^{\ul{k}, w} &= h_\sigma \cdot \mr{Sym}^{k_1-k_2}(p_{\mr{ur}, \sigma}) \otimes \mr{Sym}^{k_2}(\wedge^2(p_{\mr{ur}, \sigma})) \otimes \mr{id}_\delta, \\
     H^{\ul{k}, w} &= \mr{Sym}^{k_1-k_2}(H_\sigma) \otimes \mr{Sym}^{k_2}(\wedge^2(H_\sigma)) \otimes \delta.
\end{align*}
While $h_\sigma \cdot (p_{\mr{ur}})^{\ul{k}, w}$ itself does not extend from $S^\mu$ to $S$, a relevant restriction of this morphism does, thanks to Lemma \ref{lemadjext}. 
\par As in the elliptic case, we have natural filtrations $\mc{P} \subset H_\sigma$ and $\mc{L} \subset H_{\overline{\sigma}}$, with respect to which the Gauss-Manin connection $\nabla \colon H \to H \otimes \Omega^1_{S/\mb{F}}$ satisfies a natural transversality property. Unlike in the classical case, the Kodaira-Spencer morphism $\ul{\mr{KS}} \colon \ul{\omega} \otimes \ul{\omega} \to \Omega^1$ is not an isomorphism, but its restriction
\[
    \ul{\mr{ks}} = \ul{\mr{KS}}_\sigma \colon \mc{P} \otimes \mc{L} \to \Omega^1_{S/\mb{F}}
\]
is. With our conventions for the automorphic sheaves this becomes the isomorphism $\ul{\mr{ks}} \colon \mc{P} \otimes \det(\mc{P}) \otimes \delta^{-1} \to \Omega^1$. We can use all of this to define the operator
\[
    \theta_1 \colon \ul{\omega}^{\ul{k}, w} \longrightarrow \ul{\omega}^{\ul{k}+(p+1, p), w-1}
\]
as the following composition
\begin{align*}
    \ul{\omega}^{\ul{k}, w} \overset{\nabla}{\longrightarrow}&H^{\ul{k}, w} \otimes \Omega^1_{S/\mb{F}}\\
    \overset{\ul{\mr{ks}}^{-1}}{\longrightarrow} &H^{\ul{k}, w} \otimes \mc{P} \otimes \det(\mc{P}) \otimes \delta^{-1}\\
    {\longrightarrow} &\ul{\omega} ^{\ul{k}, w} \otimes \mc{P} \otimes \det(\mc{P})^p \otimes \delta^{-1}\\
    \longrightarrow &\ul{\omega}^{\ul{k}+(p+1, p), w-1},
\end{align*}
where the second-to-last map is $h_\sigma \cdot (p_{\mr{ur}})^{\ul{k}, w} \otimes \mr{id}$.
This produces a weight shift of the form $((p+1, p), -1)$, that is, mostly in the direction of $\det \mc{P} \cong \mc{L}$. To obtain a different weight shift we have to generalise the projection $h_\sigma \cdot (p_{\mr{ur}})^{\ul{k}, w}$. This does not seem to be possible on $S$. To remedy this, in this work we use the structure theory of the Ekedahl-Oort stratification to obtain such generalisations and construct theta operators on lower strata. Let us discuss some of the ingredients involved in this new construction. 
\par On $S^\mr{no}$, we have a short exact sequence
\[
\label{F}
\tag{F}
    0 \longrightarrow \mathcal{P}_0 \longrightarrow \mathcal{P} \longrightarrow \mathcal{P}_\mu \longrightarrow 0
\]
where $\mc{P}_0 \coloneqq \ker(V_\sigma \colon \mc{P} \to \mc{P}^{(p)})$ is an invertible sheaf and so is the quotient $\mc{P}_\mu$. Setting $H_\mu = H_\sigma/\mc{P}_0$, we have a ses of sheaves on $S^\mr{no}$
\[
\tag{HP'}
\label{hmu}
    0 \longrightarrow \mc{P}_\mu \longrightarrow H_\mu \longrightarrow \mc{L}^\vee \longrightarrow 0,
\]
which is analogous to (\ref{HP}). By general properties of the GM connection, we have that $\nabla(\mc{P}_0) \subseteq \mc{P}_0 \otimes \Omega^1_{S^\mr{no}/\mb{F}}$, which implies that $\nabla$ induces a connection $\nabla \colon H_\mu \to H_\mu \otimes \Omega^1$. The Verschiebung morphism induces a map $V \colon H_\mu \to H_\mu^{(p)}$, the image of which is $\mc{P}_\mu^p$. The almost-ordinary locus $S^\mr{ao}$ can be characterised as the locus in $S^\mr{no}$ where the restriction $V|_{\mc{P}_\mu} \colon \mc{P}_\mu \to \mc{P}_\mu^p$ is an isomorphism. Hence, we can use an idea similar to our reinterpretation of the unit-root splitting to construct a splitting of (\ref{hmu}), by defining
\[
    \begin{tikzcd}
    p_{\mr{ur}, 2} \colon H_\mu \arrow[r, two heads, "{V}"] &\mc{P}_\mu^p \arrow[r, "{V|_{\mc{P}_\mu}^{-1}}"] &\mc{P}_\mu.
    \end{tikzcd}
\]
We can also consider a \emph{partial generalised Hasse invariant} $A_{2} \in H^0(S^\mr{no}, \mc{P}_\mu^{p-1})$, corresponding to $V \colon \mc{P}_\mu \to \mc{P}_\mu^p$. The vanishing locus of $A_{2}$ is precisely $S^\mr{no} \setminus S^\mr{ao}$, the core locus of $S$. We can show that while the morphism $p_{\mr{ur}, 2}$ does not extend from $S^\mr{ao}$ to $S^\mr{no}$, the map $A_{2} \cdot p_{\mr{ur}, 2} \colon H_\mu \to \mc{P}_\mu^p$ does. More generally, the morphisms
\[
    A_{2} \cdot \mr{Sym}^k(p_{\mr{ur}, 2}) \colon \mr{Sym}^k(H_\mu) \longrightarrow \mc{P}_\mu^{k+p-1},
\]
for $k \geq 0$, extend to $S^\mr{no}$ when restricted to a relevant subsheaf of $\mr{Sym}^k(H_\mu)$.
With this partial unit-root splitting, we can define a new differential operator on the graded sheaves
\[
    \mr{gr}^\bullet(\ul{\omega}^{\ul{k}, w}) = \mr{Sym}^{k_1-k_2}(\mc{P}_0 \oplus \mc{P}_\mu) \otimes (\mc{P}_0 \otimes \mc{P}_\mu)^{k_2} \otimes \delta^w
\]
which correspond to a filtration on $\ul{\omega}^{\ul{k}, w}$ induced by (\ref{F}).
This \emph{generalised theta operator} will have  the form
\[
    \theta_2 \colon \mr{gr}^\bullet(\ul{\omega}^{\ul{k}, w}) \longrightarrow \mr{gr}^\bullet(\ul{\omega}^{\ul{k}+(p+1,1), w-1}),
\]
thus producing the weight shift $(p+1, 1)$, mostly in the direction of $\mc{P}$, we were looking for.
\par This construction works more generally on unitary Shimura varieties of signature $(n-1, 1)$, $n \geq 3$, where it gives rise to a family of generalised theta operators defined on various EO strata of the Shimura variety. In future work, we plan to extend these results to even more general Shimura varieties.
The main result is Theorem \ref{bigthm}, whose statement we recall here. For notations, see the relevant sections.
\begin{theorem}
\label{bigthmintro}
Let $1 \leq r < n$ be an integer and $(\ul{k}, w)$ an automorphic weight with $k_{n-1}\geq 0$. There exists a differential operator
\[
    \theta_r \colon \mr{gr}^{\bullet, r}(\ul{\omega}^{\ul{k}, w}) \longrightarrow \mr{gr}^{\bullet, r}(\ul{\omega}^{\ul{k}+\ul{\Delta}_r, w}),
\]
defined on the (closure of the) Ekedahl-Oort stratum $\overline{S}_{K, w_r}$, with
\[
    \ul{\Delta}_r = (p+1, p, \cdots, p, 1, \cdots, 1)
\]
where exactly the last $r-1$ entries are $1$. The operator $\theta_r$ satisfies the following properties:
\begin{enumerate}
    \item The operator $\theta_r$ is $A_{r}$-linear, that is, $\theta_r(A_{r}) = 0$, where $A_{r}$ is the partial Hasse invariant defined in \ref{parHinv}.
    \item The operator $\theta_r$ is Hecke-equivariant.
    \item Let $f \in H^0(\overline{S}_{K, w_r}, \mathrm{gr}^{\bullet, r}(\ul \omega^{\ul k, w}))$ and write it as $f = \sum_{\ul a} f_{\ul a}$, for the decomposition described in \ref{genthtssc}. Then $\theta_r(f)$ is divisible by the Hasse invariant $A_r$ if and only if for each component $f_{\ul a}$ either $A_r \mid f_{\ul a}$ or $p \mid \lvert \ul a \rvert$. 
\end{enumerate}
\end{theorem}

\subsection{Acknowledgements}
I would especially like to thank my Ph.D. supervisor, Prof.\ P.\ L.\ Kassaei, for suggesting that I work on these and related problems and for always being supportive. I am likewise grateful to my second supervisor, Prof.\ F.\ I.\ Diamond, for many helpful conversations. For their insightful comments, I also wish to thank George Boxer, Andrew Graham, Pol van Hoften and Martin Ortiz.
\par This work was supported by the Engineering and Physical Sciences Research Council [EP/L015234/1], through the EPSRC Centre for Doctoral Training in Geometry and Number Theory (The London School of Geometry and Number Theory), University College London and King's College London.

\section{Unitary Shimura varieties of signature \texorpdfstring{$(n-1,1)$}{(n-1, 1)}}

\subsection{The PEL datum}
We will work with the {integral PEL datum}, in the sense of \cite[Ch.~5]{kott92}, defined as follows:
\begin{enumerate}
    \item The simple $\Q$-algebra $B$ is just $E$, with $\Oe$ as its maximal $\Z$-order.
    \item The positive involution ${}^\ast$ on $E$ is the only non trivial element of $\mr{Gal}(E/\Q)$, which is induced by the complex conjugation $\overline{\cdot}$ via $\sigma$. The fixed field $F_0$ is $\Q$.
    \item We let $V$ be the $E$-vector space $E^n$ and take $\Lambda = \mathcal{O}_E^n$ to be the natural $\Oe$-lattice inside it, generated by the canonical basis $\{e_1, e_2, \dots, e_n\} \subset \Lambda$.
    \item Consider the perfect Hermitian pairing $\left(\cdot, \cdot\right) \colon V \times V \to E$ of signature $(n-1, 1)$ given by the diagonal matrix $I_{n-1,1}= \mr{diag}(1, \dots, 1, -1)$, with respect to the canonical basis on $V$. Notice that we can restrict $\left(\cdot, \cdot\right)$ to a perfect pairing $\left(\cdot, \cdot\right) \colon \Lambda \times \Lambda \to \Oe$. Define 
    \[ 
    \left< \cdot, \cdot\right> \coloneqq \mr{T}_{E/\Q}(\delta_{E/\Q}^{-1}(\cdot, \cdot)) = {2i {\delta}_{E/\Q}^{-1} \mr{Im}(\cdot, \cdot)}\colon V \times V \to \Q.
    \]
    This is a perfect alternating $\Q$-linear pairing such that $\left<\alpha u, v\right> =\left<u, \overline{\alpha}v\right>,\, u, v \in V, \alpha \in \Oe$, whose restriction to $\Lambda$ induces a perfect $\Z$-linear pairing $\Lambda \times \Lambda \to \Z$. The Hermitian pairing $\left(\cdot, \cdot\right)$ can be recovered from $\left< \cdot, \cdot\right>$ via the identity $2\left(u, v\right) = \left<u, \delta_{E/\Q} v \right> + \delta_{E/\Q} \left<u, v\right>$. By adjunction, $\left( \cdot, \cdot \right)$ defines an involution ${}^\ast$ of $\edom_E(V)$, which restricts to an involution of $\edom_\Oe(\Lambda)$, compatible with the conjugation on $E \subset \edom_E(V)$. With respect to the canonical basis, this involution is given by $M \mapsto I_{n-1,1}{}^t\overline{M}I_{n-1,1}$.
    \item We have the isomorphisms 
    \[
        \edom_E(V) \otimes_\Q \R \cong \edom_{E\otimes_\Q \R}(V\otimes_\Q \R) \cong \mr{M}_n(\C),
    \]
    the first being natural, the second depending on the choice of the canonical basis and of $\sigma$.
    We take $h \colon \C \to \edom_E(V) \otimes_\Q \R$, via these isomorphisms, to be the map of $\R$-algebras defined by $z \mapsto \mr{diag}(z, \dots, z, \overline{z})$. Clearly, $h(\overline{z}) = \overline{h(z)}$. Notice that $\left< \cdot, h(i) \cdot\right>$ is symmetric and positive definite on $V_\R$. In particular, through the isomorphism $\Z(1) \cong \Z$ given by our choice of $\sigma$ (and $i \in \C$), we see that $(\Lambda, h, \left<\cdot, \cdot \right>)$ is a \emph{integral polarised Hodge structure}. 
    As usual, $h$ can be seen as the action on $\R$-points of a morphism between schemes with values in real algebras, which we will also denote by $h$, defined on an $\R$-algebra $R$ as 
    \begin{align*}
        h(R) \colon \mr{Res}_{\C/\R} \mb{G}_a(R) = R \otimes_\R \C & \longrightarrow \edom_E(V) \otimes_\Q R \cong \mr{M}_n(R \otimes_\R \C),\\
        z & \longmapsto \mr{diag}(z, \dots, z, \overline{z}),
    \end{align*}
    where $\overline{z}$ is defined by the action on the first factor of the tensor product.
\end{enumerate}
These data satisfy some important properties which we now recall.
\par First of all, we can decompose $V_\C = V \otimes_\Q \C$ according to the action of $z \in \C$ via $h$, that is, $V_\C = V_1 \oplus V_2$, where 
\[
    V_1 = \{ v \in V_\C \mid h(z)v=z v\},\quad V_2 = \{ v \in V_\C \mid h(z)v=\overline{z} v\}.
\]
We can describe this {Hodge structure} of type $\{(-1,0), (0, -1)\}$ explicitly. Consider the two primitive idempotents
\[
    e_\sigma = \frac{1 \otimes 1 + \delta_{E/\Q} \otimes \delta_{E/\Q}^{-1}}{2}, \quad e_{\overline{\sigma}} = \frac{1 \otimes 1 - \delta_{E/\Q} \otimes \delta_{E/\Q}^{-1}}{2}
\]
in the ring $E \otimes_\Q E \subset E \otimes_\Q \C$. Then
\begin{align}
\label{v1}
    V_1 &= \mr{Span}_\C\left < e_\sigma\cdot e_1, \dots, e_\sigma\cdot e_{n-1}, e_{\overline{\sigma}}\cdot e_n \right >, \\
\label{v2}
    V_2 &= \mr{Span}_\C\left < e_{\overline{\sigma}}\cdot e_1, \dots, e_{\overline{\sigma}}\cdot e_{n-1}, e_\sigma \cdot e_n \right >,
\end{align}
where $e_1, e_2, \dots, e_n$ is the canonical basis of $V$ and $V_\C$, by abuse of notation.
One can see that both $V_1$ and $V_2$ are {defined over} $E$, in the sense of \cite[Def.~1.1.2.7]{kwl}, which means that the subfield of $\C$ fixed by all the $\tau \in \mr{Aut}(\C)$ such that $V_1^\tau \coloneqq V_1 \otimes_{\C, \tau} \C \cong V_1$ as $E \otimes_\Q \C$-modules is $E$.
\par Proving that such field of definition is $E$ comes down to the fact that $\tau \in \mr{Aut}(\C)$ has two possible restrictions to $E$: when $\tau|_E = \mr{id}_E$, we have $V_1^\tau \cong V_1$, when $\tau|_E \neq \mr{id}_E$, then $V_1^\tau \cong V_2$, since $\tau(\delta_{E/\Q}^{-1}) = -\delta_{E/\Q}^{-1}$, which is not isomorphic to $V_1$ as $E \otimes_\Q \C$-modules (because they have different signatures, see below). Furthermore, we can find models of $V_1$ and $V_2$ as $E \otimes_\Q E \subset E \otimes_\Q \C$-modules: if we call these, by abuse of notation, again $V_1, V_2$, we see that there is a decomposition $V \otimes_\Q E \cong V_1 \oplus V_2$, given as above, by the same idempotents. In fact, this decomposition is even integral: $e_\sigma, e_{\overline{\sigma}}$ are idempotents in the ring $\Oe \otimes_\Z \mc{O}_{E, \mr{ur}}$ and we can decompose $\Lambda \otimes_{\Z} \mc{O}_{E, \mr{ur}} \cong \Lambda_1 \oplus \Lambda_2$, with $\Lambda_i \otimes_{\mc{O}_{E, \mr{ur}}} E \cong V_i, \,i=1,2$. Summing up, by the definition given in \cite[Ch.~5]{kott92}, we have proved the following fact.
\begin{lemma}
The \emph{reflex field} of the datum $(E, {}^\ast, V, \left<\cdot,\cdot\right>, h)$ is $E$.
\end{lemma}
\label{signature}
\par We can look at the action of any $\alpha \in \Oe$ on $V_1$, or equivalently $\Lambda_1$. In particular, we can define the characteristic polynomial
\[
    p_\alpha(X) \coloneqq \det(X-\alpha|_{V_1}) = (X - \sigma(\alpha))^{n-1} (X - \overline{\sigma}(\alpha)) \in \Oe[X].
\]
This can be seen using the basis of $V_1$ described above, together with the fact that 
\[
(\alpha \otimes 1) e_\sigma = 1 \otimes \sigma(\alpha), \quad (\alpha \otimes 1) e_{\overline{\sigma}} = 1 \otimes \overline{\sigma}(\alpha).
\]
These polynomials are necessary to specify the endomorphism part of the moduli problem, in particular, to express Kottwitz's {determinant condition}, see \cite[Ch.~5]{kott92} and \cite[1.3.4]{kwl}.
\begin{remark}[CM Decomposition]
\label{cmdecomp}
One can consider the decomposition of $V \otimes_\Q E$ and $\Lambda \otimes_\Z {\mc{O}_{E, \mr{ur}}}$ induced by $e_\sigma, e_{\overline{\sigma}}$ via
\begin{align*}
    V \otimes_\Q E &\cong V_\sigma \oplus V_{\overline{\sigma}} \coloneqq e_\sigma (V \otimes_\Q E) \oplus e_{\overline{\sigma}} (V \otimes_\Q E),\\
    \Lambda \otimes_\Z {\mc{O}_{E, \mr{ur}}} &\cong \Lambda_\sigma \oplus \Lambda_{\overline{\sigma}} \coloneqq e_\sigma(\Lambda \otimes_\Z {\mc{O}_{E, \mr{ur}}}) \oplus e_{\overline{\sigma}}(\Lambda \otimes_\Z {\mc{O}_{E, \mr{ur}}}).
\end{align*}
We will call this the \emph{CM decomposition} of $V$ and $\Lambda$. One can consider this decomposition more generally for a module over the ring $\Oe \otimes_\Z \mc{O}_{E, \mr{ur}}$ or sheaves of modules over schemes over $\spec \mc{O}_{E, \mr{ur}}$ with a linear $\Oe$-action, as we will see below.
\end{remark}
\subsection{The reductive group}
From the integral datum $(\Oe, \overline{\cdot}, \Lambda, \left<\cdot, \cdot\right>, h)$, we obtain a group scheme $\mathbf{G}$ over $\Z$ whose $R$-points for $R \in \ul{\mr{Al}}\mr{g}_{\Z}$ are
\begin{align*}
    \mathbf{G}(R) &= \mathbf{GU}(\Lambda, \left(\cdot, \cdot\right))(R) \coloneqq \{g \in \edom_{\mc{O}_{E} \otimes R}(\Lambda \otimes_{\Z} R) \mid gg^\ast = \nu(g) \in R^\times \} \\
                &= \{(g, \nu(g)) \in \edom_{\mc{O}_{E} \otimes R}(\Lambda_R) \times R^\times \mid \left(gu, gv\right) = \nu(g) \left(u, v\right),\, \forall u, v \in \Lambda_R\}.
\end{align*}
The morphism of group schemes $\nu \colon \mathbf{G} \to \mb{G}_m$ is called \emph{similitude factor}. The kernel of $\nu$ is denoted $\mathbf{G}_1$. We have the following.
\begin{lemma}
\label{Levidecomp}
Let $s \colon \spec k \to \spec {\Z[1/2D]}$ be a morphism with $k$ an algebraically closed field. We have natural isomorphisms of group schemes:
\[
    \mathbf{G}_{\mc{O}_{E, \mr{ur}}} \cong \mr{GL}_{n} \times_{\mc{O}_{E, \mr{ur}}} \mathbb{G}_{m}, \, \mathbf{G}_{1, \mc{O}_{E, \mr{ur}}} \cong \mr{GL}_{n}
\]
and similarly for $\mathbf{G}_s$ and $\mathbf{G}_{1,s}$. In particular, $\mathbf{G}$ and $\mathbf{G}_1$ are reductive over $\spec {\Z[1/2D]}$.
\par Moreover, under the isomorphism $\mathbf{G}_{\mc{O}_{E, \mr{ur}}} \cong \mr{GL}_{n} \times \mathbb{G}_{m}$, the Levi subgroup $H$ of the parabolic fixing the decomposition $\Lambda \otimes_\Z \mc{O}_{E, \mr{ur}} \cong \Lambda_1 \oplus \Lambda_2$ corresponds to $\mr{GL}_{n-1} \times \mr{GL}_1 \times \mathbb{G}_{m}$. The similarly defined Levi $H_1$ of $\mathbf{G}_1$ is isomorphic after base change to $\mr{GL}_{n-1} \times \mr{GL}_1$.
\end{lemma}
\begin{proof}
We have a natural isomorphism $\Oe \otimes_\Z \mc{O}_{E, \mr{ur}} \to \mc{O}_{E, \mr{ur}} \oplus \mc{O}_{E, \mr{ur}}, \alpha \otimes \beta \mapsto (\alpha\beta, \overline{\alpha}{\beta})$, corresponding to the idempotents $e_\sigma, e_{\overline{\sigma}}$, which induces the CM decomposition described in \ref{cmdecomp} and \ref{signature}. Notice that via this isomorphism the conjugation $\overline{\cdot}$ on the left corresponds to the involution $(r,s) \mapsto (s,r)$ on the right. In particular, functorially in $A \in \ul{\mr{Al}}\mr{g}_{\mc{O}_{E, \mr{ur}}}$, we have:
\begin{align*}
    \mathbf{G}(A) &\cong \{(g, \nu(g)) \in \mr{GL}(\Lambda_{\sigma, A} \oplus \Lambda_{\overline{\sigma}, A}) \times A^\times \mid \left(gu, gv\right) = \nu(g) \left(u, v\right),\, \forall u, v\}\\
    &\cong \{(g_\sigma, g_{\overline{\sigma}}, \nu) \in \mr{GL}_n(A) \times \mr{GL}_n(A) \times A^\times \mid {}^t g_{\overline{\sigma}} I_{n-1,1} g_\sigma = \nu I_{n-1,1}\}\\
    &\cong \mr{GL}_n(A) \times A^\times.
\end{align*}
The proof for $\mathbf{G}_1$ is analogous and the case of fibres over $s$ follows by base change. The statement about the Levi subgroups follows from \ref{v1}, \ref{v2} and \ref{signature}.
\end{proof}

The real points of $\mathbf{G}, \mathbf{G}_1$ give the classical unitary groups
\[
    \mathbf{G}(\R) = \mr{GU}(n-1, 1), \quad \mathbf{G}_1(\R) = \mr{U}(n-1,1)
\]
and one can see that $H(\R) = G(U(n-1) \times U(1)), H_1(\R) = U(n-1) \times U(1)$.
\par We can restrict $h$ to a morphism of algebraic groups $h \colon \mb{S} \to \mathbf{G}_\R$, denoted again by $h$. We see that $(\mathbf{G}_\Q, h)$ is a {Shimura datum}, as defined in \cite[1.1-5]{deltds} and \cite[2.1.1]{delvds}.

\subsection{The moduli problem}
We are almost ready to formulate the PEL moduli problem we are interested in. We will be working with a neat, $p$-hyperspecial level $K \subseteq \mathbf{G}(\mb{A}^\infty)$, see \cite[1.4.1.8-13]{kwl}.
\par Let $S$ be a scheme defined over $\mathcal{O}_{E, (p)} \coloneqq \Oe \otimes_\Z \Z_{(p)}$. We will denote the category of such schemes by $\ul{\mr{Sch}}_{\mathcal{O}_{E, (p)}}$. To $S$ we can associate a quadruple $\ul{A} = (A, \lambda, \iota, \eta_K)$ where:
\begin{enumerate}
    \item $A \to S$ is an abelian scheme.
    \item $\lambda \colon A \to A^\vee$ is a prime-to-$p$ polarisation.
    \item $\iota \colon \Oe \to \edom_S(A)$ is a ring homomorphism such that:
    \begin{enumerate}
        \item The {Rosati involution} is compatible with conjugation on $\Oe$, which amounts to the relation
        \[
            \lambda \iota(\overline{\alpha}) = \iota(\alpha)^\vee \lambda, \, \alpha \in \Oe.
        \]
        We will sometimes call this the \emph{Rosati relation}.
        \item On the locally free, finitely generated $\mathcal{O}_S$-module $\mr{Lie}(A/S)$ we have an induced action of $\Oe$. We can consider as above the characteristic polynomial $\det(X-\alpha|_{\mr{Lie}(A/S)})$, for any $\alpha \in \Oe$. We require that
        \[
            \det(X-\iota(\alpha)|_{\mr{Lie}(A/S)}) = p_\alpha(X) \in \Oe[X] \subset \mathcal{O}_S[X],
        \]
        for every $\alpha \in \Oe$.
    \end{enumerate}
    \item $\eta_K$ is an {$\Oe$-linear integral level $K$-structure}, described in further details below.
\end{enumerate}

An isomorphism of two such quadruples is an isomorphism of the underlying abelian schemes compatible with the remaining pieces of data in the natural way. In \cite[Ch.~2]{kwl}, it is shown that, under our assumptions on $K$, the functor
\begin{align*}
    \ul{\mr{Sch}}_{\mathcal{O}_{E, (p)}} & \longrightarrow \ul{\mr{Set}},\\
    S & \longmapsto \{( A, \lambda, \iota, \eta_K )\}/_{\cong},
\end{align*}
is represented by a smooth, quasi-projective scheme $\mathcal{S}_K \in \ul{\mr{Sch}}_{\mathcal{O}_{E, (p)}}$. One can show that $\mc{S}_K$ has relative dimension $n-1$.

\subsubsection{Level structures}
\label{lvlstr}
Assume now that $S \in \ul{\mr{Sch}}_{\mathcal{O}_{E, (p)}}$ is connected. Let $(A, \lambda, \iota)$ be a triple, notations as above, over $S$. Also let $s$ be a geometric point of $S$. The Tate $\mathbb{A}^{p, \infty}$-module $V^p A_s$ of $A_s$ is a {lisse} $\mathbb{A}^{p, \infty}$-sheaf, which is equivalent to saying that it is determined by the $\mathbb{A}^{p, \infty}$-module $V^p A_s \cong H_{1}^{\mr{\acute{e}t}}(A_s, \mathbb{A}^{p, \infty})$ together with its continuous $\pi_{1}^{\mr{\acute{e}t}}(S, s)$-action. Notice that $V^p A_s$ contains a natural $\hat{\Z}^{p}$-lattice $T^p A_s$.

\begin{definition}
A \emph{rational level $K^p$-structure} on $A$ is a $K^p$-orbit $\eta_K$ of isomorphisms
\[
    \eta \colon V_{\mathbb{A}^{p, \infty}} \longrightarrow V^p A_s
\]
of $\mathbb{A}^{p, \infty}$-modules such that:
\begin{enumerate}
    \item $\eta$ is compatible with the action of $\Oe$ on both sides.
    \item $\eta$ is isometric, for $\left<\cdot, \cdot\right>_{\mathbb{A}^{p, \infty}}$ on $V_{\mathbb{A}^{p, \infty}}$ and the {Weil pairing} $e^\lambda$ on $V^p A_s$, meaning that there is an isomorphism $\nu \colon  \mathbb{A}^{p, \infty} \to V^p \mathbb{G}_{m, s}$, uniquely determined by $(A, \lambda, \iota, \eta_K)$, such that $e^\lambda(\eta(u), \eta(v)) = \nu\left< u, v \right>$ for all $u, v \in V_{\mathbb{A}^{p, \infty}}$.
    \item The orbit $\eta_K$ is fixed by the action of $\pi_{1}(S, s)$.
\end{enumerate}
We say moreover that the structure $\eta_K$ is \emph{integral} if every choice of $\eta$ in the orbit $\eta_K$ induces an isomorphism of $\hat{\Z}^p$-modules
\[
    \eta \colon \Lambda_{\hat{\Z}^p} \longrightarrow T^p A_s
\]
and similarly $\nu$ induces an isomorphism $\hat{\Z}^p \to T^p \mb{G}_{m,s}$.
\end{definition}
Notice that $\mathbf{G}(\hat{\Z}^p)$ acts simply transitively on the set of isomorphisms $V_{\mathbb{A}^{p, \infty}} \to V^p A_s$ satisfying 1 and 2, when this set is non-empty. The notion of $K^p$-level structure is essentially independent of the choice of the point $s$, see \cite[Lemma~1.3.8.6]{kwl}. Over a non-connected base $S$ we can define a integral level structure by working separately on each connected component, since a locally noetherian topological space is the disjoint union of its connected components, which are clopens.
\begin{remark}
In \cite[1.3]{kwl} the author gives a more general definition of level structure without any noetherianness assumption. This becomes equivalent to the notion of level we have presented here in the locally noetherian case.
\end{remark}
\subsection{Hodge and determinant sheaves}
For $G \to S$ a group scheme we will use the notation
\[
    \ul{\omega}_{G/S} \coloneqq e^\ast(\Omega^1_{G/S})
\]
to denote the cotangent space at the identity $e$ of $G$. One can show that $f^\ast( \ul{\omega}_{G/S}) \cong \Omega^1_{G/S}$, where $f \colon G \to {S}$ denotes the structure morphism. Let now $G=A$ be an abelian scheme. Then, we have $f_\ast \mc{O}_{A} = \mc{O}_{{S}}$, so that
\[
    f_\ast \Omega^1_{A/S} \cong f_\ast f^\ast \ul{\omega}_{A/S} \cong f_\ast (\mc{O}_{A}) \otimes_{\mc{O}_{S}} \ul{\omega}_{A/S} = \ul{\omega}_{A/S}.
\]
We will call the sheaf $\ul{\omega}_{A/S}$ the \emph{Hodge sheaf} of $A/S$.
\par Let $\ul A = (A, \lambda, \iota, \eta_{N}) \in \mathcal{S}_K(S)$, for some $S \in \ul{\mr{Sch}}_{\mc{O}_{E,(p)}}$. If $\ul{A}$ is in the universal class, we will simply write $\ul{\omega}$ instead of $\ul{\omega}_{A/\mc{S}_K}$. Similarly, we will often write $\ul \omega$ instead of $\ul{\omega}_{A_B/S_B}$, for some base change $S_{K,B}$ of $\mc{S}_K$ given by $B \to \spec ({\mc{O}_{E,(p)}})$, depending on the context. We will care mostly about the case $B = \spec k$, with $k$ an algebraically closed field of characteristic $p$. As an $\mathcal{O}_S$-sheaf, $\ul{\omega}_{A/S}$ is locally free of rank $n$ with an action of $\Oe$ of {signature} $(n-1, 1)$: the determinant condition on $\mr{Lie}(A/{S})$ implies a dual analogous condition on $\ul{\omega}_{A/S}$ which is equivalent to the fact that
\[
    \ul{\omega}_{A/S, \sigma} \coloneqq e_\sigma \cdot \ul{\omega}_{A/S}, \quad \ul{\omega}_{A/S, \overline{\sigma}} \coloneqq e_{\overline{\sigma}} \cdot \ul{\omega}_{A/S},
\]
are locally free $\mathcal{O}_{\mathcal{S}_K}$-sheaves of ranks $n-1$ and $1$, respectively. This is the CM decomposition of $\ul{\omega}_{A/S}$, in the sense of \ref{cmdecomp}. We call $\ul{\omega}_{A/S, \sigma}$, $\ul{\omega}_{A/S, \overline{\sigma}}$ the {$\sigma$, ${\overline{\sigma}}$-components} of the Hodge bundle, respectively.
\label{secderhamcohom}
\par We recall some basic facts concerning the {de Rham cohomology} of abelian schemes, following \cite[2.5.1]{bbm}. First of all, for $f \colon X \to S$ a smooth morphism of schemes of relative dimension $g$ we can consider the {de Rham complex} of finite flat sheaves
\[
    \Omega_{X/S}^\bullet \colon 0 \longrightarrow \mc{O}_X \overset{d_{X/S}}{\longrightarrow} \Omega^1_{X/S} \longrightarrow \dots \longrightarrow \Omega^g_{X/S} \longrightarrow 0.
\]
The {$i$-th de Rham cohomology of} $X/S$ is  the $i$-th derived pushforward $H^i_\mr{dR}(X/S) \coloneqq R^i f_\ast \Omega_{X/S}^\bullet$ of this complex. On $\Omega^\bullet_{X/S}$, one can consider the usual {truncation filtration} $F^\bullet(\Omega^\bullet_{X/S})$, that is,
\begin{align*}
    F^i(\Omega^j_{X/S}) &= 0 \text{ if } j<i, \\
    F^i(\Omega^j_{X/S}) &= \Omega^j_{X/S} \text{ if } j \geq i.
\end{align*}
The corresponding filtration spectral sequence 
\[
        E^{p,q}_1 = R^q f_\ast \Omega^p_{X/S} \Longrightarrow H^{p+q}_\mr{dR}(X/S)
\]
is called the \emph{Hodge-de Rham spectral sequence}. If we assume that $f$ is quasi-compact, then from this spectral sequence we deduce that the $H^i_\mr{dR}(X/S)$ are quasi-coherent sheaves. 
\par We consider now the case where $X=A$ is an abelian scheme. 
We have the following.
\begin{proposition}[{\cite[Prop.~2.5.2]{bbm}}]
\label{derhamabelian}
Let $A/S$ be an abelian scheme. Then:
\begin{enumerate}
    \item For $i \geq 0$ the sheaves $H^i_\mr{dR}(A/S)$ are locally free and their formation commutes with base change.
    \item The natural morphism $\bigwedge^i H^1_\mr{dR}(A/S) \to H^i_\mr{dR}(A/S)$ is an isomorphism for all $i \geq 0$.
    \item The {Hodge-de Rham spectral sequence} of $A/S$ degenerates on the first page.
\end{enumerate}
\end{proposition}

\begin{remark}
Point 3 of \ref{derhamabelian} is particularly important, since the corresponding statement for $X/k$ a non-singular variety over an algebraically closed field $k$ can fail in positive characteristic, see for instance \cite[I]{mumford}.
\end{remark}
Here we are mostly interested in the case $p+q=1$ in point $3$ of Proposition \ref{derhamabelian}, which is the content of \cite[Lemma~2.5.3]{bbm}. In particular, we have the short exact sequence
\begin{equation}
\label{hfilt1}
    0 \longrightarrow \ul{\omega}_{A/S} \longrightarrow H^1_\mr{dR}(A/S) \longrightarrow R^1f_\ast \mc{O}_A \longrightarrow 0.
\end{equation}
From of \cite[5.1.1]{bbm}, we see that $R^1f_\ast \mc{O}_A$ is naturally isomorphic to
\[
    \mr{Lie}(A^\vee/S) \cong \ul{\omega}_{A^\vee/S}^\vee \coloneqq \ul{\hom}_{\mc{O}_S}(\ul{\omega}_{A^\vee/S}, \mc{O}_S),
\]
so we can rewrite the previous short exact sequence as 
\begin{equation}
\label{hfilt2}
    0 \longrightarrow \ul{\omega}_{A/S} \longrightarrow H^1_\mr{dR}(A/S) \longrightarrow \ul{\omega}_{A^\vee/S}^\vee \longrightarrow 0.
\end{equation}
From this, we see that one can make point 1 of Proposition \ref{derhamabelian} more explicit and deduce that, in fact, $H^1_\mr{dR}(A/S)$ has rank $2g$, with $g = \dim(A/S)$.
In particular, if $A$ is endowed with a prime-to-$p$ polarisation $\lambda$, we can use it to identify $\omega_{A^\vee/S}^\vee$ and $\omega_{A/S}^\vee$ to get 
\begin{equation}
\label{hfilt3}
    0 \longrightarrow \ul{\omega}_{A/S} \longrightarrow H^1_\mr{dR}(A/S) \longrightarrow \ul{\omega}_{A/S}^\vee \longrightarrow 0.
\end{equation}
We call \ref{hfilt2} the \emph{natural Hodge filtration}, as opposed to \ref{hfilt3}, which we call the \emph{polarised Hodge filtration}. We simply call the short exact sequences \ref{hfilt1}, \ref{hfilt2} and \ref{hfilt3} the \emph{Hodge filtration} when it is clear from the context to which one we are referring.
\par If we suppose now that $S \in \ul{\mr{Sch}}_{\mc{O}_{E,(p)}}$ and $\ul{A} \in \mc{S}_K(S)$, we can also consider the action of $\Oe$ and split the natural Hodge filtration according to the CM decomposition of its terms, to obtain
\begin{align}
    \label{hfiltsigma1}
    0 \longrightarrow \ul{\omega}_{A/S, \sigma} &\longrightarrow H^1_\mr{dR}(A/S)_\sigma \longrightarrow \ul{\omega}_{A^\vee/S,{\sigma}}^\vee \longrightarrow 0, \\
    \label{hfiltosigma1}
    0 \longrightarrow \ul{\omega}_{A/S, \overline{\sigma}} &\longrightarrow H^1_\mr{dR}(A/S)_{\overline{\sigma}} \longrightarrow \ul{\omega}_{A^\vee/S, \overline{\sigma}}^\vee \longrightarrow 0,
\end{align}
where the action of $\Oe$ on $\ul{\omega}_{A^\vee/S}$ is induced by the action of $\Oe$ on $A^\vee$ given by $a \mapsto \iota(a)^\vee$. Notice that because of the Rosati relation, the polarisation {conjugates} the action on $\ul{\omega}_{A/S}$, that is, we have isomorphisms
\[
    \ul{\omega}_{A^\vee/S, \sigma} \cong \ul{\omega}_{A/S, \overline{\sigma}}, \quad \ul{\omega}_{A^\vee/S, \overline{\sigma}} \cong \ul{\omega}_{A/S, {\sigma}}
\]
induced by $\lambda$.
In particular, splitting according to the $\Oe$-action the short exact sequence \ref{hfilt3} we obtain
\begin{align}
    \label{hfiltsigma2}
    0 \longrightarrow \ul{\omega}_{A/S, \sigma} &\longrightarrow H^1_\mr{dR}(A/S)_\sigma \longrightarrow \ul{\omega}_{A/S,\overline{\sigma}}^\vee \longrightarrow 0, \\
    \label{hfiltosigma2}
    0 \longrightarrow \ul{\omega}_{A/S, \overline{\sigma}} &\longrightarrow H^1_\mr{dR}(A/S)_{\overline{\sigma}} \longrightarrow \ul{\omega}_{A/S, {\sigma}}^\vee \longrightarrow 0.
\end{align}
We call \ref{hfiltsigma1} and \ref{hfiltosigma1} the $\sigma$ and {$\overline{\sigma}$-components}, respectively, of the natural Hodge filtration. Similarly, we call \ref{hfiltsigma2} and \ref{hfiltosigma2} the $\sigma$ and {$\overline{\sigma}$-components}, respectively, of the polarised Hodge filtration. From \ref{hfiltsigma2} and \ref{hfiltosigma2}, given that $\ul{\omega}_{A/S, \sigma}$ and $\ul{\omega}_{A/S, \overline{\sigma}}$ have ranks $n-1$ and $1$, respectively, we can deduce that $H^1_\mr{dR}(A/S)_\sigma$ and $H^1_\mr{dR}(A/S)_{\overline{\sigma}}$ both have rank $n$.
\par From the Hodge filtration \ref{hfilt3} we further deduce the isomorphism $\det H^1_\mr{dR}(A/S) \cong \mc{O}_S$, which depends on the choice of the polarisation $\lambda$.
\begin{definition}
We call $\delta_{A/S} \coloneqq \det H^1_\mr{dR}(A/S)$ the \emph{determinant sheaf}.
\end{definition}

Suppose again that $S \in \ul{\mr{Sch}}_{\mc{O}_{E,(p)}}$ and $\ul{A} \in \mc{S}_K(S)$, so that we can consider the invertible sheaves 
\begin{align*}
    \delta_{A/S, \sigma} &\coloneqq \det H^1_\mr{dR}(A/S)_\sigma, \\
    \delta_{A/S, \overline{\sigma}} &\coloneqq \det H^1_\mr{dR}(A/S)_{\overline{\sigma}}.
\end{align*}
Notice that $\delta_{A/S, \sigma} \otimes \delta_{A/S, \overline{\sigma}} \cong \delta_{A/S}$. In particular, the choice of a polarisation will induce an isomorphism $\delta_{A/S, \sigma}^{-1} \cong \delta_{A/S, \overline{\sigma}}$.
Taking determinants of \ref{hfiltsigma1} and \ref{hfiltosigma1} we deduce
\begin{equation}
\label{determinanteq}
    \det \ul{\omega}_{A/S, \sigma} \cong \delta_{A/S,\sigma} \otimes \ul{\omega}_{A^\vee/S, \sigma}, \quad \ul{\omega}_{A/S, \overline{\sigma}} \cong \delta_{A/S,\overline{\sigma}} \otimes \det \ul{\omega}_{A^\vee/S, \overline{\sigma}},
\end{equation}
identifications which \emph{do not involve} the polarisation $\lambda$. By making use of $\lambda$ instead, one can deduce similar isomorphisms from \ref{hfiltsigma2} and \ref{hfiltosigma2}.
We call $\delta_{A/S, \sigma}$ and $\delta_{A/S, \overline{\sigma}}$ the $\sigma$ and $\overline{\sigma}$\emph{-determinant sheaf}, respectively.
\begin{remark}
\label{noncanonicremark}
The non-canonicity of the isomorphism $\det H^1_\mr{dR}(A/S) \cong \mc{O}_S$ is relevant when one considers the action of Hecke operators. In fact, the isomorphism in question is not Hecke-equivariant and the invertible sheaf $\delta_{A/S}$ is not trivial as a Hecke module.
\end{remark}
Let $k$ be a perfect field of characteristic $p$ which is also a $\mc{O}_E$-algebra, $S \in \ul{\mr{Sch}}_k$ a reduced scheme and take $\ul{A} \in \mc{S}_K(S)$.
\begin{lemma}[\texorpdfstring{$\delta$}{The determinant sheaf} is torsion]
\label{lemtorsiondelta}
There is a natural isomorphism, independent of the level structure and the polarisation $\lambda$, $\delta_{A/S, \sigma}^{p-1} \cong \mc{O}_{S}$.
\end{lemma}
\begin{proof}
Write $H = H^1_\mr{dR}(A/S)$. Then, by Dieudonn\'e theory and in particular {\cite[Cor.~5.11]{odathesis}}, we have two natural short exact sequences, which, in particular, do not depend on the level structure,
\begin{align*}
    0 \longrightarrow \ul{\omega}_{A/S,\sigma}^{(p)} \longrightarrow &H^{(p)}_\sigma \overset{F_\sigma}{\longrightarrow} \mc{U}_\sigma \longrightarrow 0,\\
    0 \longrightarrow \mc{U}_\sigma \longrightarrow &H_\sigma \overset{V_\sigma}{\longrightarrow} \ul{\omega}_{A/S,\sigma}^{(p)} \longrightarrow 0.
\end{align*}
Take the highest exterior power of both: we obtain $\delta_{A/S, \sigma}^p \cong \det \ul{\omega}_{A/S,\sigma}^{p} \otimes \mc{U}_\sigma$ from the first sequence and $\delta_{A/S, \sigma} \cong \det \ul{\omega}_{A/S,\sigma}^{p} \otimes \mc{U}_\sigma$ from the second. From this we deduce the natural isomorphism $\delta_{A/S, \sigma}^{p-1} \cong \mc{O}_{S}$.
\end{proof}

\begin{remark}
Although we will not make use of this fact here, we remark the same argument generalises naturally to other PEL Shimura varieties and to $p$ unramified in $E$. 
\end{remark}

\subsection{The Weyl modules of \texorpdfstring{$\mr{GL}_m$}{the General Linear Group}}
\label{weylmodules}
We recall the description of a natural class of algebraic representations $\ul{\mr{Re}}\mr{p}_R(\mr{GL}_m)$ of $\mr{GL}_{m, R}$, $m \geq 2$, into locally free finite rank $R$-modules, for $R$ a Dedekind domain. Our main reference will be \cite{jantzen}. Until the end of this discussion we write $G = \mr{GL}_{m, R}$ to ease our notations.
\par Let $B \subset G$ be the Borel subgroup of {upper}-triangular matrices, take $T \subset B$ the torus of diagonal matrices and $U \subset B$ the unipotent radical of $B$. We will denote by ${B}^- \supset {U}^-$ the Borel of $G$ opposite to $B$ with respect to $T$, with its unipotent radical.
Write $X(T) \coloneqq \hom(X, \mb{G}_m)$ for the group of characters of $T$. We are interested in representations associated to \emph{dominant weights} of $G$ with respect to $B \supset T$, which form the submonoid $X(T)_+ \subset X(T)$. Through the natural isomorphism
\begin{align*}
    \Z^m &\longrightarrow X(T),\\
    (k_1, \dots, k_m) & \longmapsto (\mr{diag}(t_1, \dots, t_m) \mapsto t_1^{k_1} \cdots t_n^{k_m})
\end{align*}
dominant weights correspond to tuples $(k_1, k_2, \dots, k_m) \in \Z^m$ such that $k_i \geq k_{i+1}$. We will sometimes write $\ul{k} = (k_1, k_2, \dots, k_m) \in X(T)$ for short. The projection $B^- \to T$ allows us to consider $\ul{k}$ a degree one representation of $B^-$. \par
Thanks to \cite[I.5.8]{jantzen}, we can associate to any $B^-$-module $M$ a sheaf $\mc{L}(M)$ on the projective quotient $\pi \colon G \to G/B^-$ and, in particular, we can apply this construction to any $\ul{k} \in X(T)$. From \cite[I.5.12]{jantzen}, we see that $H^0({G}/B^-, \mc{L}_{\ul{k}})$ is isomorphic to $\mr{ind}_{B^-}^{G}(\ul{k})$. To ease notations, like \cite{jantzen}, we write
\[
    H^i(\ul{k}) \coloneqq H^i(G/B^-, \mc{L}_{\ul{k}}), \, H^i(\ul{k}, s) \coloneqq H^i((G/B^-)_s, \mc{L}_{\ul{k}, s}),
\]
for $s \in \spec R$.
\par 
Let us also recall that for $\ul k$ dominant and such that $k_m \geq 0$, that is, when $\ul k$ is a \emph{partition}, we can define a \emph{Schur functor}
\[
    \mathbb{S}^{\ul{k}} \colon \ul{\mr{Mod}}_R \to \ul{\mr{Mod}}_R
\]
from the category of $R$-modules to itself (or, more generally, the category of quasi-coherent sheaves over a scheme $S$). See for instance \cite{schurschur} or \cite[Ch.~8]{fultontableaux}. Let $\ul l$ denote the partition of $\lvert \ul k \rvert = \sum_i k_i$ conjugate to $\ul k$, that is, $l_j \coloneqq \# \{ i \mid k_i \geq j \}$, for $j \geq 1$. For $M \in \ul{\mr{Mod}}_R$, we can describe it as
\[
    \mathbb{S}^{\ul{k}}(M) = \im( \otimes_j \wedge^{l_j}(M) \to \otimes_j M^{\otimes l_j} \to M^{\otimes \lvert k \rvert} \to \otimes_i M^{\otimes k_i} \to \otimes_i \mr{Sym}^{k_i}(M))
\]
where the first map is induced by $\wedge^l M \to M^{\otimes l}, m_1 \wedge m_2 \wedge \cdots \wedge m_l \mapsto \sum_{\nu \in \mr{S}_l} \mr{sgn}(\nu) m_{\nu(1)} \otimes m_{\nu(2)} \otimes \cdots m_{\nu(l)}$, the maps in the middle are identifications of $M^{\otimes \lvert k \rvert}$ grouping its factors first in terms of the columns and then in terms of the rows of the Young diagram for $\ul k$ and the last map is the natural quotient. In particular, if $M$ is free of rank $r$ with ordered basis $m_1, m_2, \dots, m_r \in M$, so is $\mathbb{S}^{\ul{k}}(M)$, with basis elements corresponding with \emph{semi-standard Young tableaux} of diagram $\ul k$ and entries in $\{1, 2, \dots, r\}$.
\begin{example}
Let $M \in \ul{\mr{Mod}}_R$ be finite locally free of constant rank $r$. Then, for $\ul k = (k, 0, \dots, 0)$, we get $\mb{S}^{\ul k}(M) = \mr{Sym}^k(M)$. For $\ul k = (1, 1, \dots, 1)$, instead, we have $\mb{S}^{\ul k}(M) = \wedge^m (M)$ (which is $0$, if $r < m$).
\end{example}

\begin{proposition}
\label{propalgreps}
We have the following:
\begin{enumerate}
    \item The sheaf $\mc{L}_{\ul{k}}$ is invertible and, if $\ul{k}$ is dominant, ample.
    \item The $G$-module $H^0(\ul{k})$ is finite, locally free over $R$ and, for a dominant weight $\ul{k}$, non-zero. Moreover, $H^i(\ul{k})=0$, for $i>0$, for $\ul{k}$ dominant.
    \item If $k_m \geq 0$, we have a natural isomorphism $\mathbb{S}^{\ul{k}}(\mr{st}) \cong H^0(\ul{k})$ in $\ul{\mr{Re}}\mr{p}_R(G)$, where $\mr{st} = R^m$ denotes the standard representation of $G$.
\end{enumerate}
\end{proposition}
\begin{proof}
Since $R$ is noetherian, so is $G/B^-$. Moreover, by \cite[II.1.10(2)]{jantzen}, $\pi \colon G \to G/B^-$ is a locally trivial quotient map, so by \cite[I.5.16(2)]{jantzen} $\mc{L}_{\ul{k}}$ is an invertible sheaf. By \cite[II.4.4]{jantzen}, for $\ul{k}$ dominant, $\mc{L}_{\ul{k}, s}$ is ample for all $s \in \spec R$, so that $\mc{L}_{\ul{k}}$ is ample on $G/B^-$.
\par Since $\mc{L}_{\ul{k}}$ is an invertible sheaf over $G/B^-$ and $G/B^-$ is flat over $R$, $\mc{L}_{\mr{k}}$ is flat over $\spec R$. The fact that $H^0(\ul{k})$ is finitely generated over $R$ follows from \cite[I.5.12(c)]{jantzen}. Since $H^0(\ul{k})$ is a submodule of $R[X_{i,j}, \det X^{-1}]_{1 \leq i,j \leq m}$, it is flat, because $R$ is a Dedekind domain and $R[X_{i,j}, \det X^{-1}]$ is flat, see  \cite[\href{https://stacks.math.columbia.edu/tag/092S}{Lemma 092S}]{stacks-project}. Therefore $H^0(\ul{k})$ is finite locally free. Let as above $s \in \spec R$. If $\ul{k}$ is dominant, by \cite[II.2.6]{jantzen} and by {Kempf's vanishing theorem}, \cite[II.4.5]{jantzen}, we have that
\[
    H^0(\ul{k}, s) \neq 0, \quad H^i(\ul{k}, s)=0 \text{ for } i>0,
\]
respectively. Therefore, by semicontinuity and Grauert's theorem, $H^i(\ul{k}) = 0$ for $i>0$ and thus, by proper base change, $H^0(\ul{k}) \otimes_R k(s) \cong H^0(\ul{k}, s)$, so that $H^0(\ul{k}) \neq 0$, in particular. 
\par 
Let $\ul k$ be a dominant weight such that $k_m \geq 0$ and let $I$ be a semi-standard Young tableaux of diagram $\ul k$. Let $e_I \in \mb{S}^{\ul k}(\mr{st})$ denote the corresponding basis element, obtained from the standard basis $e_1, \cdots, e_m \in \mr{st}$. In particular, write $I_0$ for the tableau obtained filling the $i$-th row with $k_i$ copies of $i$, so that the element $e_{I_0}$ is the image of
\[
    (e_1 \wedge e_2 \wedge \cdots \wedge e_{l_1}) \otimes (e_1 \wedge \cdots \wedge e_{l_2}) \otimes \cdots \otimes (e_1 \wedge \cdots \wedge e_{l_m}) \in \otimes_j \wedge^{l_j}(\mr{st})
\]
via the map $\otimes_j \wedge^{l_j}(\mr{st}) \to \mb{S}^{\ul k}(\mr{st})$. One can see that $e_{I_0}$ has $T$-weight $\ul k$. Given $I$, we can also construct an element $f_I$ of $H^0(\ul k) \subseteq R[X_{i, j}, \det X^{-1}]$. For each column of $\ul k$, say the $j$-th one, consider the minor determinant $M_{I, j}$ of the sub-matrix of $X =( X_{i, j} )$ obtained removing the first $l_j$ columns of $X$ and the rows corresponding to the entries in the $j$-th column of the tableau $I$; if $l_1 = m$, then we set $M_{I, j} = 1$ (empty determinant). Then we take $f_I = \det(X)^{-\mr{len}(\ul l)} \prod_{j} M_{I, j}$, where $\mr{len}(\ul l) \coloneqq \#\{j \mid l_j \neq 0\}$ is the \emph{length} of $\ul l$. Notice that $\mr{len}(\ul l) = k_1$.
We can define the map $\phi \colon \mb{S}^{\ul k}(\mr{st}) \to H^0(\ul k)$ sending $e_I$ to $f_I$. Using the Cauchy-Binet formula for minors of a given rank, one can check that, indeed, $f_I \in H^0(\ul k)$ and the map $\phi$ is $G$-linear. Since both $\mb{S}^{\ul k}(\mr{st})$ and $H^0(\ul k)$ are finitely generated and projective, if $\phi \otimes_R k(s)$ is an isomorphism for all $s$, then so is $\phi$. We now show that this is the case. Both $\mb{S}^{\ul{k}}(\mr{st}_s)$ and $H^0(\ul{k}, s)$ have a one-dimensional highest weight space of weight $\ul{k}$, which coincides with the space of $U_s$-invariants. For $H^0(\ul{k}, s)$ this is \cite[II.2.2(a)]{jantzen}. For $\mb{S}^{\ul{k}}(\mr{st}_s)$ a highest weight vector is given by $e_{I_0}$ and the action of a unipotent matrix $u \in U_s \subset G_s$ on $\mb{S}^{\ul{k}}(\mr{st}_s)$ can be described explicitly on the basis $\{e_I\}$; the resulting matrix $u_{\ul{k}}$ (ordering lexicographically the basis $\{e_I\}$) is unipotent upper-triangular. One can use this description of the action of $U_s$ on the elements $e_I$ to show that $(\mb{S}^{\ul{k}}(\mr{st}_s))^U$ is one dimensional. Therefore, both $\mb{S}^{\ul{k}}(\mr{st}_s)$ and $H^0(\ul{k}, s)$ have a simple socle containing the highest weight vector, see the proof of \cite[II.2.3]{jantzen}. Since $\phi \otimes k(s)$ maps a highest weight vector to a non-zero vector in $H^0(\ul{k}, s)$, $\phi \otimes k(s) \neq 0$ and it is an isomorphism between the socles. In particular, since the socle of $\mb{S}^{\ul{k}}(\mr{st}_s)$ is simple, $\ker (\phi \otimes k(s)) = 0$. To conclude, it is enough to prove that $\phi \otimes k(s)$ is surjective.
Since the construction of both $\mb{S}^{\ul k}(\mr{st})$ and $H^0(\ul k)$ is compatible with base change, we can assume without loss of generality that $R = \Z$. For $s = (0)$, we have an isomorphism because of the Borel-Bott-Weil theorem, see \ref{reducibility}. Thus $\mb{S}^{\ul k}(\mr{st})$ and $H^0(\ul k)$ have the same rank and therefore, being injective, $\phi \otimes k(s)$ is an isomorphism for all $s \in \spec R$.
\end{proof}

\begin{remark}
\label{reducibility}
The fibres $H^0(\ul{k}) \otimes_R k(s) \cong H^0(\ul{k}, s)$ are $\mr{GL}_{m, k(s)}$-modules of finite $k(s)$-dimension and the socle 
\[
    L(\ul{k}, s) = \mr{soc}\, H^0(\ul{k}, s)
\]
is a simple $\mr{GL}_{m, k(s)}$-module, \cite[II.2.3]{jantzen}. Moreover, varying $\ul{k} \in X(T)_+$ we obtain all simple $\mr{GL}_{m, k(s)}$-modules, \cite[II.2.7]{jantzen}. In fact, for $\mr{char}\, k(s) = 0$ we have, as a consequence of the \emph{Borel-Bott-Weil theorem}, that $L(\ul{k}, s) = H^0(\ul{k}, s)$, see \cite[II.5]{jantzen} and the discussion there. When $\mr{char}\, k(s) > 0$ this does not hold for general $m$ and $\ul{k}$: for $m = 2$, with $\mr{char}\, k(s) = p$ we can consider $\ul{k}=(p, 0)$ and $\mr{Sym}^p(\mr{st})$ is reducible, since it contains as a proper submodule $\mr{st}^{(p)}$, the $p$-twist of the standard representation, via $e_i \otimes 1 \mapsto e_i^p, \, i=1,2$.
\end{remark}
We call the $G$-module $H^0(\ul{k})$ the \emph{(dual) Weyl module} associated to the dominant weight $\ul{k}$. Let $\ul{k}, \ul{k}' \in X(T)_+$. Then $\ul{k} + \ul{k}' \in X(T)_+$ and we have a morphism of $G$-representations
\begin{equation}
    \label{morphrep}
    H^0(\ul{k}) \otimes_R H^0(\ul{k}') \longrightarrow H^0(\ul{k}+\ul{k}')
\end{equation}
defined by multiplication of global sections. This will allow us to define corresponding morphisms of twisted modules.

\begin{remark}[Representations of $\mr{GL}_1$]
\label{repgl1}
We excluded the commutative case $m=1$ of $\mr{GL}_1 = \mb{G}_m$ from our discussion, but this is easily dealt with. In the notations of \cite[I.2]{jantzen}, we have that $\mb{G}_m \cong \mr{Diag}(\Z)$, a special case of a torus $T \cong \mr{Diag}(\Z^m)$ that we implicitly used above, which implies that every algebraic representation $M$ of $\mb{G}_m$ decomposes as
\[
    M \cong \oplus_{n \in \Z} M_n
\]
the action of $\mb{G}_m$ on the simple module $M_n$ being given on points by $(z, m) \mapsto z^n \cdot m$, see \cite[I.2.11]{jantzen} for more details.
\end{remark}

\subsection{Twist of a sheaf by a representation}
\label{thetwist}
We will recall the construction of the twist, or contracted product, of a torsor by a representation, following \cite{sga1}.
\par Let $R$ be a Dedekind domain, $S \in \ul{\mr{Sch}}_R$, $\rho \in \ul{\mr{Re}}\mr{p}_R(\mr{GL}_m)$ and take $\mc{F}$ a locally free $\mc{O}_S$-sheaf of rank $m$ over $S$. We can consider the {contracted product}
\[
    F^{\rho} = \ul{\mr{Isom}}(\mc{O}_S^m, \mc{F}) \times^{\mr{GL}_m} \rho
\]
defined as the (fppf or Zariski) sheafification of 
\[
    T \longmapsto (\mr{Isom}(\mc{O}_T^m, \mc{F}_T) \times \rho(T))\big / \mr{GL}_m(T),
\]
where $T \in \ul{\mr{Sch}}_S$ and $\mr{GL}_m(T)$ acts on the right via $(x,y)\cdot g = (xg, g^{-1}y)$, for $x \in \mr{Isom}(\mc{O}_T^m, \mc{F}_T), y \in \rho(T), g \in \mr{GL}_m(T)$. This is represented by a vector bundle $F^\rho \to S$, see the discussion following \cite[XII.4.3]{sga1}. Taking the sheaf of sections of $F^\rho$, we obtain the sheaf $\mc{F}^\rho$, which we call the \emph{twist of $\mc{F}$ by $\rho$}.
\begin{example}
Consider $\rho = \mb{S}^{\ul{k}}(\mr{st})$, for some dominant weight $\ul{k}$ with $k_m \geq 0$. Then $\mc{F}^\mr{\rho} = \mb{S}^\ul{k}(\mc{F})$. To ease notations, we will denote this twisted sheaf by $\mc{F}^{\ul k}$.
\end{example}
The construction $\mc{F} \mapsto \mc{F}^\rho$, for a fixed $\rho$, is compatible with isomorphisms $\mc{F} \cong \mc{F}'$: for every isomorphism $\phi \colon \mc{F} \to \mc{F}'$ of locally free sheaves of finite rank $m$ over $S$, we can define a natural isomorphism $\phi^\rho \colon \mc{F}^\rho \to \mc{F}'^\rho$. Similarly, one can see that it is in fact compatible with the $p$-twist operation, when we work on $S$ of positive characteristic.
\par Moreover, the construction of the twist is also functorial in the representation $\rho$ and it gives rise to a {tensor functor}:
\[
    \rho \otimes \rho' \rightsquigarrow \mc{F}^{\rho} \otimes \mc{F}^{\rho'}.
\]
In particular, for $\ul{k}, \ul{k}' \in X(T)_+$ dominant weights, we have a natural morphism of sheaves
\[
    \mc{F}^{\ul{k}} \otimes \mc{F}^{\ul{k}'} \longrightarrow \mc{F}^{\ul{k}+\ul{k}'},
\]
coming from \ref{morphrep}.
We will often use this morphism implicitly in our constructions.
\begin{remark}
Notice that for $S=\spec{k(s)}$ a point $s \in \spec R$ of $\mr{char}\, k(s) = p$ we have that $L(\ul{k}, s) \subseteq H^0(\ul{k}, s)$, the inclusion being proper for most $\ul{k}$ (that is, $\ul{k}$ not in the {$p$-restricted alcove}, see \cite[II.5.7]{jantzen}). It seems to be a hard task, for general $m$, to characterise explicitly the Jordan--H\"older factors of $L(\ul{k}, s) \otimes L(\ul{k}', s)$. It would be interesting, working towards more general definitions of theta operators in positive characteristic, to describe the possible simple quotients $L(\ul{w}, s)$ of such tensor products, so as to obtain morphisms of sheaves of the form
\[
    \mc{F}^{L(\ul{k}, s)} \otimes \mc{F}^{L(\ul{k}', s)} \longrightarrow \mc{F}^{L(\ul{w}, s)},
\]
which would provide a wider range of weight shifting operators. Some descriptions amenable to explicit computations are available are $m=2, 3$, see \cite{dotyhenke}, \cite{bowmandoty}. In a similar way, one could consider maps $\mc{F}^{\ul k} \otimes \mc{F}^{\ul k'} \to \mc{F}^{\ul l}$, where $\mb{S}^{\ul l}$ is any representation obtained applying the Littlewood--Richardson rule to $\mb{S}^{\ul k} \otimes \mb{S}^{\ul k'}$.
This is a direction of research that we do not pursue here, but we plan to return on these matters in the future.
\end{remark}

\subsection{The automorphic sheaves}
Take now $R = \mc{O}_{E, (p)}$, a localisation of $\mc{O}_{E, \mr{ur}}$. In particular, we have that $\mathbf{G}_{1,R} \cong \mr{GL}_{n, R}$, thanks to \ref{Levidecomp}, and the Levi subgroup $H_{1, R}$ described there is isomorphic to $\mr{GL}_{n-1} \times \mr{GL}_1$. Let $S \in \ul{\mr{Sch}}_R$ and consider as above $\ul{A} = (A, \lambda, \iota, \eta_K) \in \mc{S}_K(S)$. We can associate to $\ul{A}$ the Hodge sheaves $\ul{\omega}_{A/S, \sigma}$ and $\ul{\omega}_{A/S, \overline{\sigma}}$, locally free of ranks $n-1$, $1$, respectively. If we take $\ul{k}$ a dominant weight for $\mr{GL}_{n-1}$ and $l \in \Z$ a weight for $\mr{GL}_1 \cong \mr{Diag}(\Z)$, see \ref{repgl1}, we can define the twisted sheaves
\[
    \ul{\omega}_{A/S}^{\ul{k}, l} \coloneqq \ul{\omega}_{A/S, \sigma}^{\ul{k}} \otimes \ul{\omega}_{A/S, \overline{\sigma}}^{\otimes l}.
\]
This would be a natural choice for a notion of automorphic sheaves relative to the moduli problem $\mc{S}_K$. But recall that we have the relation
\[
    \det \ul{\omega}_{A/S, \sigma} \cong \ul{\omega}_{A^\vee/S, \sigma} \otimes \delta_{A/S, \sigma} \underset{\lambda}{\cong} \ul{\omega}_{A/S, \overline{\sigma}} \otimes \delta_{A/S, \sigma}
\]
induced by the polarisation $\lambda$, from \ref{determinanteq}. In particular, using the polarisation $\lambda$, we can ``absorb'' the term $\ul{\omega}_{A/S, \overline{\sigma}}^{l}$ in $\ul{\omega}_{A/S, \sigma}^{\ul{k}}$, by multiplying by $\delta_{A/S, \sigma}^{-l}$. So we can alternatively consider, for $\ul{k}$ dominant as above and $w \in \Z$, the sheaves
\[
    \ul{\omega}_{A/S}^{\ul{k}, w} \coloneqq \ul{\omega}_{A/S, \sigma}^{\ul{k}} \otimes \delta_{A/S, {\sigma}}^{ w}.
\]
This convention will come with the advantage of allowing us to define Hecke-equivariant theta operators.
\begin{definition}[Automorphic sheaves]
Let $\ul{k}$ and $w$ be as above. We call $\ul{\omega}_{A/S}^{\ul{k}, w}$ the \emph{automorphic sheaf of weight $(\ul{k}, w)$ over $S$ associated to $\ul{A}$}.
\end{definition}
Since this construction is compatible with base change, we will work mainly with the automorphic sheaves of weights $(\ul{k}, w)$ associated to $\ul{A}$ in the universal class of $\mc{S}_K$, or some base change of it, for instance a $\mr{mod}\, p$ reduction $S_K$ of $\mc{S}_K$. We will denote these sheaves simply by $\ul{\omega}^{\ul{k}, w}$. For $R$ any $\mc{O}_{E, (p)}$-algebra, let $B = \spec R$ and consider $\mc{S}_{K, B}$ the base change of $\mc{S}_K$. Then we can consider the global sections of $\ul{\omega}^{\ul{k}, w}_B \coloneqq \ul{\omega}_{A/\mc{S}_{K,B}}^{\ul{k}, w}$.
\begin{definition}
We call $H^0(\mc{S}_{K, B}, \ul{\omega}^{\ul{k}, w}_B)$ the space of \emph{modular forms of weight $(\ul{k}, w)$ and level $K$ with coefficients in $R$}.
\end{definition}

\begin{remark}[Algebraic Koecher principle]
In the setting in which we are working, the results of \cite{highkoecher} apply, since $\dim(\mc{S}_K^\mr{min} \setminus \mc{S}_K) = n-1 \geq 2$, and we have, for a smooth, projective, toroidal compactification $\mc{S}_K^\mr{tor}$ and the minimal compactification $\mc{S}_K^\mr{min}$, that
\[
    H^0(\mc{S}_{K,B}, \ul{\omega}^{\ul{k}, w}_B) = H^0(\mc{S}^\mr{tor}_{K, B}, \ul{\omega}^{\mr{can}, \ul{k}, w}_B) = H^0(\mc{S}^\mr{min}_{K, B}, \ul{\omega}^{\mr{min}, \ul{k}, w}_B).
\]
This is called the \emph{Koecher principle}. This will allow us to carry out most of our constructions and computations on $\mc{S}_{K, B}$, while still being able to deduce interesting consequences that make use of the geometry of $\mc{S}_{K, B}^{\mr{tor}}$ and $\mc{S}_{K, B}^{\mr{min}}$ and the extensions of the bundles $\ul{\omega}^{\ul{k}, w}_B$.
\end{remark}

\section{Hecke operators}
\label{sectamehecke}
In this section we recall the definition of the algebra of tame Hecke operators via cohomological correspondences.
\par Let $f \colon X \to Y$ be a finite flat morphism of schemes (hence separated and quasi-compact). In particular, one can consider the {trace morphism} $\mr{tr} \colon f_\ast(\mc{O}_X) \to \mc{O}_Y$, a morphism of quasi-coherent $\mc{O}_Y$-modules. Using the trace we can define a duality morphism, denoted $\mr{Tr}$,
\[
    \mr{Tr} \colon f_\ast(\mc{O}_X) \longrightarrow f_\ast(\mc{O}_X)^\vee = \ul{\hom}_{\mc{O}_Y}(f_\ast(\mc{O}_X), \mc{O}_Y).
\]
One can see that $f_\ast(\mc{O}_X)^\vee$ has a natural structure of sheaf of $f_\ast(\mc{O}_X)$-modules and $\mr{Tr}$ is the morphism of $f_\ast(\mc{O}_X)$-modules defined by $1 \mapsto (\mr{tr} \colon f_\ast(\mc{O}_X) \to \mc{O}_Y)$. When $\mr{Tr} \colon f_\ast(\mc{O}_X) \longrightarrow f_\ast(\mc{O}_X)^\vee$ is an isomorphism we say that $f$ is \emph{separable}. 
\begin{lemma}
\label{lemsepet}
The finite flat morphism $f$ is separable if and only if it is {unramified}, that is, $\Omega^1_{X/Y}=0$. Therefore, $f$ is separable if and only if it is \'etale.
\end{lemma}
\begin{proof}
Since $f$ is of finite presentation, being unramified is the same as being formally unramified. Moreover, the statement is local in nature, so we can assume that $Y = \spec A, X = \spec B$, with $A \to B$ a finite flat morphism. This is then \cite[Thm.~4.6.7-8, 8.3.6]{fordsepalg}.
\end{proof}
Consider $\mc{G}$ a quasi-coherent $\mc{O}_Y$-module over $Y$. Then we can endow
\[
    f_\ast(\mc{O}_X)^\vee \otimes_{\mc{O}_Y} \mc{G} = \ul{\hom}_{\mc{O}_Y}(f_\ast(\mc{O}_X), \mc{G})
\]
with a structure of quasi-coherent $f_\ast(\mc{O}_X)$-module. Since $f$ is affine, any quasi-coherent $f_\ast(\mc{O}_X)$-module is $f_\ast$ of some quasi-coherent $\mc{O}_X$-module, in a natural, functorial way. We write $f^! \mc{G}$ for the $\mc{O}_X$-sheaf corresponding to $f_\ast(\mc{O}_X)^\vee \otimes_{\mc{O}_Y} \mc{G}$. The trace map $\mr{Tr}$ induces a morphism
\[
    \mr{Tr}_\mc{G} \colon f_\ast f^\ast \mc{G} = f_\ast \mc{O}_X \otimes_{\mc{O}_Y} \mc{G} \longrightarrow f_\ast(\mc{O}_X)^\vee \otimes_{\mc{O}_Y} \mc{G} = f_\ast f^! \mc{G},
\]
hence, by the same correspondence, a morphism $\mr{Tr}_\mc{G} \colon f^\ast \mc{G} \longrightarrow f^!\mc{G}$.
By \ref{lemsepet}, this map is an isomorphism when $f$ is finite \'etale.
The functor $f^!\colon \ul{\mr{QCoh}}_Y \to \ul{\mr{QCoh}}_X$ is right adjoint to $f_\ast$, that is, we have a natural functorial isomorphism
\[
    \hom_{\mc{O}_Y}(f_\ast \mc{F}, \mc{G}) \cong \hom_{\mc{O}_X}(\mc{F}, f^!\mc{G}), \quad \forall \mc{F} \in \ul{\mr{QCoh}}_X, \mc{G} \in \ul{\mr{QCoh}}_Y.
\]
From this adjunction we obtain the natural counit morphism $f_\ast f^! \mc{G} \to \mc{G}$. The composition
\[
    \mr{tr}_\mc{G} \colon f_\ast f^\ast \mc{G} \overset{\mr{Tr}_\mc{G}}{\longrightarrow} {f_\ast f^!\mc{G}} \longrightarrow \mc{G}
\]
is also sometimes called trace morphism. When $\mc{G} = \mc{O}_Y$, this map is the same trace morphism $f_\ast \mc{O}_X \to \mc{O}_Y$ we began with. Notice that we have a morphism of $\mc{O}_Y$-modules 
$\mc{G} \longrightarrow f_\ast f^\ast \mc{G}$ and, since $f$ is faithfully flat, it is injective.
\par Let $X, Y$ be two schemes over a base scheme $S$ and $\mc{F}, \mc{G}$ two quasi-coherent modules, over $X$ and $Y$ respectively.
\begin{definition}
A \emph{finite flat correspondence} between $(X, \mc{F})$ and $(Y, \mc{G})$ is the datum of $(Z, p_1, p_2, \phi)$ a third scheme over $S$ with finite flat morphisms $p_1\colon Z \to X$ and $p_2 \colon Z \to Y$ of $S$-schemes, together with a morphism
\[
    \phi \colon p_2^\ast \mc{G} \longrightarrow p_1^\ast \mc{F}
\]
of quasi-coherent sheaves over $Z$.
\end{definition}
Notice that being $p_1, p_2$ affine, their pushforwards $p_{1,\ast}, p_{2,\ast}$ are exact (on quasi-coherent sheaves). In particular we have
\[
    H^\bullet(X, p_{1,\ast} p_1^\ast\mc{F}) = H^\bullet(Z, p_1^\ast\mc{F}) \text{ and } H^\bullet(Y, p_{2,\ast} p_2^\ast\mc{G}) = H^\bullet(Z, p_2^\ast\mc{G}),
\]
so that we can consider 
\begin{align*}
    T_{\phi} = T^i_{\phi} \colon H^i(Y, \mc{G}) \overset{p_2^\ast}{\longrightarrow} H^i(Z, p_2^\ast\mc{G}) \overset{H^i(\phi)}{\longrightarrow} H^i(Z, p_1^\ast \mc{F}) \overset{\mr{tr}_\mc{G}}{\longrightarrow} H^i(X, \mc{F}), \quad i \geq 0.
\end{align*}
Therefore, from a cohomological correspondence $Z$ we deduce a family of morphisms $T_\phi^i$ in cohomology. If $S = \spec R$, then these morphisms are $R$-linear between $R$-modules.
\par Write $\mathbf{G}^{p, \infty} = \mathbf{G}(\mb{A}^{p, \infty})$, consider $g \in \mathbf{G}^{p, \infty}$ and take $K_g = K \cap gKg^{-1}$. In particular, $K_{g,p} = K_p = \mathbf{G}(\Z_p)$ and since we assumed that $K$ is neat, so is $K_g$. Consider also $g^{-1} K_g g = K \cap g^{-1} K g$. Then we have two finite \'etale maps
\[
    p_1 \colon \mc{S}_{K_g} \longrightarrow \mc{S}_K, \quad p_2 \colon \mc{S}_{K_g} \longrightarrow \mc{S}_K.
\]
We sketch their definitions using the language and notations of \ref{lvlstr}. The first one is defined on points $\ul{A} \in \mc{S}_{K_g}$ by sending
\[
    \ul{A} = (A, \iota, \lambda, \eta_{K_g}) \longmapsto (A, \iota, \lambda, \eta_{K})
\]
where $\eta_K$ is the level $K$-structure obtained from $\eta_{K_g}$ taking the $K$-orbit of any $\eta$ in $\eta_K$. In particular, we see that $p_1$ has degree given by the cardinality of the finite group $K^p/K^p_g$. On the other hand, $p_2$ is obtained as the composition of the finite, \'etale, surjective morphism $\mc{S}_{g^{-1}K_g g} \longrightarrow \mc{S}_K$ defined like $p_1$ with the map
\begin{align*}
    [g]\colon \mc{S}_{K_g} &\longrightarrow \mc{S}_{g^{-1}K_g g},\\
    (A, \iota, \lambda, \eta_{K_g}) &\longmapsto (A', \iota', \lambda', \eta'_{g^{-1}K_g g})
\end{align*}
where $A'$ is the unique (up to isomorphism) abelian scheme prime-to-$p$ quasi-isogenous to $A$ via $f \colon A \to A'$ such that $V^p(f_{s})^{-1}(T^p A'_{s}) = \eta(g(\Lambda \otimes_{\Z} \hat{\Z}^p))$, ${s}$ any geometric point in $S$, and $\iota', \lambda', \eta'_{g^{-1}K_g g}$ defined as in \cite[Prop.~1.4.3.4]{kwl}. See also \cite[Cor.~1.3.5.4]{kwl}. Then, $[g]$ is clearly an isomorphism with inverse $[g^{-1}]$. Notice that for $g \in K^p$ both $p_1$ and $p_2$ are the identity and for any $g' \in K^p g K^p$ the morphisms $p_{1, g'}, p_{2, g'}$, obtained through the same construction using $g'$ instead of $g$, are the same as $p_{1, g}$ and $p_{2, g}$, respectively. In particular, both $p_1$ and $p_2$ really only depend on the double coset $K^p g K^p$.
\par For every automorphic weight $(\ul{k}, w)$ we have the natural identifications given by base change
\[
    p_1^\ast \ul{\omega}_{A/\mc{S}_K}^{\ul{k}, w} \cong \ul{\omega}_{A/\mc{S}_{K_g}}^{\ul{k}, w} \cong p_2^\ast \ul{\omega}_{A/\mc{S}_K}^{\ul{k}, w}.
\]
More generally, suppose that we have a family $\{\mc{F}_{K'}\}_{K'}$, where the ${K'} \subseteq K$ are neat $p$-hyperspecial compact opens, and the $\mc{F}_{K'}$ are quasi-coherent sheaves on $\mc{S}_{{K'}}$ compatible with pullback via the morphisms
\begin{align*}
    \mc{S}_{K'} &\longrightarrow \mc{S}_{K''}, \quad K' \subseteq K'' \subseteq K, \\
    [g] \colon \mc{S}_{K'} &\longrightarrow \mc{S}_{g^{-1}K_g' g}, \quad g \in \mathbf{G}^{p, \infty}.
\end{align*}
In particular, we obtain a natural cohomological correspondence $(\mc{S}_{K_g}, p_1, p_2, \phi)$ and, as above, $\mc{O}_{E, (p)}$-linear operators
\[
    T_g = T_g^i \colon H^i(\mc{S}_K, \mc{F}_K) \longrightarrow H^i(\mc{S}_K, \mc{F}_K).
\]
We call $T_g$ the {(tame) Hecke operator} given by $g \in \mathbf{G}^{p, \infty}$. We will use the action of these operators over more general bases. First of all, let $R \in \ul{\mr{Al}}\mr{g}_{\mc{O}_{E,(p)}}$. Then we can pull the Hecke correspondence back to $B = \spec R$ and obtain $R$-linear operators 
\[
    T_g \colon H^i(\mc{S}_{K, B}, \mc{F}_K) \longrightarrow H^i(\mc{S}_{K, B}, \mc{F}_K).
\]
Furthermore, suppose that we have a family of locally closed (resp.\ closed, resp.\ open) subschemes $\{\tilde{\mc{S}}_{K', B} \to \mc{S}_{K', B}\}_{K'}$, compatible in the sense that the diagrams
\[
    \begin{tikzcd}
        & \tilde{\mc{S}}_{K', B} \arrow[r] \arrow[d] 
        \arrow[dr, phantom, "\usebox\pullback" , very near start, color=black] & \mc{S}_{K', B} \arrow[d, "p_{K'/K''}"] \\
        & \tilde{\mc{S}}_{K'', B} \arrow[r] & \mc{S}_{K'', B}
    \end{tikzcd}
\]
for $K' \subseteq K'' \subseteq K$ are all cartesian. Then, we can again pull the Hecke correspondence back to $\tilde{\mc{S}}_{\cdot, B}$ and obtain Hecke operators
\[
    T_g \colon H^i(\tilde{\mc{S}}_{K, B}, \mc{F}_K) \longrightarrow H^i(\tilde{\mc{S}}_{K, B}, \mc{F}_K).
\]
This will allow us in particular to talk about Hecke operators acting on non-maximal Ekedahl-Oort strata. Notice that the action of $T_g$ makes sense for compatible families of quasi-coherent sheaves $\{\mc{F}_{K'}\}_{K'}$ defined on the subschemes $\{\tilde{\mc{S}}_{K', B}\}_{K'}$, even when these are not defined as restrictions of sheaves on the tower $\{\mc{S}_{K', B}\}_{K'}$.
\begin{remark}[Unramified Hecke Algebra]
From the observations made above, we see that we have in particular defined an $R$-linear action on $H^\bullet(\tilde{\mc{S}}_{K, B}, \mc{F}_K)$ of the \emph{unramified Hecke algebra}. This algebra can be defined as the the restricted tensor product (taken with respect to the identity elements) of the local Hecke algebras at unramified primes
\[
    \mc{H}_\mr{ur}(\mathbf{G}, K, R) = {\bigotimes}'_{\ell\,\centernot{\mid}\, p D_{E/\Q}} \mc{H}(\mathbf{G}_\ell, K_\ell, R),
\]
where $\mathbf{G}_\ell = \mathbf{G}(\Q_\ell)$. The local Hecke algebras $\mc{H}(\mathbf{G}_\ell, K_\ell, R)$ themselves are defined as the convolution algebras of locally constant and compactly supported functions $f\colon K_\ell \backslash \mathbf{G}_\ell \slash K_\ell \to R$. In particular, for all but finitely many primes $\ell$ we have that $K_\ell = \mathbf{G}(\Z_\ell)$. If, moreover, $\mathbf{G}_{\Q_\ell}$ is quasi-split, then we call $\ell$ a \emph{good prime}. All but finitely many primes are good. The local Hecke algebras at good primes $\ell$ are commutative and, in fact, can be described, up to some additional hypothesis on $R$, by the \emph{Satake isomorphism}, see \cite{grosssatake}. The action of the local Hecke algebras allows one to define \emph{Satake parameters} which, in turn, control the Galois representation (in general, conjecturally) associated to a Hecke eigenform, see, for instance, \cite[2.3.1]{EFGMM} and \cite{skinner}.
\end{remark}
\begin{remark}[Hecke operators at \texorpdfstring{$p$}{p}]
As shown in \cite[7]{fakpil}, one can define an action of certain Hecke operators defined \emph{at $p$}, for some $g \in \mathbf{G}(\Q_p)$. We will not discuss these operators, but we remark that they are still constructed from cohomological correspondences, even though the maps $p_1, p_2$, notation as above, involved are finite flat and \emph{not \'etale}, so that, in view of \ref{lemsepet}, one has to take the ramification they introduce into account and normalise the action accordingly.
\end{remark}

\section{Gauss-Manin connection \& Kodaira-Spencer map}
In this section we recall how the Gauss-Manin connection and the Kodaira-Spencer map are defined and how they are related. We then proceed to describe some of their basic properties. The classical reference for the definition of the GM connection is \cite{katzoda}. We will follow \emph{loc.\ cit.\ }closely, but also benefit from the exposition given in the first chapter of \cite{katzdiff} and in \cite{katznilpotent}. For the more context-specific, or arithmetic, properties of the KS morphism we will follow \cite[2.3.5]{kwl} and provide other references where we deemed it necessary.
\subsection{Generalities on filtrations}
\label{secgenonfilt}
We record here general, basic facts concerning filtrations. These will be essential in our construction of generalised theta operators.
\subsubsection{Tensor product of filtrations}
Let $S$ be a scheme and consider $(\mc{F}, F^\bullet(\mc{F})), (\mc{G}, F^\bullet(\mc{G}))$ finite locally free sheaves over $S$ endowed with finite, separated, exhaustive, descending filtrations
\begin{align*}
    \mc{F} &= F^0(\mc{F}) \supseteq F^1(\mc{F}) \supseteq \cdots \supseteq F^r(\mc{F}) \supseteq F^{r+1}(\mc{F}) = 0, \\
    \mc{G} &= F^0(\mc{G}) \supseteq F^1(\mc{G}) \supseteq \cdots \supseteq F^s(\mc{G}) \supseteq F^{s+1}(\mc{G}) = 0,
\end{align*}
by finite locally free subsheaves $F^i(\mc{F}), F^j(\mc{G})$, such that the graded pieces $\mr{gr}^i \mc{F}, \mr{gr}^j \mc{G}$ are themselves finite locally free. Then we can defined a \emph{tensor product filtration} on $\mc{F} \otimes_{\mc{O}_S} \mc{G}$ with the same properties as the two filtrations on $\mc{F}$ and $\mc{G}$ by setting:
\[
    F^k(\mc{F} \otimes \mc{G}) \coloneqq \sum_{i+j = k} F^i(\mc{F}) \otimes F^j(\mc{G}).
\]
From the hypotheses on $(\mc{F}, F^\bullet(\mc{F})), (\mc{G}, F^\bullet(\mc{G}))$ we can deduce the following.
\begin{lemma}[Grothendieck Group Decomposition]
\label{lemgrotgrpdecomp}
The graded pieces of $F^\bullet(\mc{F} \otimes \mc{G})$ can be described as:
\[
    \mr{gr}^k(\mc{F} \otimes \mc{G}) \cong \bigoplus_{i+j = k} \mr{gr}^i(\mc{F}) \otimes \mr{gr}^j(\mc{G}).
\]
\end{lemma}
Notice that if we have a finite family of finite locally free sheaves $(\mc{F}_i, F^\bullet(\mc{F}_{i}))_{1 \leq i \leq m}$ with filtrations subject to the same conditions as above, we can define by induction a filtration on their tensor product, which turns out to be
\[
    F^k\left(\bigotimes_{i=1}^m \mc{F}_i\right) \cong \bigoplus_{\lvert\ul{j}\rvert = k} F^{j_1}(\mc{F}_{1}) \otimes F^{j_2}(\mc{F}_{2}) \otimes \cdots \otimes F^{j_m}(\mc{F}_{m}),
\]
where $\ul{j} = (j_1, \dots, j_m) \in \Z^m$ and $\lvert \ul{j} \rvert = \sum_i j_i$. Its graded pieces are
\[
    \mr{gr}^k\left(\bigotimes_{i=1}^m \mc{F}_i\right) \cong \bigoplus_{ \lvert\ul{j}\rvert = k} \mr{gr}^{j_1}\mc{F}_{1} \otimes \mr{gr}^{j_2}\mc{F}_{2} \otimes \cdots \otimes \mr{gr}^{j_m}\mc{F}_{m}.
\]

\subsubsection{Dual filtration}
Let $S$ and $(\mc{F}, F^\bullet(\mc{F}))$ be as above. We can define a decreasing filtration on $\mc{F}^\vee$ by setting
\[
    F^i(\mc{F}^\vee) \coloneqq \left(\mc{F}/F^{r-i+1}(\mc{F})\right)^\vee.
\]
Its graded pieces are
\[
    \mr{gr}^i(\mc{F}^\vee) \cong F^i(\mc{F}^\vee) / F^{i+1}(\mc{F}^\vee) \cong \frac{(\mc{F}/F^{r-i+1}(\mc{F}))^\vee}{(\mc{F}/F^{r-i}(\mc{F}))^\vee} \cong \mr{gr}^{r-i}(\mc{F})^\vee.
\]

\subsubsection{Filtration of the exterior algebra}
\label{seckoszulfiltext}
Let $S$ be as above and consider a ses of finite locally free sheaves over $S$ of constant rank
\[
    0 \longrightarrow \mc{F}' \longrightarrow \mc{F} \longrightarrow \mc{F}'' \longrightarrow 0.
\]
The \emph{Koszul filtration} of $\wedge^j \mc{F}$, for some $0 \leq j \leq \mr{rk}(\mc{F})$, is defined by
\[
    K^i(\wedge^j \mc{F}) \coloneqq \im(\wedge^i \mc{F}' \otimes_{\mc{O}_S} \wedge^{j - i} \mc{F} \longrightarrow \wedge^j \mc{F}),\quad 0 \leq i \leq j.
\]
This filtration can also obtained by taking quotients from the tensor product filtration on $\mc{F}^{\otimes j}$. We can extend this to an exhaustive, separated filtration by setting $K^i(\wedge^j \mc{F}) = \wedge^j \mc{F}$ for $i<0$ and $K^i(\wedge^j \mc{F}) = 0$ for $i>j$. In that case, $(K^i(\wedge^\bullet \mc{F}))_{i \in \Z}$ is an exhaustive, separated filtration of $\wedge^\bullet \mc{F}$.
In particular, we have $K^0(\wedge^\bullet \mc{F}) = \wedge^\bullet \mc{F}$ and $K^i(\wedge^i \mc{F}) = \wedge^i \mc{F}'$, which could be $0$, if $i > \mr{rk}(\mc{F}')$. The filtration is compatible with the algebra structure given by the $\wedge$ product in the sense that $K^i \wedge K^{i'} \subseteq K^{i+i'}(\wedge^\bullet \mc{F})$.
The graded terms of the filtration are
\[
    \mr{gr}^i(\wedge^j \mc{F}) = K^i/K^{i+1} \cong \wedge^{i} \mc{F}' \otimes_{\mc{O}_S} \wedge^{j-i} \mc{F}'', \quad 0 \leq i \leq j \leq \mr{rk}(\mc{F}).
\]
In particular, when $j=\mr{rk}(\mc{F})$ and $i=\mr{rk}(\mc{F}')$ we get the usual isomorphism of invertible sheaves
\[
    \det \mc{F} \cong \det \mc{F}' \otimes \det \mc{F}''.
\]
\subsubsection{Filtration of the symmetric algebra}
\label{secfiltsym}
Let $S$ be as above. Let $\mc{G}$ be a finite locally free sheaf. In what follows, for $j<0$, we set
\[
    \mr{Sym}^j(\mc{G}) \coloneqq \left( \mr{Sym}^{-j}(\mc{G}) \right)^\vee
\]
and whenever $\mr{Sym}^{-j}(\mc{G})$ is endowed with a filtration, $\mr{Sym}^j(\mc{G})$ is endowed with the dual filtration.
\par One can consider $\mc{F}$ a finite locally free sheaf over $S$, filtered by finite locally free subsheaves with finite locally free quotients
\[
    \mc{F} = F^0(\mc{F}) \supseteq F^1(\mc{F}) \supseteq \cdots \supseteq F^r(\mc{F}) \supseteq F^{r+1}(\mc{F}) = 0,
\]
and thus, by taking the quotient of the filtration on $\mc{F}^{\otimes j}$, we get a filtration on $\mr{Sym}^j \mc{F}$. One can give an explicit description of this filtration: $F^k\left(\mr{Sym}^j \mc{F}\right)$ is the image of the morphism
\[
     \sum_{\lambda \in \Lambda(k)} \mr{Sym}^{e_0(\lambda)} F^0(\mc{F}) \otimes \mr{Sym}^{e_1(\lambda)} F^1(\mc{F}) \otimes \cdots \otimes \mr{Sym}^{e_r(\lambda)} F^r(\mc{F}) \to \mr{Sym}^j \mc{F},
\]
where $\Lambda(k)$ is the set of monomials $\lambda = X_0^{e_0}X_1^{e_1} \cdots X_r^{e_r}$ of degree $j$ such that $\sum_{1 \leq i \leq r} i e_i = k$. In particular, we have
\[
    F^0\left(\mr{Sym}^j \mc{F}\right) = \mr{Sym}^j \mc{F}, \quad F^{jr}\left(\mr{Sym}^j \mc{F}\right) = \mr{Sym}^{j} F^r(\mc{F}).
\]
The graded term $\mr{gr}^k\left(\mr{Sym}^j \mc{F}\right)$ is naturally isomorphic to
\[
     \bigoplus_{\lambda \in \Lambda(k)} \mr{Sym}^{e_0(\lambda)}(\mr{gr}^0(\mc{F})) \otimes \mr{Sym}^{e_1(\lambda)}(\mr{gr}^1(\mc{F})) \otimes \cdots \otimes \mr{Sym}^{e_r(\lambda)}(\mr{gr}^r(\mc{F})).
\]
\begin{remark}
\label{omegatoHfiltration}
We will need, for instance, to consider a filtration on $\mb{S}^{\ul{k}}(H_\sigma)$,
where $\ul{k} = (k_1 \geq k_2 \geq \cdots \geq k_{n-1}) \in \Z^{n-1}_{\geq 0}$ and $H = H^1_\mr{dR}(A/S)$, for $A/S$ an abelian scheme with $\Oe$-action. To construct such filtration, first, we consider the Koszul filtrations $K^\bullet(\wedge^j H_\sigma)$, $j = 1, \dots, n-1$, induced by the Hodge filtration \ref{hfiltsigma2}. Taking tensor products we obtain a filtration on $\otimes_j \wedge^{l_j}(H_\sigma)$, $\ul l$ conjugate to $\ul k$. Similarly we construct filtrations on $H_\sigma^{\lvert \ul k \rvert}$ and $\otimes_i \mr{Sym}^{k_i}(H_\sigma)$. Finally we notice that the construction of $\mb{S}^{\ul k}(H_\sigma)$ is compatible with these filtrations.\\
When constructing generalised theta operators we will consider more complicated filtrations, but they will be built following the same ideas.
\end{remark}
\subsection{Generalities on connections}
We fix for the rest of this section the following data, a special case of \cite[1]{katzdiff}:
\[
    X \overset{f}{\longrightarrow} S \overset{g}{\longrightarrow} B
\]
where $f$ is smooth and proper and $g$ is smooth. In particular, the sheaves $H_\mr{dR}^i(X/S)$ are finite locally free. We will mainly be interested in the case $X = A$ where $\ul{A}$ is the universal object over $S=\mc{S}_{K, B}$, for some $B \in \ul{\mr{Sch}}_{\mc{O}_{E,(p)}}$.
\begin{definition}[Connection]
Let $\mc{F}$ be a quasi-coherent sheaf on $S$. A \emph{$B$-linear connection} on $\mc{F}$ is a morphism
\[
    \nabla \colon \mc{F} \longrightarrow \mc{F} \otimes_{\mc{O}_S} \Omega^1_{S/B}
\]
of abelian sheaves satisfying the \emph{Leibniz rule}
\[
    \nabla(h s) = h\nabla(s) + s \otimes d_{S/B}(h),
\]
where $h$ and $s$ are sections of $\mc{O}_S$ and $\mc{F}$, respectively, over some open $U \subset S$ and $d_{S/B} \colon \mc{O}_S \to \Omega^1_{S/B}$ is the exterior differential.
\end{definition}
We can derive from $\nabla$, for $(\mc{F}, \nabla)$ a sheaf with connection, a family of morphisms of abelian sheaves
\[
    \nabla_i \colon \mc{F} \otimes \Omega^i_{S/B} \longrightarrow \mc{F} \otimes \Omega^{i+1}_{S/B}
\]
defining it on sections $s$ and $\omega$ of $\mc{F}$ and $\Omega^i_{S/B}$, respectively, over some open $U \subset S$, via a generalised Leibniz rule
\[
    s \otimes \omega \longmapsto \nabla(s) \wedge \omega + s \otimes d\omega
\]
and then extending by linearity. The \emph{curvature} of the connection $\nabla$ is defined as $K(\nabla) \coloneqq \nabla_1 \circ \nabla \colon \mc{F} \to \mc{F} \otimes \Omega^2_{S/B}$, which is an $\mc{O}_S$-linear morphism. A connection is said to be \emph{integrable} if $K(\nabla) = 0$.
\par Given two quasi-coherent sheaves with connections $(\mc{F}, \nabla_{\mc{F}}), (\mc{G}, \nabla_{\mc{G}})$ we can use the Leibniz rule to define the tensor product connection $(\mc{F} \otimes \mc{G}, \nabla_{\mc{F} \otimes \mc{G}})$, obtained by setting on local sections
\begin{align*}
    \nabla_{\mc{F} \otimes \mc{G}} \colon \mc{F} \otimes_{\mc{O}_S} \mc{G} &\longrightarrow \mc{F} \otimes \mc{G} \otimes \Omega^1_{S/B}, \\
    f \otimes g & \longmapsto \nabla(f) \otimes g + f \otimes \nabla(g)
\end{align*}
and then extending by linearity, where, of course, we are identifying $\mc{F} \otimes \Omega^1 \otimes \mc{G} \cong \mc{F} \otimes \mc{G} \otimes \Omega^1$. In particular, from $(\mc{F}, \nabla)$ we can define a natural connection $\nabla \colon \mc{F}^{\otimes j} \to \mc{F}^{\otimes j} \otimes \Omega^1_{S/B}$, for $j \geq 0$, still denoted by $\nabla$. From this, taking quotients, we get well-defined connections on  $\wedge^j \mc{F}$ and $\mr{Sym}^j \mc{F}$. For $s_1, s_2, \dots, s_j \in \mc{F}(U)$, over $U \subset S$ some open, we have
\begin{align*}
    \nabla \colon s_1 \wedge s_2 \wedge \cdots \wedge s_j &\longmapsto \sum_{k=1}^j s_1 \wedge \cdots \wedge \nabla(s_k) \wedge \cdots \wedge s_j,\\
    \nabla \colon s_1 s_2 \cdots s_j &\longmapsto \sum_{k=1}^j s_1 \cdots \nabla(s_k) \cdots s_j.
\end{align*}
Consider in particular the case where $\mc{F}$ is finite locally free and it admits a one step-filtration by a finite locally free subsheaf $\mc{F}'$, with finite locally free quotient $\mc{F}/\mc{F}'$. Then we can take the Koszul filtration on $\wedge^i \mc{F}$ and we immediately see, working with local sections, that the construction we just gave entails the following.
\begin{lemma}[Koszul transversality]
\label{koszultransv}
The connection $\nabla$ respects the Koszul filtration up to a shift by one, that is, we have, for $i \geq 1$:
\[
    \nabla(K^i(\wedge^j \mc{F})) \subseteq K^{i-1}(\wedge^j \mc{F}) \otimes \Omega^1_{S/B}.
\]
More generally, if $(\mc{F}, F^\bullet(\mc{F}))$ is a finite locally free sheaf finitely filtered by finite locally free sheaves with finite locally free quotients, we have
\begin{align*}
    \nabla(F^i(\mc{F}^{\otimes j})) &\subseteq F^{i-1}(\mc{F}^{\otimes j}) \otimes \Omega^1_{S/B},\\
    \nabla(F^i(\mr{Sym}^j\mc{F})) &\subseteq (\mr{Sym}^j F^{i-1}(\mc{F})) \otimes \Omega^1_{S/B},
\end{align*}
for $i \geq 1$.
\end{lemma}
\subsection{The Gauss-Manin connection}
Consider the {first exact sequence} of $X/S/B$ which, for $f\colon X \to S$ smooth, is
\begin{equation}
\label{cotgcpxGM}
    0 \longrightarrow f^\ast \Omega^1_{S/B} \longrightarrow \Omega^1_{X/B} \longrightarrow \Omega^1_{X/S} \longrightarrow 0.    
\end{equation}
Then we can take the corresponding Koszul filtration of $\Omega^j_{X/B}$, for each $j=0, 1, \dots, \dim(X/B)$, whose graded pieces are
\[
    \mr{gr}^i (\Omega^j_{X/B}) = \Omega^{j-i}_{X/S} \otimes_{\mc{O}_X} f^\ast \Omega^i_{S/B}.
\]
One can check that $K^i(\Omega^\bullet_{X/B})$ is a sub-complex of $\Omega^\bullet_{X/B}$. In particular, the finite filtration $K^\bullet(\Omega^\bullet_{X/B})$ gives a convergent spectral sequence whose first page is 
\begin{align*}
    E_1^{i, j} &= R^{i+j}f_\ast(\mr{gr}^i(\Omega^\bullet_{X/B})) = R^{i+j}f_\ast( \Omega^{\bullet-i}_{X/S}  \otimes_{\mc{O}_X} f^\ast \Omega^i_{S/B})=\\
    & = R^{j}f_\ast(\Omega^{\bullet}_{X/S}  \otimes_{\mc{O}_X} f^\ast \Omega^i_{S/B}) \cong R^{j}f_\ast(\Omega^{\bullet}_{X/S}) \otimes_{\mc{O}_S} \Omega^i_{S/B}\\
    & = H^j_\mr{dR}({X/S})  \otimes_{\mc{O}_S} \Omega^i_{S/B}.
\end{align*}
Consider $\nabla = \nabla^j = d_1^{0, j}$, the differentials on the first page of $E_\bullet$. The multiplicative structure of $K^\bullet(\Omega^\bullet_{X/B})$ induces a multiplicative structure on $E_\bullet$, in the sense of \cite[p.~203]{katzoda} or \cite[5.4.8]{weibel}, given by cup products on $\Omega^\bullet_{S/B}, H^\bullet_\mr{dR}(X/S)$. From the definition of this multiplicative structure we deduce the following.
\begin{proposition}
The map $\nabla^j$ gives an integrable connection on ${H}^j_\mr{dR}(X/S)$.
\end{proposition}
For more details see {\cite[2]{katzoda}}.
\begin{definition}[Gauss-Manin connection]
We call the natural connection $\nabla$ on $H^i_\mr{dR}(X/S)$ the \emph{Gauss-Manin connection (relative to $X/S/B$)}.
\end{definition}
Since the GM connection and the ones derived from it are essentially the only connections we will discuss, we reserve the notation $\nabla$ for them, from now on.
\begin{proposition}[Griffiths transversality]
\label{propgrifftransv}
Let $F^\bullet = F^\bullet(H^\bullet_\mr{dR}(X/S))$ denote the Hodge filtration. The GM connection respects this filtration up to a shift by 1, that is, we have
\[
    \nabla(F^i(H^j_\mr{dR}(X/S))) \subseteq  F^{i-1}(H^j_\mr{dR}(X/S)) \otimes \Omega^1_{X/S}.
\]
In particular, the GM connection induces a mapping between graded terms.
\end{proposition}
For more details see \cite[1.4.1]{katzdiff}. Notice that for $X=A$ an abelian scheme this essentially reduces to \ref{koszultransv}, thanks to \ref{derhamabelian}, point 2.
\subsubsection{Base change}
Let $f' \colon X' \to S$ be a proper, smooth morphism and let $\phi \colon X \to X'$ be a morphism of $S$-schemes. The construction of the GM connection is functorial in the following sense.
\begin{lemma}
\label{lemGMpullback}
With notation as above, we have the natural commutative diagram
\[
\begin{tikzcd}
    & H^i_\mr{dR}(X'/S) \arrow[d, "\phi^\ast"] \arrow[r, "\nabla"] & H^i_\mr{dR}(X'/S) \otimes \Omega^1_{S/B} \arrow[d, "\phi^\ast \otimes \mr{id}"]\\
    & H^i_\mr{dR}(X/S) \arrow[r, "\nabla"] & H^i_\mr{dR}(X/S) \otimes \Omega^1_{S/B}.
\end{tikzcd}
\]
\end{lemma}
\begin{proof}
The morphism $\phi$ induces a natural commutative diagram whose rows are \ref{cotgcpxGM} for $X$ and $X'$:
\[
    \begin{tikzcd}
        & 0 \arrow[r] & f^\ast \Omega^1_{S/B} \arrow[r] \arrow[equal]{d} & \phi^\ast \Omega^1_{X'/B} \arrow[r] \arrow[d] & \phi^\ast \Omega^1_{X'/S} \arrow[r] \arrow[d] &0 \\
        & 0 \arrow[r] & f^\ast \Omega^1_{S/B} \arrow[r] & \Omega^1_{X/B} \arrow[r] & \Omega^1_{X/S} \arrow[r] &0.
    \end{tikzcd}
\]
This in turn induces a natural morphism $\phi^\ast K^\bullet(\Omega^\bullet_{X'/B}) \to K^\bullet(\Omega^\bullet_{X/B})$, whence a morphism of spectral sequences $\phi^\ast \colon E_\bullet(X'/S/B) \to E_\bullet(X/S/B)$, which gives us what we want.
\end{proof}
\par Let now $T$ be a smooth scheme over $B$, with $u \colon T \to S$ a morphism of $B$-schemes. Then we have a natural base change morphism
\[
    u^\ast H^i_\mr{dR}(X/S) \longrightarrow H^i_\mr{dR}(X_T/T)
\]
which is an isomorphism for every $u$. 

The construction of $\nabla$ immediately implies the following.
\begin{lemma}
\label{lemGMbasechange}
With $X, S, B$ and $u \colon T \to S$ as above, we have the commutative diagram
\[
    \begin{tikzcd}
        & u^\ast H^i_\mr{dR}(X/S) \arrow[d, equal] \arrow[dr, "u^\ast({\nabla})"] \\
        & H^i_\mr{dR}(X_T/T) \arrow[r, "\nabla"] & H^i_\mr{dR}(X_T/T) \otimes \Omega^1_{T/B}.
    \end{tikzcd}
\]
where the diagonal map is the pullback connection induced by base change $u^\ast H^i_\mr{dR}(X/S) \cong H^i(X_T/T)$ and the natural map $d u \colon u^\ast \Omega^1_{S/B} \to \Omega^1_{T/B}$.
\end{lemma}

\begin{remark}
Let us consider for a moment the case $X=A$, for $\ul{A} \in \mc{S}_{K, B}(S)$, $B \in \ul{\mr{Sch}}_{\mc{O}_{E,(p)}}$. Then, in particular, we have an action of $\mc{O}_E$ on $H^i_\mr{dR}(A/S)$. Proposition \ref{lemGMpullback} implies that the GM connection respects the $\Oe$ action, that is, for $H=H^1_\mr{dR}(A/S)$, we have restrictions
\begin{align*}
    \nabla &\colon H_\sigma \longrightarrow H_\sigma \otimes_{\mc{O}_S} \Omega^1_{S/B},\\
    \nabla &\colon H_{\overline{\sigma}} \longrightarrow H_{\overline{\sigma}} \otimes_{\mc{O}_S} \Omega^1_{S/B}.
\end{align*}
Similarly, if $B$ is of characteristic $p$ we can consider the isogenies Frobenius and Verschiebung and take their pullback actions $F\colon H^{(p)} \to H$ and $V\colon H \to H^{(p)}$. Then \ref{lemGMpullback} tells us that $\nabla$ commutes with these actions and therefore it respects kernels and images of both. This will be relevant when defining generalised theta operators.
\end{remark}

\subsection{The Kodaira-Spencer map}
The short exact sequence \ref{cotgcpxGM} gives a natural element
\[
    \ul{\mr{KS}} \coloneqq \ul{\mr{KS}}_{X/S/B} \in \mr{Ext}_{\mc{O}_X}(\Omega^1_{X/S}, f^\ast \Omega^1_{S/B}) \cong H^1(X, \ul{\mr{Der}}(X/S) \otimes f^\ast \Omega^1_{S/B}).
\]
This is called the \emph{Kodaira-Spencer class}. From the Leray spectral sequence we obtain
\[
    H^1(X, \ul{\mr{Der}}(X/S) \otimes f^\ast \Omega^1_{S/B}) \longrightarrow H^0(S, R^1f_\ast(\ul{\mr{Der}}(X/S) \otimes f^\ast \Omega^1_{S/B})).
\]
We denote the image of $\ul{\mr{KS}}$ via this morphism by $\ul{\mr{KS}}$ again and consider
\begin{align*}
    \ul{\mr{KS}} &\in H^0(S, R^1f_\ast(\ul{\mr{Der}}(X/S) \otimes_{\mc{O}_X} f^\ast \Omega^1_{S/B})) \\
    &\cong H^0(S, R^1f_\ast(\ul{\mr{Der}}(X/S)) \otimes_{\mc{O}_S} \Omega^1_{S/B}) \\
    &\cong \hom_{S}(\ul{\mr{Der}}(S/B), R^1f_\ast(\ul{\mr{Der}}(X/S))).
\end{align*}
\begin{definition}
The morphism $\ul{\mr{KS}} \colon \ul{\mr{Der}}(S/B)\to R^1f_\ast(\ul{\mr{Der}}(X/S))$ is called the \emph{Kodaira-Spencer mapping}.
\end{definition}
\begin{proposition}[{\cite[1.4.1.7]{katzdiff}}]
\label{propKSGM}
Suppose that the Hodge-de Rham spectral sequence of $X/S$ degenerates on the first page. Then the morphism induced on the graded pieces from \ref{propgrifftransv} by the GM connection
\[
    \begin{tikzcd}
        & \mr{gr}_F^i H_\mr{dR}^{i+j}(X/S) \arrow[d, "\sim"] \arrow[r, "\nabla"] & \mr{gr}_F^{i-1} H_\mr{dR}^{i+j}(X/S) \otimes \Omega^1_{S/B} \arrow[d, "\sim"]\\
        & R^jf_\ast \Omega^i_{X/S} \arrow[r, "\nabla"] &R^{j+1}f_\ast \Omega^{i-1}_{X/S} \otimes \Omega^1_{S/B}
    \end{tikzcd}
\]
is given by the cup product with the KS section
\[
    \ul{\mr{KS}} \in H^0(S, R^1f_\ast(\ul{\mr{Der}}(X/S)) \otimes_{\mc{O}_S} \Omega^1_{S/B}).
\]
\end{proposition}
\par Let now $X=A$ be an abelian scheme over $S$. We are interested in the special case of \ref{propKSGM} for $A/S/B$. Let $H = H^1_\mr{dR}(A/S)$. We have, on the one hand, a morphism:
\[
    \ul{\omega}_{A/S} \subset H \overset{\nabla}{\longrightarrow} H \otimes \Omega^1_{S/B} \twoheadrightarrow \ul{\omega}_{A^\vee/S} \otimes_{\mc{O}_S} \Omega^1_{S/B}
\]
whence, equivalently, a morphism
\[
    \ul{\omega}_{A/S} \otimes_{\mc{O}_S} \ul{\omega}_{A^\vee/S} {\longrightarrow} \Omega^1_{S/B}.
\]
On the other hand, we know that $\Omega^1_{A/S} \cong f^\ast \ul{\omega}_{A/S}$, so that $\ul{\mr{Der}}(A/S) \cong f^\ast \ul{\omega}_{A/S}^\vee$ and $R^1f_\ast \ul{\mr{Der}}(A/S) \cong R^1f_\ast(\mc{O}_{A}) \otimes_{\mc{O}_S} \ul{\omega}_{A/S}^\vee \cong \ul{\omega}_{A^\vee/S}^\vee \otimes_{\mc{O}_S} \ul{\omega}_{A/S}^\vee$. By duality then the KS morphism is equivalent to
\[
    \ul{\mr{KS}} \colon \ul{\omega}_{A/S} \otimes_{\mc{O}_S} \ul{\omega}_{A^\vee/S} {\longrightarrow} \Omega^1_{S/B}
\]
Then, \ref{propKSGM} tells us that these two morphisms $\ul{\omega}_{A/S} \otimes_{\mc{O}_S} \ul{\omega}_{A^\vee/S} {\to} \Omega^1_{S/B}$ are one and the same.
\par Assume now that $A \to S =\mc{S}_{K,B}$ is the universal object, $B \to \spec \mc{O}_{E, (p)}$. Then, using the identification
\[
    \ul{\omega}_{A/S, \sigma} \otimes_{\mc{O}_S} \ul{\omega}_{A^\vee/S, \sigma} \cong \ul{\omega}_{A/S, \sigma} \otimes_{\mc{O}_S} \det \ul{\omega}_{A/S, \sigma} \otimes \delta_{A/S, \sigma}^{-1},
\]
we see that $\ul{\mr{KS}}$ induces a morphism
\[
    \ul{\mr{ks}} = \ul{\mr{ks}}_{A/S/B} \coloneqq \ul{\mr{KS}}_\sigma \colon \ul{\omega}_{A/S, \sigma} \otimes_{\mc{O}_S} \det \ul{\omega}_{A/S, \sigma} \otimes \delta_{A/S, \sigma}^{-1} {\longrightarrow} \Omega^1_{S/B}.
\]
Then we have the following.
\begin{proposition}[{\cite[2.3.5.2]{kwl}}]
\label{propksiso}
The map $\ul{\mr{ks}}$ is an isomorphism.
\end{proposition}
\begin{proof}
The proof in \cite[2.3.5.2]{kwl} is given for $B = \spec \mc{O}_{E, (p)}$ but the construction of $\ul{\mr{ks}}$ is compatible with base change, so the statement over a general base $B \to \spec \mc{O}_{E, (p)}$ follows. Notice that, in the notation of \cite[2.3.5.1]{kwl}, the injection $\ul{\omega}_{A/S, \sigma} \otimes_{\mc{O}_S} \ul{\omega}_{A^\vee/S, \sigma} \to \ul{\omega}_{A/S} \otimes_{\mc{O}_S} \ul{\omega}_{A^\vee/S}$ composed with the projection
\begin{align*}
    \ul{\omega}_{A/S} \otimes_{\mc{O}_S} \ul{\omega}_{A^\vee/S} \to \ul{\omega}_{A/S} \otimes_{\mc{O}_S} \ul{\omega}_{A^\vee/S}/\big(&\lambda^\ast(y)\otimes z - \lambda^\ast(z)\otimes y,\\
    & (\alpha \cdot x) \otimes y - x \otimes (\alpha \cdot y) \big)
\end{align*}
is an isomorphism, where $x \in \ul{\omega}_{A/S}(U), y, z \in \ul{\omega}_{A^\vee/S}(U), \alpha \in \Oe$ are sections over $U \subset S$ some open that varies.
\par Equivalently, from \cite[III.9]{faltchai} one can deduce a description of the KS map in terms of the universal vector extension. This and an argument using Grothendieck-Messing theory, as explained for instance in \cite[pp.~116-118]{gbtmontreal}, also gives a proof of the proposition: see \cite[Prop.~2.1.5]{bellaiche} for the case $n=3$, which can be easily generalised to our setting. See also the proof of \ref{propmoreonKScharp}.
\end{proof}

\section{Ekedahl-Oort stratification}
We refer to \cite{EOVW} and \cite{boxer} for the general theory of EO stratifications on PEL Shimura varieties and to \cite{wooding} for the unitary case in particular.
\par We will write $S_K$ for the geometric special fibre $\mc{S}_K \times \spec(\mb{F})$. By \ref{Levidecomp}, $\mathbf{G}_{\mb{F}} \cong \mr{GL}_n \times \mb{G}_m$ and we can take $T \subset B \subset \mathbf{G}_{\mb{F}}$ to be the diagonal torus and Borel of upper-triangular matrices, respectively, determined by this isomorphism. Let $W$ be the Weyl group of $\mathbf{G}_{\mb{F}}$ with respect to $T$, with set of simple roots $\Delta$ determined by $B$. Write $[s(1), s(2), \dots, s(g)]$ for the permutation of $g$ elements acting as $i \mapsto s(i)$. We can identify $W$  with
\[
    \{(w_\sigma, w_{\overline{\sigma}}) \in \mr{S}_n \times \mr{S}_n \mid w_0 w_{\overline{\sigma}} w_0 = w_\sigma \}
\]
where $w_0 = [n, n-1, \dots, 2, 1]$ is the longest element of $\mr{S}_n$, the symmetric group of order $n$, with length given with respect to its simple reflections $\{s_i = (i, i+1),  1 \leq i \leq n-1\}$. The projection $W \to \mr{S}_n$ onto the first factor induces an isomorphism of Coxeter groups, so that
\[
    \Delta = \{(s_i, w_0 s_i w_0) \in \mr{S}_n \times \mr{S}_n\}.
\]
We consider the parabolic subgroup $P \subset \mathbf{G}_{\mb{F}}$, of type $I \subseteq \Delta$, defined as the stabiliser of the flag $\ker F \subset \mb{D}(G)$, where $\mb{D}(G)$ is the (contravariant) Dieudonn\'e module of $G$, any principally quasi-polarised \btone\ over $\mb{F}$ with $\mc{O}_E$-action, height $2n$ and CM dimensions $d_\sigma = n-1, d_{\overline{\sigma}}=1$. Denote by $\ul{\mr{BT}}_{1, \mb{F}}^{(n-1,1)}$ the set of the isomorphism classes of all such \btone's. Write ${}^I W$ for the set of minimal length representatives of the quotient $W_I \backslash W$, where $W_I$ is the Weyl group of $P$. One can show that $P$ does not depend on the choice of $G \in \ul{\mr{BT}}_{1, \mb{F}}^{(n-1,1)}$ and in fact we have the following.
\begin{lemma}[{\cite[3.1.2, 3.4.1]{wooding}}]
The Levi factor of $P$ is the $H_\mathbb{F}$ from \ref{Levidecomp} and $I = \Delta \setminus \{s_{n-1}\}$.
\par Moreover, ${}^I W$ is given by the set of (inverse) shuffles
\[
    \{w \in \mr{S}_n \mid w^{-1}(1) < w^{-1}(2) < \cdots < w^{-1}(n-1)\}
\]
which has cardinality $n$. For $w \in {}^I W$ the length is given by 
\[
l(w) = \sum_{1 \leq i \leq n-1} w^{-1}(i) - i.
\]
The order induced on ${}^I W$ by the length function is total and it coincides with the Bruhat and Ekedahl-Oort orders.
\end{lemma}
We can describe the elements of ${}^I W$ explicitly. In fact, $w \in \mr{S}_n$ satisfies the condition that $w^{-1}(1) < w^{-1}(2) < \cdots < w^{-1}(n-1)$ if and only if it is of the form
\[
    w_{r} = \left( \begin{array}{cccccccc}
        1 & 2 & \cdots & r-1 & r & r+1 & \cdots & n \\
        1 & 2 & \cdots & r-1 & n & r & \cdots & n-1
    \end{array} \right).
\]
In particular, $w_{r}^{-1} = [1, 2, \dots, r-1, r+1, \dots, n, r]$, so $l(w_{r}) = n-r$. Moreover, $w_0 w_{r} w_0$ is a shuffle of the form $w^{-1}(2) < w^{-1}(3) < \cdots < w^{-1}(n)$. The total order $\leq$ on these elements is $w_{1} \geq w_{2} \geq \cdots \geq w_{n} = \mr{id}$.
\par For each $G \in \ul{\mr{BT}}_{1, \mb{F}}^{(n-1,1)}$ we also have a flag $D_\bullet$ in $\mb{D}(G)$ coming from the canonical filtration of $G$. Denote by $P_{G} \subset \mathbf{G}_{\mb{F}}$ the parabolic fixing this flag and write $w(G) = w(P, P_G) \in {}^IW$ for the relative position of the two parabolic subgroups $P$ and $P_G$. Then we have the following.
\begin{theorem}[{\cite[Thm.~6.7]{gsas}, \cite[4.3.2, 4.2.8-10-18]{boxer}}]
\label{thmEO}
We have the following:
\begin{enumerate}
    \item There is a bijection
        \begin{align*}
            \ul{\mr{BT}}_{1, \mb{F}}^{(n-1,1)} &\longrightarrow {}^I W,\\
            G &\longmapsto w(G). 
        \end{align*}
        Moreover, $w(G)$ is determined by the permutation associated to the canonical filtration of $G$ in \cite[4.2.4]{boxer} together with the CM dimensions of the graded pieces $\mr{gr}_i(G)$ of that filtration.
    \item There is a decomposition of $S_K$ into a disjoint union of reduced locally closed subschemes as
        \[
            S_K = \bigsqcup_{w \in {}^I W} S_{K, w}
        \]
        such that:
        \begin{enumerate}
            \item For $k/\mb{F}$ algebraically closed, and $x \in S_{K}(k)$, then $x \in S_{K, w}(k)$ if and only if $w = w(A_x[p])$.
            \item For each $w \in {}^I W$ the subscheme $S_{K, w}$ is non-empty, equidimensional of dimension $l(w)$, smooth and quasi-affine.
            \item We have the closure relation
                \[
                 \overline{S}_{K, w_r} = {S}_{K, w_r} \sqcup \overline{S}_{K, w_{r+1}} = {S}_{K, w_r} \sqcup {S}_{K, w_{r+1}} \sqcup \cdots \sqcup S_{K, w_n},
                \]
                for all $1 \leq r \leq n-1$.
        \end{enumerate}
\end{enumerate}
\end{theorem}
In particular, we must have $S_K = \overline{S}_{K, w_1} = S_{K, w_1} \sqcup \overline{S}_{K, w_2}$, so that $S_{K, w_1}$ is a dense open, called the \emph{ordinary locus}, of dimension $n-1$, which is the dimension of the smooth quasi-projective $\mb{F}$-scheme $S_K$. We will sometimes denote the ordinary locus $S_{K, w_1}$ by $S_K^\mu$. At the other end we have the zero-dimensional closed stratum $S_{K, w_n}$, which we call the \emph{core locus}. We will also call $S_{K, w_2}$ the \emph{almost ordinary locus} and $S_{K, w_{n-1}}$ the \emph{almost core locus}.

\subsection{The standard Dieudonn\'e modules}

Following \cite[4.5]{wooding}, for given a point $\ul{A} \in S_{K, w_r}(k)$, where $k/\mb{F}$ is an algebraically closed extension and $1 \leq r \leq n$, we can describe the structure of the Dieudonn\'e module $D = \mb{D}(A[p])$ using standard objects, defined in \cite[4.9]{gsas}. In fact, from \cite[Thm.~4.7]{gsas} we deduce the following.

\begin{proposition}
\label{propEOunit}
Let $D = D_\sigma \oplus D_{\overline{\sigma}}$ be the CM decomposition of $D$. There are $k$-bases $\{e_i\}_{1 \leq i \leq n}$, $\{f_i\}_{1 \leq i \leq n}$ of $D_\sigma$, $D_{\overline{\sigma}}$, respectively, such that $\left<e_i, f_j\right> = \delta_{ij}$ and the semilinear Frobenius $F \colon D \to D$ and Verschiebung $V \colon D \to D$ are described as follows:
    \begin{align} 
        & F(e_i) = \left \{ \begin{array}{cc}
             0 & \text{ if } i\neq r, \\
             e_1 & \text{ if } i=r,
        \end{array} \right.
        & F(f_i) = \left \{ \begin{array}{cc}
             f_{i+1} & \text{ if } 1 \leq i \leq r-1, \\
             0 & \text{ if } i=r, \\
             f_{i} & \text{ if } r+1 \leq i \leq n,
        \end{array} \right.\\
        & V(e_i) = \left \{ \begin{array}{cc}
             0 & \text{ if } i=1, \\
             e_{i-1} & \text{ if } 2 \leq i \leq r, \\
             e_{i} & \text{ if } r+1 \leq i  \leq n,
        \end{array} \right.
        & V(f_i) = \left \{ \begin{array}{cc}
             f_r & \text{ if } i=1, \\
             0 & \text{ otherwise}.
        \end{array} \right.
    \end{align}
    In particular, we have
    \begin{align*}
        D[F] = \ker F &= \mr{Span}_{k} \left<e_1, \dots, e_{r-1}, f_r, e_{r+1}, \dots, e_n\right>,\\
        D[V] = \ker V &= \mr{Span}_{k} \left<e_1, f_2,  \dots, f_n\right>
    \end{align*}
    and:
    \begin{enumerate}
        \item if $r=1$, the $p$-rank of $A[p]$ is $n$ and that of $A[p]_{\overline{\sigma}}$ is $n-1$,
        \item if $r>1$, the $p$-rank of $A[p]$ and $A[p]_{\overline{\sigma}}$ is $n-r$.
    \end{enumerate}
    In particular, the EO stratification coincides with the $p$-rank stratification for the $\overline{\sigma}$-component of the $p$-torsion.
\end{proposition}

\subsection{Partial Hasse invariants}
\label{parHinv}
We are interested in the action of $V$ on $\ul{\omega}_{A_s/k(s)} \cong D[F]$ for $\ul{A}_s \in S_{K, w_r}(k)$, $D=\mb{D}(A_s[p])$, $k/\mb{F}$ algebraically closed and $1 \leq r \leq n-1$. In particular, from \ref{propEOunit}, one can see that the rank function
\begin{align*}
    \lvert \overline{S}_{K, w_r} \rvert & \longrightarrow \Z_{\geq 0}, \\
    s &\longmapsto \mr{rk}(\ker \, V^{r-1} \colon \ul{\omega}_{A_s/k(s), \sigma} \to \ul{\omega}_{A_s/k(s), \sigma}^{(p^{r-1})}) 
\end{align*}
is constant and equal to $r-1$. Therefore, the sheaf
\[
    \ul{\omega}_{0, r} \coloneqq \ker( V^{r-1} \colon \ul{\omega}_{A/\overline{S}_{K, w_r}, \sigma} \to \ul{\omega}_{A/\overline{S}_{K, w_r}, \sigma}^{(p^{r-1})})
\]
is locally free of rank $r-1$ on $\overline{S}_{K, w_r}$. Moreover, the quotient sheaf $\ul{\omega}_{\mu, r} \coloneqq \ul{\omega}_{\sigma}/\ul{\omega}_{0, r}$ is also finite locally free, of rank $n-r$, over $\overline{S}_{K, w_r}$. Over $S_{K, w_r}$, $\ul{\omega}_{0, r}$ and $\ul{\omega}_{\mu, r}$ correspond to the local-local and multiplicative part of the canonical filtration of the $p$-torsion of $A$, hence the names. In particular, the quotient morphism $V \colon \ul{\omega}_{\mu, r} \to \ul{\omega}_{\mu, r}^{(p)}$ is an isomorphism over the dense open ${S}_{K, w_r} \subseteq \overline{S}_{K, w_r}$ and has a kernel of rank $1$ over the boundary $\partial \overline{S}_{K, w_r} = \overline{S}_{K, w_r} \setminus {S}_{K, w_r}$. Therefore, if we consider the corresponding section $A_r = \det V \in H^0(\overline{S}_{K, w_r}, \det \ul{\omega}_{\mu, r}^{p-1})$, we have the following.
\begin{lemma}
The section $A_r$ is nowhere vanishing on ${S}_{K, w_r}$ and zero everywhere on the boundary $\partial \overline{S}_{K, w_r}$. Moreover, the construction of $A_r$ is independent of the level $K$.
\end{lemma}
We call the $A_r$'s \emph{(generalised) partial Hasse invariants}. They are factors of the generalised Hasse invariants of \cite{boxer}. 

\subsection{Smoothness of closures in the EO stratification}
We establish a geometric property of the KS morphism on lower EO strata. Consider $2 \leq r \leq n-1$. We will show in particular that $\overline{S}_{K, w_r}$ is smooth. We will write
\[
    \nabla_r \colon H^1_\mr{dR}(A/\overline{S}_{K, w_r}) \longrightarrow H^1_\mr{dR}(A/\overline{S}_{K, w_r}) \otimes \Omega^1_{\overline{S}_{K, w_r}/\mb{F}}
\]
to denote the Gauss-Manin connection relative to $A/\overline{S}_{K, w_r}/\mb{F}$.
We will use the same notation for the restriction of $\nabla_r$ to $S_{K, w_r} \subseteq \overline{S}_{K, w_r}$. Recall that $\nabla_r$ is related to to $\nabla \colon H^1_\mr{dR}(A/S_K) \to H^1_\mr{dR}(A/S_K) \otimes \Omega^1_{S_K/\mb{F}}$ via base change, see \ref{lemGMbasechange}. We will write $u \colon S_{K, w_r} \to S_{K}$ and $\overline{u} \colon \overline{S}_{K, w_r} \to S_K$ for the natural immersions. To ease notations, we set $\tilde{\delta} = \det \ul{\omega}_{A/S, \sigma} \otimes \delta_{\sigma}^{-1}$, with $S$ depending on the context. We also write $\mc{C}_{\overline{S}_{K, w_r}/S_K}$ to denote the {conormal sheaf} of $\overline{S}_{K, w_r}$ in $S_K$.
\begin{proposition}
\label{propmoreonKScharp}
We have a commutative diagram of short exact sequences
\begin{equation}
\label{lowerEOKSiso}
    \begin{tikzcd}
        &0 \arrow[r] & \ul{\omega}_{0, r} \otimes \tilde{\delta}\arrow[r] \arrow[d] &\ul{\omega}_{A/\overline{S}_{K, w_r}, \sigma} \otimes \tilde{\delta} \arrow[r] \arrow[d, "\overline{u}^\ast(\ul{\mr{ks}})"] &\ul{\omega}_{\mu, r} \otimes \tilde{\delta} \arrow[r] \arrow[d, "\ul{\mr{ks}}_{\mu, r}"] &0\\
        &0 \arrow[r] &\mc{C}_{\overline{S}_{K, w_r}/S_K} \arrow[r] &\overline{u}^\ast(\Omega^1_{S_{K}/\mb{F}}) \arrow[r] & \Omega^1_{\overline{S}_{K, w_r}/\mb{F}} \arrow[r] &0,
    \end{tikzcd}
\end{equation}
whose vertical arrows are isomorphisms. In particular, $\overline{S}_{K, w_r}$ is smooth and
\[
    \nabla_r(\ul{\omega}_0) \subseteq \ul{\omega}_0 \otimes \Omega^1_{\overline{S}_{K, w_r}/\mb{F}}.
\]
\end{proposition}
\begin{proof}
We prove this via Grothendieck-Messing theory. For what we intend to prove the level $K$ is immaterial, so we will drop it from the notation. We will write $S_r$ (resp.\ $\overline{S}_r$)  instead of $S_{K, w_r}$ (resp.\ $\overline{S}_{K, w_r}$). \\
First of all, notice that over $S_r$ we have $\nabla_r(\ul{\omega}_0) \subseteq \ul{\omega}_0 \otimes \Omega^1_{S_r/\mb{F}}$, by \ref{lemGMpullback}. By \ref{lemGMbasechange} and \ref{propKSGM}, this means that over $S_r$ the composition
\[
    \ul{\omega}_0 \otimes \tilde{\delta} \longrightarrow \ul{\omega}_{A/S_r, \sigma} \otimes \tilde{\delta} \overset{\ul{\mr{ks}}}{\longrightarrow} u^\ast(\Omega^1_{S/\mb{F}}) \longrightarrow \Omega^1_{S_r/\mb{F}}
\]
is zero. By density of $S_r \subseteq \overline{S}_r$, we deduce the existence of the commutative diagram \ref{lowerEOKSiso} and the morphism $\ul{\mr{ks}}_{\mu}$. We now show that its vertical arrows are isomorphism and that the second row is a short exact sequence by working at points.\\
Consider a point $s \in \overline{S}_r \subseteq S$ and write $k = k(s)$. 
We want to describe the Zariski tangent spaces 
\[
    T_{\overline{S}_r, s} = T_{\overline{S}_r/\mathbb{F}, s}, \quad T_{S, s} = T_{S/\mathbb{F}, s}.
\]
These are $k$-vector spaces whose elements $v$ correspond to lifts of $s \colon \spec k \to \overline{S}_r$ to $v \colon \spec k[\epsilon] \to \overline{S}_r$ (resp.\ $S$), where $k[\epsilon] \coloneqq k[T]/(T^2)$, that is, morphisms fitting in the following commutative diagram
\[
    \begin{tikzcd}
        & \spec k \arrow[r] \arrow[rr, bend left = 20, "s"] & \spec k[\epsilon] \arrow[r, "v"] & \overline{S}_r \text{ (resp. } S\text{)},
    \end{tikzcd}
\]
where the morphism $\spec k \to \spec k[\epsilon]$ is the $\spec$ over $k$ of $k[\epsilon] \to k, \epsilon \mapsto 0$. By definition of $S$, resp. $\overline{S}_r$, any $v \in T_{S, s}$, resp. $v \in T_{\overline{S}_r, s}$, corresponds to an abelian scheme with PEL structure $\ul{A}_v \in S(k[\epsilon])$, resp. $\ul{A}_v \in \overline{S}_r(k[\epsilon])$, lifting the object $\ul{A} \in \overline{S}_r(k)$ corresponding to $s$. By this we mean that we have a cartesian diagram
\[
    \begin{tikzcd}
        & A \arrow[r] \arrow[d]\arrow[dr, phantom, "\usebox\pullback" , very near start, color=black] & A_v \arrow[d]\\
        & \spec k \arrow[r] & \spec k[\epsilon].
    \end{tikzcd}
\]
Grothendieck-Messing theory tells us that the category of such lifts $\ul{A}_v$ is equivalent to the category of lifts of the $k$-subspace $\ul{\omega} = \ul{\omega}_{A/k}$ of $H = H^1_\mr{dR}(A/k)$ to sub-$k[\epsilon]$-modules $\Tilde{\ul{\omega}}$ of $H[\epsilon] = H \otimes_k k[\epsilon]$, the trivial lift of $H$, such that:
\begin{enumerate}
    \item $\Tilde{\ul{\omega}}$ is a free direct summand of $H[\epsilon]$ of rank $n$.
    \item $\Tilde{\ul{\omega}}$ is $\mc{O}_E$-stable (of type necessarily $(n-1, 1)$).
    \item $\Tilde{\ul{\omega}}$ is maximal isotropic for the perfect alternating pairing on $H[\epsilon]$ induced by the de Rham pairing $\left<\cdot, \cdot\right>^\lambda_\mr{dR}$ on $H$.
\end{enumerate}
Moreover, for $v \in T_{S, s}$, we have that:
\begin{enumerate}
    \setcounter{enumi}{3}
    \item  $v \in T_{\overline{S}_r, s}$ if and only if the lift $\ul{A}_v$ corresponds to $\Tilde{\ul{\omega}}$ such that $\Tilde{\ul{\omega}}_\sigma[V^{r-1}]$ is free of rank $r-1$.
\end{enumerate}
Such a $\Tilde{\ul{\omega}}$ corresponds to a $k$-linear morphism $h \colon \ul{\omega} \to H/\ul{\omega} \cong \ul{\omega}_{A^\vee/k}^\vee \underset{\lambda^\ast}{\cong} \ul{\omega}^\vee$ subject to the conditions:
\begin{enumerate}
    \item The pairing $\left<\cdot, h(\cdot)\right>^\lambda_\mr{dR}$ on $\ul{\omega}$ is symmetric.
    \item $h$ is $\mc{O}_E$-linear. 
\end{enumerate}
The first condition corresponds to $\Tilde{\ul{\omega}}$ being maximal isotropic, the second to the fact that $\Tilde{\ul{\omega}}$ is $\mc{O}_E$-stable. Furthermore, the extra condition that $\Tilde{\ul{\omega}}_\sigma[V^{r-1}]$ is free of rank $r-1$ corresponds to:
\begin{enumerate}
    \setcounter{enumi}{2}
    \item $h$ is identically $0$ on $\ul{\omega}_\sigma[V^{r-1}] \subseteq \ul{\omega}_\sigma$.
\end{enumerate}
Notice that the first two conditions together imply that $h$ is determined by
\[
h_\sigma \colon \ul{\omega}_\sigma \longrightarrow (H/\ul{\omega})_\sigma \cong \ul{\omega}_{A^\vee/k, \sigma}^\vee \underset{\lambda^\ast}{\cong} \ul{\omega}_{\overline{\sigma}}^\vee.
\]
The correspondence between $\tilde{\ul{\omega}}$ and $h$ is given as follows. For any $\Tilde{\ul{\omega}}$ subject to the conditions above and $x \in \ul{\omega}$ we have some $x + \epsilon z \in \Tilde{\ul{\omega}}$, $z \in H $. Then $h(x) = \overline{z} \in \ul{\omega}^\vee$. This is well defined because $\tilde{\ul{\omega}}$ is isotropic. Conversely, given $h$, if for any $x \in \ul{\omega}$ we write $h'(x) \in H$ for some lift of $h(x)$, then $\tilde{\ul{\omega}}$ is the submodule generated by $x + \epsilon(y + h'(x))$, for all the $x, y \in \ul{\omega}$. Therefore, the deformations $T_{S, s}$ of $A$ are in correspondence with $\hom_k(\ul{\omega}_\sigma, \ul{\omega}_{\overline{\sigma}}^\vee) \cong \ul{\omega}_\sigma^\vee \otimes_k \ul{\omega}_{\overline{\sigma}}^\vee$, and $T_{\overline{S}_r, s} \subset T_{S, s}$ corresponds to 
\[
    \hom_k(\ul{\omega}_\mu, \omega_{\overline{\sigma}}^\vee) = \{h \colon \ul{\omega}_\sigma \to \ul{\omega}_{\overline{\sigma}}^\vee \mid h|_{\ul{\omega}_\sigma[V^{r-1}]} = 0\} \subset \hom_k(\ul{\omega}_\sigma, \ul{\omega}_{\overline{\sigma}}^\vee).
\]
In fact, as discussed in \cite[p.~40]{bellaiche}, this correspondence is actually given by isomorphisms of $k$-vector spaces and, as we noted in \ref{propksiso}, the isomorphism $T_{S, s} \to \ul{\omega}_\sigma^\vee \otimes_k \ul{\omega}_{\overline{\sigma}}^\vee$ is the dual of the fibre at $s$ of the KS morphism. Hence, we have a commutative diagram
\[
    \begin{tikzcd}
        &0 \arrow[r] & T_{\overline{S}_r, s} \arrow[r] \arrow[d] & T_{S, s} \arrow[r] \arrow[d] &\mc{N}_{\overline{S}_r/S, s} \arrow[d]\\
        &0 \arrow[r] & \ul{\omega}_\mu^\vee \otimes \ul{\omega}_{\overline{\sigma}}^\vee \arrow[r] & \ul{\omega}_\sigma^\vee \otimes \ul{\omega}_{\overline{\sigma}}^\vee \arrow[r] &\ul{\omega}_0^\vee \otimes \ul{\omega}_{\overline{\sigma}}^\vee \arrow[r] &0\\
    \end{tikzcd}
\]
where the vertical arrows are isomorphisms and $\mc{N}_{\overline{S}_r/S, s}$ denotes the normal sheaf at $s$ of $\overline{S}_r$ in $S$. This proves at once that $\overline{S}_r$ is smooth and by duality that the diagram of exact sequences \ref{lowerEOKSiso} is exact and its vertical arrows are isomorphisms. 
\par Finally, using the functoriality of $\nabla_r$ from \ref{lemGMpullback} with $\phi = V^{r-1}$ we deduce that $\nabla_r(\ul{\omega}_0) \subseteq \ul{\omega}_0 \otimes \Omega^1_{\overline{S}_r/\mb{F}}$.
\end{proof}

\begin{remark}
The smoothness part of \ref{propmoreonKScharp} generalises point 3 of \cite[Thm.~IV.7]{koblitz}. It should be possible to give an alternative proof using arguments similar to those of \cite[IV.9]{koblitz}, which are quite explicit.
\\ Notice that the smoothness of closed EO strata in the $p$-split case stands in contrast with the $p$-inert case: for instance, when $n=3, 4$, one can show that the 1-dimensional closed strata are (unions of) chains of Fermat curves of degree $p+1$ intersecting at superspecial points, with every superspecial point being at the intersection of many irreducible components. See \cite[Ex.~G]{vollwedh} or \cite{deshagor17} for more details.
\end{remark}

\section{Generalised splittings and operators}

\subsection{Unit-root splitting}
We present here the classic unit-root splitting of the Hodge filtration \ref{hfiltsigma1} over the ordinary locus $S^\mu_K$ and then we show how one can obtain a similar splitting over lower EO strata. In our discussion we will make implicit use of \ref{propEOunit} and of the following technical lemma, inspired by \cite[3.3.1]{EFGMM}.
\begin{lemma}
\label{lemadjext}
Let $\overline{S}$ be a scheme. Consider $\mc{F}, \mc{G}$ and $\mc{H}$ finite locally free sheaves on $\overline{S}$, with $\mc{F} \subseteq \mc{G}$ such that $\mc{G}/\mc{F}$ is also finite locally free with constant rank. Suppose that we have a morphism
\[
    \phi \colon \mc{G} \longrightarrow \mc{H}
\]
of $\mc{O}_{\overline{S}}$-sheaves such that the restriction $\phi|_\mc{F}$ of $\phi$ to $\mc{F}$ is an isomorphism on a dense open $S \subseteq \overline{S}$. Write $d = \det \phi|_\mc{F} \in H^0(\overline{S}, \det \mc{H} \otimes \det \mc{F}^{-1})$ for the determinant of that restriction. Then the morphism
\[
    \mc{G} \overset{\phi}{\longrightarrow} \mc{H} \overset{\phi^{-1}}{\longrightarrow} \mc{F} \overset{d}{\longrightarrow} \mc{F} \otimes \det \mc{H} \otimes \det \mc{F}^{-1}
\]
extends naturally from $S$ to $\overline{S}$. Moreover, consider $\mb{S}^{\ul{k}}(\mc{G})$, for integers $\ul k = (k_1 \geq k_2 \geq \cdots \geq k_r \geq 0)$ and $r$ the rank of $\mc{F}$. Let $F(\mb{S}^{\ul{k}}(\mc{G}))$ the penultimate step of the descending filtration from $\mb{S}^{\ul{k}}(\mc{G})$ to $\mb{S}^{\ul{k}}(\mc{F})$, defined as in \ref{omegatoHfiltration}. Then, the map
\[
    F(\mb{S}^{\ul{k}}(\mc{G})) \overset{\mb{S}^{\ul{k}}(\phi)}{\longrightarrow} \mb{S}^{\ul{k}}(\mc{H}) \overset{\mb{S}^{\ul{k}}(\phi^{-1})}{\longrightarrow} \mb{S}^{\ul{k}}(\mc{F}) \overset{d}{\longrightarrow} \mb{S}^{\ul{k}}(\mc{F}) \otimes \det \mc{H} \otimes \det \mc{F}^{-1}
\]
extends from $S$ to $\overline{S}$.

\end{lemma}
\begin{proof}
The key observation, from \cite[3.3.1]{EFGMM}, is that the morphism
\[
    d\cdot \phi^{-1} \colon \mc{H} \longrightarrow \mc{F} \otimes \det \mc{H} \otimes \det \mc{F}^{-1}
\]
is the \emph{adjugate} $\phi|_\mc{F}^{\mr{adj}}$ of the morphism $\phi|_\mc{F}$ and thus it is naturally defined over all of $\overline{S}$. Notice that on $\mc{F}$ the morphism $\phi|_\mc{F}^\mr{adj} \circ \phi$ is simply multiplication by $d$.
\par The statement is trivial when $\ul{k} = \ul 0$. We present the proof in the case $\ul{k} = (k, 0, \dots, 0)$, where $\mb{S}^{\ul{k}}(\mc{G}) = \mr{Sym}^k(\mc{G})$. The general case is proved similarly. One can consider the commutative diagram
\[
    \begin{tikzcd}
    & \mr{Sym}^{k-1}(\mc{F}) \otimes \mc{G} \arrow{r} \arrow{rd} & F^{k-1}(\mr{Sym}^k(\mc{G})) \arrow{d} \\
    & & \mr{Sym}^k(\mc{F}) \otimes \det \mc{H} \otimes \det \mc{F}^{-1},
    \end{tikzcd}
\]
with the vertical arrow being the morphism we want to extend. The diagonal arrow is simply $\mr{id} \otimes (\phi|_\mc{F}^\mr{adj} \circ \phi)$, so it extends from $S$ to $\overline{S}$, and it clearly factors through $F^{k-1}(\mr{Sym}^k(\mc{G}))$, so we conclude.
\end{proof}

\par 
\subsubsection{The ordinary splitting}
On $S = S^\mu_K$ the morphism $F \colon H^{(p)}_\sigma \to H_\sigma$ induces via \ref{hfiltsigma1} an isomorphism
\[
    F \colon \ul{\omega}_{A^\vee/S, \sigma}^{\vee, p} \longrightarrow \ul{\omega}_{A^\vee/S, \sigma}^{\vee}.
\]
Moreover, $F$ kills the subsheaf $\ul{\omega}_{A/S, \sigma}^{(p)} \subset H^{(p)}_\sigma = H^{1, (p)}_\mr{dR}(A/S^\mu_K)$ and the composition
\[
    \ul{\omega}_{A^\vee/S, \sigma}^{\vee} \overset{F^{-1}}{\longrightarrow} \ul{\omega}_{A^\vee/S, \sigma}^{\vee, p} \overset{F}{\longrightarrow} H_\sigma
\]
gives a natural splitting of the Hodge filtration \ref{hfiltsigma1}. As in the proof of \ref{lemtorsiondelta}, we write $\mc{U}_\sigma = \im(F \colon H^{(p)}_\sigma \to H_\sigma)$, the locally free subsheaf giving the splitting. In particular, we have that $\mc{U}_\sigma = \ker (V \colon H_\sigma \to H_\sigma^{(p)})$. Moreover, $V \colon H_\sigma \to H_\sigma^{(p)}$ has image $\ul{\omega}_{A/S, \sigma}^{(p)}$ and, over $S^\mu_K$, it induces the isomorphism $V \colon \ul{\omega}_{A/S, \sigma} \to \ul{\omega}_{A/S, \sigma}^{(p)}$. This implies that the composition
\[
    p_{\mr{ur}, 1} \colon H_\sigma \overset{V}{\longrightarrow} \ul{\omega}_{A/S, \sigma}^{(p)} \overset{V^{-1}}{\longrightarrow} \ul{\omega}_{A/S, \sigma}
\]
is the projection $H_\sigma \to \ul{\omega}_{A/S, \sigma}$ parallel to $\mc{U}_\sigma$, or, in other words, the projection $H_\sigma \to \ul{\omega}_{A/S, \sigma}$ relative to the splitting of \ref{hfiltsigma1} given by $\mc{U}_\sigma$. This is usually called the \emph{unit-root splitting}. We have proved the following.
\begin{lemma}[Ordinary unit-root splitting]
Over the ordinary locus $S^\mu_K$, the Hodge filtration \ref{hfiltsigma1} admits a natural splitting induced by the Frobenius $F \colon H_\sigma^{(p)} \to H_\sigma$. The corresponding projection $p_{\mr{ur}, 1} \colon H_\sigma \to \ul{\omega}_\sigma$ is induced by $V$ as described above.
\end{lemma}
The morphism $p_{\mr{ur}, 1}$ cannot be extended naturally to $S_K = \overline{S}^\mu_K$, because the factor $V^{-1} \colon \ul{\omega}_{\sigma}^{(p)} \to \ul{\omega}_{\sigma}$ does not make sense outside $S^\mu_K$. If we consider instead the map $A_1 \cdot p_{\mr{ur}, 1}$, we see that it can be extended to the whole $S_K$: the linear map
\[
    A_1 \cdot V^{-1} \colon \ul{\omega}_{\sigma}^{(p)} \longrightarrow \ul{\omega}_{\sigma} \otimes \det \ul{\omega}_\sigma^{p-1}
\]
is the {adjugate} to the morphism $V \colon \ul{\omega}_\sigma \to \ul{\omega}_\sigma^{(p)}$, which is defined over $S_K$. In fact, from \ref{lemadjext} we deduce the following. 
\begin{lemma}
\label{lemurprojord}
The morphism $A_1 \cdot p_{\mr{ur}, 1} \colon H_{\sigma} \to \ul{\omega}_{\sigma} \otimes \det \ul{\omega}_\sigma^{p-1}$ extends to $S_K$. Similarly, for $\ul{k} = (k_1, \dots, k_{n-1})$, with $k_1 \geq  k_2 \geq \cdots \geq k_{n-1} \geq 0$, we have
\[
    A_1 \cdot \mb{S}^{\ul{k}}(p_{\mr{ur}, 1})|_{F(\mb{S}^{\ul{k}}(H_\sigma))} \colon F(\mb{S}^{\ul{k}}(H_\sigma)) \longrightarrow \mb{S}^{\ul{k}}(\ul{\omega}_\sigma) \otimes \det \ul{\omega}_\sigma^{p-1}
\]
extends to $S_K$.
\end{lemma}

\subsubsection{Partial splittings}
\label{secaosplit}
Fix some $2 \leq r \leq n-1$. On $S=S_{K, w_r}$, the map $F \colon \ul{\omega}_{A^\vee/S, \sigma}^{\vee, p} \longrightarrow \ul{\omega}_{A^\vee/S, \sigma}^{\vee}$ is zero. Nevertheless, we can still look at the surjection $V \colon H_\sigma \to \ul{\omega}_\sigma^{(p)}$. On $S$, the map $V \colon \ul{\omega}_\sigma \to \ul{\omega}_\sigma^{(p)}$ is not surjective, but it induces an isomorphism between the multiplicative part $\ul{\omega}_\mu$ and its $p$-twist. In particular, we can consider the composition
\[
    p_{\mr{ur}, r} \colon H_{\mu, r} \overset{V}{\longrightarrow} \ul{\omega}_\mu^{(p)}  \overset{V^{-1}}{\longrightarrow} \ul{\omega}_{\mu, r}.
\]
This composition is the right inverse of the inclusion $\ul{\omega}_{\mu, r} \to H_{\mu, r}$. That is, the kernel of $p_{\mr{ur}, r}$ gives a splitting of the ses
\[
    0 \longrightarrow \ul{\omega}_{\mu, r} \longrightarrow H_{\mu, r} \longrightarrow \ul{\omega}_{A^\vee/S, \sigma}^\vee \longrightarrow 0.
\]
We call this the \emph{(generalised) partial unit-root splitting}. Recall that we write $A_{r} = \det(V \colon \ul{\omega}_\mu \to \ul{\omega}_\mu^{(p)})$. As in the ordinary case, we can prove the following using Lemma \ref{lemadjext}.
\begin{lemma}
\label{almostordpur}
The morphism $A_{r} \cdot p_{\mr{ur}, r} \colon H_{\mu, r} \longrightarrow \ul{\omega}_{\mu, r} \otimes \det  \ul{\omega}_{\mu, r}^{p-1}$ extends to $\overline{S}_{K, w_r}$. Similarly, for $\ul{k} = (k_1, k_2, \dots, k_{n-r})$, with $k_1 \geq  k_2 \geq \cdots \geq k_{n-r} \geq 0$, we have
\[
    A_{r} \cdot \mb{S}^{\ul{k}}(p_{\mr{ur}, r})|_{F(\mb{S}^{\ul{k}}(H_{\mu, r}))} \colon F(\mb{S}^{\ul{k}}(H_{\mu, r})) \longrightarrow \mb{S}^{\ul{k}}(\ul{\omega}_{\mu, r}) \otimes \det \ul{\omega}_{\mu, r}^{p-1}
\]
extends to $\overline{S}_{K, w_r}$. 
\end{lemma}
In the process of constructing the generalised theta operators we will give a mild generalisation of \ref{almostordpur}, which again follows from Lemma \ref{lemadjext}.
\subsection{Generalised theta operators}
\label{genthtssc}
We prove here the main result of this paper, Theorem \ref{bigthm}. Let us start by recalling some notations. Pick $1 \leq r \leq n-1$ an integer and $(\ul{k}, w)$ an automorphic weight with $k_{n-1}\geq 0$. On $\overline{S}_{K, w_r}$ we have the short exact sequence
\[
    0 \longrightarrow \ul{\omega}_{0, r} \longrightarrow \ul{\omega}_{\sigma} \longrightarrow \ul{\omega}_{\mu, r} \longrightarrow 0,
\]
where $\ul{\omega}_{0, r} = \ker (V^{r-1} \colon \ul{\omega}_\sigma \to \ul{\omega}_\sigma^{(p^{r-1})})$ is a locally free sheaf of rank $r-1$. This short exact sequence induces a filtration $F^{\bullet, r}$ on $\ul{\omega}^{\ul{k}, w}$ over $\overline{S}_{K, w_r}$, which is trivial when $r=1$ (see \ref{secgenonfilt} for more details on the definition of $F^{\bullet, r}$). The filtration $F^{\bullet, r}$ gives rise to the graded sheaf $\mr{gr}^{\bullet, r}(\ul{\omega}^{\ul{k}, w})$ over $\overline{S}_{K, w_r}$. Notice that the graded sheaf $\mr{gr}^{\bullet, r}(\ul \omega^{\ul k, w})$ decomposes as a direct sum of sheaves
\[
    \mr{gr}^{\ul a, r}(\ul \omega^{\ul k, w}) \coloneqq\, \sum_{\ul b} \im( \otimes_j \wedge^{l_j-b_j}(\ul \omega_0) \otimes \wedge^{b_j}(\ul \omega_\mu) \to \otimes_i \mr{Sym}^{k_i-a_i}(\ul \omega_0) \otimes \mr{Sym}^{a_i}(\ul \omega_\mu))
\]
where $\ul a, \ul b \in \Z^{n-1}_{\geq 0}$ are such that $a_i \leq k_i, b_j \leq l_j$ for all $i, j$ and $\lvert \ul a \rvert = \lvert \ul b \rvert$. In particular, given $f \in H^0(\overline{S}_{K, w_r}, \mathrm{gr}^{\bullet, r}(\ul \omega^{\ul k, w}))$, we can write it uniquely as $f = \sum_{\ul a} f_{\ul a}$, where $f_{\ul a} \in H^0(\overline{S}_{K, w_r}, \mr{gr}^{\ul a, r}(\ul \omega^{\ul k, w}))$.
\begin{theorem}
\label{bigthm}
Let $1 \leq r < n$ be an integer and $(\ul{k}, w)$ an automorphic weight with $k_{n-1}\geq 0$. There exists a differential operator
\[
    \theta_r \colon \mr{gr}^{\bullet, r}(\ul{\omega}^{\ul{k}, w}) \longrightarrow \mr{gr}^{\bullet, r}(\ul{\omega}^{\ul{k}+\ul{\Delta}_r, w}),
\]
defined on the (closure of the) Ekedahl-Oort stratum $\overline{S}_{K, w_r}$, with
\[
    \ul{\Delta}_r = (p+1, p, \cdots, p, 1, \cdots, 1)
\]
where exactly the last $r-1$ entries are $1$. The operator $\theta_r$ satisfies the following properties:
\begin{enumerate}
    \item The operator $\theta_r$ is $A_{r}$-linear, that is, $\theta_r(A_{r}) = 0$, where $A_{r}$ is the partial Hasse invariant defined above.
    \item The operator $\theta_r$ is Hecke-equivariant.
    \item \label{wtfiltration} Let $f \in H^0(\overline{S}_{K, w_r}, \mathrm{gr}^{\bullet, r}(\ul \omega^{\ul k, w}))$ and write it as $f = \sum_{\ul a} f_{\ul a}$. Then $\theta_r(f)$ is divisible by the Hasse invariant $A_r$ if and only if for each component $f_{\ul a}$ either $A_r \mid f_{\ul a}$ or $p \mid \lvert \ul a \rvert$. 
\end{enumerate}
\end{theorem}
We prove this theorem throughout this section.

\begin{remark}
In what follows we will consider several sheaves, as well as morphisms between them, which are naturally defined over $S_K$ (or even $\mc{S}_K$), but will be restricted to subschemes $S \subseteq S_K$. To ease notations, for a sheaf $\mc{F}$ on $S_K$ we will not in general write $\mc{F}|_{S}$ and similarly for a morphism $\phi \colon \mc{F} \to \mc{G}$ we will not write $\phi|_{S} \colon \mc{F}|_{S} \to \mc{G}|_{S}$, instead we will simply keep writing $\mc{F}$ and $\phi \colon \mc{F} \to \mc{G}$, with the understanding that everything is restricted to a subscheme that will be clear from context.
\end{remark}

\subsubsection{Ordinary operator}
Let $(\ul{k}, w)$ be an automorphic weight such that $k_{n-1} \geq 0$, with the
corresponding automorphic sheaf $\ul{\omega}_{A/S_K}^{\ul{k},w}$, which we denote simply by $\ul{\omega}^{\ul{k},w}$. Consider $H = H^1_\mr{dR}(A/S_K)$ and write
\[
    H^{\ul{k}, w} \coloneqq \mb{S}^\ul{k}(H_\sigma) \otimes_{\mc{O}_S} \delta_{\sigma}^w.
\]
The GM connection induces a morphism
\[
    \ul{\omega}^{\ul{k}, w} \overset{\nabla}{\longrightarrow} H^{\ul{k}, w} \otimes \Omega^1_{S_K/\mb{F}}.
\]
By \ref{koszultransv}, the image of this map lies in the penultimate step of the natural filtration $F^\bullet$ of $H^{\ul{k}, w}$ described in \ref{omegatoHfiltration}.
In particular, from \ref{lemurprojord} we deduce the following.
\begin{lemma}[{\cite[3.3.2]{EFGMM}}]
\label{lemexturtheta}
Consider the composition
\[
    \psi \colon \ul{\omega}^{\ul{k}, w} \overset{\nabla}{\longrightarrow} H^{\ul{k}, w} \otimes \Omega^1_{S/\mb{F}} {\longrightarrow} \ul{\omega}^{\ul{k}, w} \otimes \det \ul{\omega}_\sigma^{p-1} \otimes \Omega^1_{S_K/\mb{F}} \longrightarrow \ul{\omega}^{\ul{k}+\ul{p-1}, w} \otimes \Omega^1_{S_K/\mb{F}},
\]
where the second arrow is $A_1 \cdot (\mb{S}^{\ul{k}}(p_{\mr{ur},1}) \otimes \mr{id}_{\delta_\sigma} \otimes \mr{id}_{\Omega^1})$ and $(\ul{p-1}, 0)=((p-1, \cdots, p-1), 0)$ is the weight of $A_1$. Then $\psi$ extends from $S_K^\mu$ to $S_K$.
\end{lemma}
We may now define $\theta_1$, the \emph{ordinary theta operator}, as follows. Composing  $\psi$ from \ref{lemexturtheta} with $\mr{id} \otimes \ul{\mr{ks}}^{-1}$, we obtain
\begin{equation}
\label{eqntheta1step}
    \ul{\omega}^{\ul{k}, w} {\longrightarrow} \ul{\omega}^{\ul{k}+\ul{p-1}, w} \otimes \ul{\omega}_\sigma \otimes \det \ul{\omega}_\sigma \otimes \delta_\sigma^{-1}.
\end{equation}
To define $\theta_1$, we further compose \ref{eqntheta1step} with the natural morphism $\ul{\omega}^{\ul k + \ul{p-1}, w} \otimes \ul{\omega}_\sigma \otimes \det \ul{\omega}_\sigma \to \ul{\omega}^{\ul k + {\ul \Delta}_1, w}$ and obtain
\[
    \theta_1 \colon \ul{\omega}^{\ul{k}, w} {\longrightarrow} \ul{\omega}^{\ul{k}+\ul{\Delta}_1, w-1},
\]
where $\ul{\Delta}_1 = (p+1, p, \dots, p)$. This proves the existence part of \ref{bigthm} in the case $r=1$. Note that this theta operator is already defined in \cite{EFGMM}.

\begin{proposition}
\label{propbastheta1}
The operator $\theta_1$ satisfies the following fundamental properties:
\begin{enumerate}
    \item For $f \in H^0(S_K, \ul{\omega}^{\ul{k}, w}), g \in H^0(S_K, \ul{\omega}^{\ul{k}', w'})$ we have
    \[
        \theta_1(fg) = f\theta_1(g) +\theta_1(f)g,
    \]
    that is, the operator $\theta_1$ is a derivation of the algebra of modular forms of level $K$ with coefficients in $\mb{F}$.
    \item The operator $\theta_1$ is $A_1$-linear, that is, $\theta_1(A_1) = 0$.
    \item The operator $\theta_1$ is Hecke-equivariant.
\end{enumerate}
\end{proposition}
\begin{proof}
\begin{enumerate}
    \item This is clear from the definition of $\nabla$ and thus of $\theta_1$.
    \item We can choose a small enough dense open $U \subseteq S^\mu_K$ such that $\ul{\omega}_\sigma$ and $H_\sigma$ are both free on $U$ and the Hodge filtration \ref{hfiltsigma1} splits. In particular, we can pick a local basis $e_1, \dots, e_{n-1}$ of $\ul{\omega}_\sigma$ and complete it to a local basis of $H_\sigma$ with the addition of some $e_n$. Then, in these local coordinates, the ordinary Hasse invariant is $A_1|_U = a \cdot (e_1 \wedge \cdots \wedge e_{n-1})^{p-1}$, for some $a \in \mc{O}_{S_K}(U)$. With these notations, from the definition of $A_1$, one can compute that
    \begin{align*}
        (\wedge^{p-1} V)(A_1)|_U &= a (Ve_1 \wedge \cdots \wedge Ve_{n-1})^{p-1} \\
        &= a (a (e_1 \wedge \cdots \wedge e_{n-1})^{p})^{p-1} \\
        &= a^p (e_1 \wedge \cdots \wedge e_{n-1})^{p(p-1)}\\
        &= A_1^p.
    \end{align*}
    Therefore, from the definition of $p_{\mr{ur}, 1} = V|_{\ul{\omega}_\sigma}^{-1} \circ V_{\sigma}$ and \ref{lemGMpullback} we have on $U$ that
    \begin{align*}
        (\wedge^{p-1}p_{\mr{ur}, 1}) \circ \nabla(A_1) &= (\wedge^{p-1}V^{-1})(\wedge^{p-1} V) \nabla(A_1)\\
                &= (\wedge^{p-1} V^{-1}) \nabla((\wedge^{p-1} V) A_1) \\
                & = (\wedge^{p-1} V^{-1}) \nabla(A_1^p) = 0.
    \end{align*}
    Hence, $\theta_1(A_1)|_U = 0$. By density of $U \subset S_K$, we deduce that $\theta_1(A_1) = 0$.
    \item The construction of $\theta_1$ does not depend on the level structure and it is compatible with changes of base involved in the definition of the Hecke operators. Moreover, any prime-to-$p$ quasi-isogeny $f \colon A \to A'$, for $A, A' \in S_K(T)$, $T \to S_K$ finite \'etale, induces an isomorphism $f^\ast \colon \ul{\omega}_{A'/T} \to \ul{\omega}_{A/T}$ such that
    \[
        \begin{tikzcd}
            & \ul{\omega}_{A'/T, \sigma} \otimes_{\mc{O}_S} \det \ul{\omega}_{A'/T, \sigma} \otimes \delta_{A'/T, \sigma}^{-1} \arrow{r}{\ul{\mr{ks}}_{A'/T}} \arrow{d}{f^\ast} & \Omega^1_{T/\mb{F}} \arrow[equal]{d} \\
            & \ul{\omega}_{A/T, \sigma} \otimes_{\mc{O}_S} \det \ul{\omega}_{A/T, \sigma} \otimes \delta_{A/T, \sigma}^{-1} \arrow{r}{\ul{\mr{ks}}_{A/T}} & \Omega^1_{T/\mb{F}}
        \end{tikzcd}
    \]
    is commutative. All of this, combined with \ref{lemGMpullback}, implies that $\theta_1$ is Hecke-equivariant. Notice that this relies on our convention for the definition of $\ul{\mr{ks}}$: had the polarisation been involved, the morphism $\ul{\mr{ks}}$ would have introduced a non-trivial Hecke-action, see \cite[Lem.~4.1.2]{EFGMM}.
    
\end{enumerate}
\end{proof}
\begin{remark}
We can provide an alternative proof of point 2 of \ref{propbastheta1}. From, for instance, the proof of \cite[Prop.~4.1.1]{padicpropMSMF}, we see that we can cover $S^\mu_K$ with maps $U \to S_K^\mu$, finite \'etale of degree $\# \mr{GL}_{n-1}(\mb{F}_p)$, on which we can find a local basis $e_1, \dots, e_{n-1} \in \ul{\omega}_\sigma(U)$, which extends to a basis $e_1, \dots, e_n \in H_\sigma(U)$, such that: 
\begin{align*}
    V(e_i) &= e_i^{(p)} = e_i \otimes 1, \quad i = 1, \dots, n-1,\\
    V(e_n) &= 0.
\end{align*}
This follows from applying the proposition in \emph{loc.\ cit.}\ with $H = \ul{\omega}^\vee_\sigma|_O$, $S_1 = O$, for some non-empty affine open $O\subseteq S^\mu_K$, with notations $H$ and $S_1$ as in the reference, and $F = V^{\vee}$. In that case, we have
\[
    A_1|_U = (e_1 \wedge \cdots \wedge e_{n-1})^{p-1}.
\]
Moreover, by \ref{lemGMpullback} and \ref{lemGMbasechange}, we have that on $U$
\[
    0 = \nabla(e_i^{(p)}) = \nabla(V(e_i)) = V(\nabla(e_i)), \quad i = 1, 2, \dots, n-1,
\]
since $dF_{U, \mr{abs}} = 0$. This implies that $\nabla(e_i) = e_n \otimes \omega_i$, for some sections $\omega_i \in \Omega^1_{S_K/\mb{F}}(U)$ which form a local basis of $\Omega^1_{S_K/\mb{F}}$, by \ref{propksiso}. With this description of the $\nabla(e_i)$'s, one can see that over $U$ we have $\theta_1(A_1)|_U = 0$, thanks to a simple computation, using the definitions of $\nabla, p_{\mr{ur}, 1}$ and thus $\theta_1$. Therefore, since we can cover $S^\mu_K$ with such finite \'etale $U$'s,  $\theta_1(A_1)$ vanishes on the ordinary locus $S^\mu_K$ and by density we have $\theta_1(A_1) = 0$ on all of $S_K$.\\
This shows that $\theta_1$ admits a relatively simple and explicit description \'etale-locally on the ordinary locus. In fact, one can prove that it is enough to take $U = I$, with $I$ being the (small) {Igusa variety} $I \to S^\mu_K$ of level $1$ and achieve the same results working globally. This is related to the approach taken in \cite{deshagor19} to construct theta operators when $p$ is inert. See in particular \cite[2.1.4]{deshagor19}. See also \cite{howe}.\\
One could use this approach more generally on $S_{K, w_r}$, $2 \leq r \leq n-1$, instead of $S_K^\mu$, to obtain \'etale-local bases of $\ul{\omega}_{\mu, r}$ such that $V \colon \ul{\omega}_{\mu, r} \to \ul{\omega}_{\mu, r}^{(p)}$ admits a description relative to them analogous to the one given above. We do not pursue this direction here.
\end{remark}

\subsubsection{Generalised theta operators}
Fix $2 \leq r \leq n-1$. Let $(\ul{k}, w)$ be an automorphic weight such that $k_{n-1} \geq 0$ and consider the sheaf $\ul{\omega}^{\ul{k},w}$. We look in particular at the restriction of $\ul{\omega}^{\ul{k}, w}$ to $\overline{S}_{K, w_r}$, where we have, as above, the short exact sequence
\[
    0 \longrightarrow \ul{\omega}_{0, r} \longrightarrow \ul{\omega}_\sigma \longrightarrow \ul{\omega}_{\mu, r} \longrightarrow 0.
\]
From this sequence, we get a Koszul filtration on $\wedge^j \ul{\omega}_\sigma$, for all $j=1, \dots, n-1$, following the notations of \ref{seckoszulfiltext}. Similarly we can define a Koszul filtration on $\wedge^j H_\sigma$, $j=1, \dots, n-1$, from 
\[
    0 \longrightarrow \ul{\omega}_{0, r} \longrightarrow H_\sigma \longrightarrow H_{\mu, r} \longrightarrow 0.
\]
From \ref{lemGMpullback} and \ref{propmoreonKScharp}, we deduce that the GM connection on $\overline{S}_{K, w_r}$ induces morphisms between the graded pieces of these two Koszul filtrations as follows
\[
\begin{tikzcd}
    & \mr{gr}^i(\wedge^j \ul{\omega}_\sigma) \arrow[r, "\nabla_r"] \arrow[equal]{d} &\mr{gr}^i(\wedge^j H_\sigma) \otimes \Omega^1_{\overline{S}_{K, w_r}/\mb{F}} \arrow[equal]{d}\\
    & \wedge^i \ul{\omega}_{0, r} \otimes \wedge^{j-i} \ul{\omega}_{\mu, r} \arrow[r, "\nabla_r"] &\wedge^i \ul{\omega}_{0, r} \otimes \wedge^{j-i} H_{\mu, r} \otimes \Omega^1_{\overline{S}_{K, w_r}/\mb{F}}
\end{tikzcd}
\]
In fact, we can be more precise. From the inclusion $\ul{\omega}_{\mu, r} \subset H_{\mu, r}$ we obtain a third Koszul filtration on $\wedge^{j-i} H_{\mu, r}$, which we denote $K^\bullet(\wedge^{j-i} H_{\mu, r})$. Applying \ref{koszultransv} to $K^\bullet$, we see that actually
\[
    \nabla_r(\mr{gr}^i(\wedge^j \ul{\omega}_\sigma)) \subseteq \wedge^i \ul{\omega}_{0, r} \otimes K^{j-i-1}(\wedge^{j-i} H_{\mu, r}) \otimes \Omega^1_{\overline{S}_{K, w_r}/\mb{F}} \subseteq \mr{gr}^i(\wedge^j H_\sigma) \otimes \Omega^1_{\overline{S}_{K, w_r}/\mb{F}}.
\]
Let $\psi$ be the following composition
\begin{align*}
    \psi \colon \mr{gr}^i(\wedge^j \ul{\omega}_\sigma) &\overset{\nabla_r}{\longrightarrow} \wedge^i \ul{\omega}_{0, r} \otimes K^{j-i-1}(\wedge^{j-i} H_{\mu, r}) \otimes \Omega^1_{\overline{S}_{K, w_r}/\mb{F}}\\
    &\longrightarrow \mr{gr}^i(\wedge^j \ul{\omega}_\sigma) \otimes \det \ul{\omega}_{\mu, r}^{p-1} \otimes \Omega^1_{\overline{S}_{K, w_r}/\mb{F}},
\end{align*}
considered over $S_{K, w_r}$, where the second map acts as $A_{r} \cdot \wedge^{j-i} p_{\mr{ur}, r}$ on the factor $K^{j-i-1}(\wedge^{j-i} H_{\mu, r})$ and is the identity on the others.
From \ref{almostordpur} we deduce that $\psi$ extends to $\overline{S}_{K, w_r}$. \par This construction can be generalised. From the Koszul filtrations on $\wedge^j \ul{\omega}_\sigma, \wedge^j H_\sigma$ defined by the subsheaf $\ul{\omega}_{0, r}$, we obtain filtrations on 
\[
    \otimes_j \wedge^{l_j}( \ul \omega_\sigma), \otimes_j \wedge^{l_j}(H_\sigma),\quad j=1, 2, \dots, n-1,
\]
as explained in \ref{secfiltsym}. From these, in turn, we obtain filtrations on $\ul{\omega}^{\ul{k}, w}$ and $H^{\ul{k}, w} = \mb{S}^{\ul{k}}(H_\sigma) \otimes \delta_\sigma^w$, the filtration on $\delta_\sigma^w$ being trivial. The morphism $p_{\mr{ur}, r} \colon H_{\mu, r} \to \ul{\omega}_{\mu, r}$ from \ref{secaosplit}, along with $\mr{id}_{\ul{\omega}_{0, r}}, \mr{id}_{\delta_\sigma}$, lends naturally a morphism of sheaves
\[
    \mr{gr}^\bullet(p_{\mr{ur}, r}) \colon \mr{gr}^\bullet(H^{\ul{k}, w}) \longrightarrow \mr{gr}^\bullet(\ul{\omega}^{\ul{k}, w})
\]
on $S_{K, w_r}$. In fact, under the natural identification $\mr{gr}^\bullet(\mb{S}^{\ul k}(H_\sigma)) = \mb{S}^{\ul k}(\mr{gr}^\bullet(H_\sigma)) = \mb{S}^{\ul k}(\ul \omega_{0,r} \oplus H_{\mu, r})$, we have $\mr{gr}^\bullet(p_{\mr{ur}, r}) = \mb{S}^{\ul k}(\mr{id}_{\ul \omega_{0,r}} \oplus p_{\mr{ur}, r}) \otimes \mr{id}_{\delta_\sigma}^w$.
From  \ref{lemadjext} and \ref{almostordpur} we deduce the following, which is analogous to Lemma \ref{lemexturtheta}.
\begin{lemma}
Consider the composition
\[
    \psi \colon \mr{gr}^\bullet(\ul{\omega}^{\ul{k}, w}) \overset{\nabla_r}{\longrightarrow} \mr{gr}^\bullet(H^{\ul{k}, w}) \otimes \Omega^1_{\overline{S}_{K, w_r}/\mb{F}} {\longrightarrow} \mr{gr}^\bullet(\ul{\omega}^{\ul{k}, w}) \otimes \det \ul{\omega}_\mu^{p-1} \otimes \Omega^1_{\overline{S}_{K, w_r}/\mb{F}},
\]
where the second arrow is $A_{r} \cdot \mr{gr}^\bullet(p_{\mr{ur}, r}) \otimes \mr{id}_{\Omega^1}$. Then, $\psi$, which a priori is only defined on $S_{K, w_r}$, extends to $\overline{S}_{K, w_r}$.
\end{lemma}
We can further compose $\psi$ with $\mr{id} \otimes \ul{\mr{ks}}_{\mu, r}^{-1}$ defined in \ref{propmoreonKScharp} to get
\begin{align*}
    \mr{gr}^\bullet(\ul{\omega}^{\ul{k}, w}) &\overset{\psi}{\longrightarrow} \mr{gr}^\bullet(\ul{\omega}^{\ul{k}, w}) \otimes \det \ul{\omega}_{\mu, r}^{p-1} \otimes \Omega^1_{\overline{S}_{K, w_r}/\mb{F}}  \\
    & \overset{\mr{id} \otimes \ul{\mr{ks}}^{-1}_{\mu}}{\longrightarrow} \mr{gr}^\bullet(\ul{\omega}^{\ul{k}, w}) \otimes \det \ul{\omega}_{\mu, r}^{p-1} \otimes \ul{\omega}_\mu \otimes \det \ul{\omega}_\sigma \otimes \delta_{\sigma}^{-1}\\
    & \longrightarrow \mr{gr}^\bullet(\ul{\omega}^{\ul{k}+\ul{\Delta}_r, w-1}),
\end{align*}
where the last map is given by the natural multiplication morphisms.
Therefore, we obtain a morphism on $\mr{gr}^\bullet(\ul{\omega}^{\ul{k}, w})$, shifting the weight by $\ul{\Delta}_r =(p+1, p, \cdots, p, 1, \cdots, 1)$. We call this composition
\[
    \theta_r \colon \mr{gr}^\bullet(\ul{\omega}^{\ul{k}, w}) \longrightarrow \mr{gr}^\bullet(\ul{\omega}^{\ul{k}+\ul{\Delta}_r, w-1})
\]
the \emph{generalised (partial) theta operator} relative to the stratum $\overline{S}_{K, w_r}$. This proves the existence part of Theorem \ref{bigthm} in the case $r \geq 2$.

\begin{proposition}[Basic properties of $\theta_r$]
\label{propbastheta2}
The operator $\theta_r$ satisfies the following fundamental properties:
\begin{enumerate}
    \item For $f \in H^0(\overline{S}_{K, w_r}, \mr{gr}^\bullet(\ul{\omega}^{\ul{k}, w})), g \in H^0(\overline{S}_{K, w_r}, \mr{gr}^\bullet(\ul{\omega}^{\ul{k'}, w'}))$ we have
        $\theta_r(fg) = f\theta_r(g) +\theta_r(f)g$.
    \item The operator $\theta_r$ is $A_{r}$-linear, that is, $\theta_r(A_{r}) = 0$.
    \item The operator $\theta_r$ is Hecke-equivariant.
\end{enumerate}
\end{proposition}
\begin{proof}
\begin{enumerate}
    \item This follows from the construction of $\theta_r$ and the properties of $\nabla_r$.
    \item Consider $U \subseteq S_{K, w_r}$ a small enough dense open, so that we can choose a local basis $e_1, \dots, e_n$ of $H_\sigma$ over $U$ such that $e_1, \dots, e_{r-1}$ is a basis of $\ul{\omega}_0$, $e_1, \dots, e_{n-1}$ a basis of $\ul{\omega}_\sigma$ and $e_n$ reduces to a basis of $H_\sigma/\ul{\omega}_\sigma$.
    In particular, we see that
    \[
        A_{r}|_U = a (\overline{e}_r \wedge \cdots \wedge \overline{e}_{n-1})^{p-1}, 
    \]
    for some $a \in \mc{O}_{S_{K, w_r}}(U)$, where $\overline{\cdot}$ denotes the reduction through $H_\sigma \to H_{\mu, r}$. As in \ref{propbastheta1}, one can see that $\theta_r(A_{r})|_U = 0$, so that by density $\theta_r(A_{r}) = 0$.
    \item The same arguments that we used for $\theta_1$ work.
\end{enumerate}
\end{proof}

\begin{remark}
Notice that multiplication by $A_r$, for $1 \leq r \leq n-1$, is also Hecke-equivariant. By this we mean that it induces the commutative diagram
\[
        \begin{tikzcd}
            & H^i(\overline{S}_{K, w_r}, \mr{gr}^{\bullet, r}(\ul{\omega}^{\ul{k}, w})) \arrow{d}{A_r\cdot} \arrow{r}{T_g} &H^i(\overline{S}_{K, w_r}, \mr{gr}^{\bullet, r}(\ul{\omega}^{\ul{k}, w})) \arrow{d}{A_r\cdot} \\
            & H^i(\overline{S}_{K, w_r}, \mr{gr}^{\bullet, r}(\ul{\omega}^{\ul{k}+\ul{w}_r, w})) \arrow{r}{T_g} &H^i(\overline{S}_{K, w_r}, \mr{gr}^{\bullet, r}(\ul{\omega}^{\ul{k}+\ul{w}_r, w})),
        \end{tikzcd}
\]
for any $i \geq 0$, $g \in \mathbf{G}(\mb{A}^{p, \infty})$ and automorphic weight $(\ul{k}, w)$, where $\ul{w}_r = (0, \dots, p-1, \dots, 0)$ is the weight shift produced by $A_r$. This follows, with notations as in \ref{sectamehecke}, from the commutativity of the diagrams
\[
    \begin{tikzcd}
    & \mr{gr}^{\bullet, r}(\ul{\omega}^{\ul{k}, w}) \arrow{r}{p_2^\ast} \arrow{d}{A_r\cdot} & p_{2, \ast} p_2^\ast \mr{gr}^{\bullet, r}(\ul{\omega}^{\ul{k}, w}) \arrow{d}{A_r\cdot} & p_{1, \ast} p_1^\ast \mr{gr}^{\bullet, r}(\ul{\omega}^{\ul{k}, w}) \arrow{r}{\mr{tr}} \arrow{d}{A_r\cdot} & \mr{gr}^{\bullet, r}(\ul{\omega}^{\ul{k}, w}) \arrow{d}{A_r\cdot}\\
    & \mr{gr}^{\bullet, r}(\ul{\omega}^{\ul{k}+\ul{w}_r, w}) \arrow{r}{p_2^\ast} & p_{2, \ast} p_1^\ast \mr{gr}^{\bullet, r}(\ul{\omega}^{\ul{k}+\ul{w}_r, w}),  & p_{1, \ast} p_1^\ast \mr{gr}^{\bullet, r}(\ul{\omega}^{\ul{k}+\ul{w}_r, w}) \arrow{r}{\mr{tr}} & \mr{gr}^{\bullet, r}(\ul{\omega}^{\ul{k}+\ul{w}_r, w})
    \end{tikzcd}
\]
of quasi-coherent sheaves over $S_K$. The commutativity of the former follows from the fact that $A_r$ is invariant under base change. The commutativity of the latter, from the fact that any prime-to-$p$ quasi-isogeny $f \colon A \to A'$, for $A, A' \in \overline{S}_{K, w_r}(T)$, $T \to \overline{S}_{K, w_r}$ finite \'etale, induces an isomorphism $f^\ast \colon \ul{\omega}_{A'/T} \to \ul{\omega}_{A/T}$ that respects the filtration $\ul{\omega}_{0, r} \subseteq \ul{\omega}_{\sigma}$ and sends $A_r(A'/T) \in H^0(T, \det \ul{\omega}_{A'/T, \mu}^{p-1})$ to $A_r(A/T) \in H^0(T, \det \ul{\omega}_{A/T, \mu}^{p-1})$.
\end{remark}

\begin{remark}
Notice that even though $\theta_r$ is not defined on $\ul{\omega}^{\ul{k}, w}$, it is defined on the graded parts of a filtration on such sheaves. Since this filtration is Hecke-equivariant, so is the one it induces on $H^0(\overline{S}^\mr{min}_{K, w_r}, \ul{\omega}^{\ul{k}, w})$ and, therefore, every Hecke eigensystem appearing in $H^0(\overline{S}_{K, w_r}, \ul{\omega}^{\ul{k}, w})$ will also appear in 
\[ H^0(\overline{S}_{K, w_r}, \mr{gr}^{\bullet, r}(\ul{\omega}^{\ul{k}, w})).
\]
Furthermore, thanks to the work of \cite[11]{koskgold}, the Hecke-eigensystems found in $\{H^0(\overline{S}_{K, w_r}, \ul{\omega}^{\ul{k}, w})\}_{\ul{k}, w}$ are the same as those found in $\{H^0(\overline{S}_{K}, \ul{\omega}^{\ul{k}, w})\}_{\ul{k}, w}$.
\end{remark}

\subsubsection{Operators $B_r$}
Here we conclude the proof of Theorem \ref{bigthm}. In particular, we prove point \ref{wtfiltration}. We will achieve this by describing the action of the restriction of $\theta_r$ to $\overline{S}_{K, w_{r+1}}$. Let us consider $V \colon \ul \omega_{\mu, r} \to \ul\omega_{\mu, r}^{(p)}$ over the substratum ${S}_{K, w_{r+1}} \subset \overline{S}_{K, w_{r}}$. From the canonical filtration of $A[\mf{p}_\sigma]$ we deduce, over $\overline{S}_{K, w_{r+1}}$, the existence of a short exact sequence
\begin{equation}
\label{onemorefilt}
    0 \longrightarrow \ul \omega_{\mu, r, 0} \longrightarrow \ul \omega_{\mu, r} \longrightarrow \ul \omega_{\mu, r+1} \longrightarrow 0,
\end{equation}
where $\ul \omega_{\mu, r, 0} = \ker(V \colon \ul \omega_{\mu, r} \to \omega_{\mu, r}^{(p)}) = \ul \omega_{0, r+1} / \ul \omega_{0, r}$ is an invertible sheaf. Notice that $V \colon \ul \omega_{\mu, r+1} \to \ul \omega_{\mu, r+1}^{(p)}$ gives a splitting of the $p$-twist of \ref{onemorefilt} over $S_{K, w_{r+1}}$, since that map is an isomorphism on $S_{K, w_{r+1}}$, so that we get in particular a projection $\pi_{r, 0} \colon \ul \omega_{\mu, r}^{(p)} \to \ul \omega_{\mu, r, 0}^p$. Considering now $V \colon H_{\mu, r} \to \ul \omega_{\mu, r}^{(p)} \subseteq H_{\mu, r}^{(p)}$ over ${S}_{K, w_{r+1}}$, we see that one obtains an induced isomorphism $\pi_{r, 0}V \colon H_{\mu, r} /\ul \omega_{\mu, r} \cong (\det \ul \omega_\sigma)^{-1} \otimes \delta_\sigma \to \ul \omega_{\mu, r, 0}^{p}$ and thus a nowhere vanishing section $B_{r,0} \in H^0({S}_{K, w_{r+1}}, \ul \omega_{\mu, r, 0}^{p} \otimes \det \ul \omega_\sigma \otimes \delta_{\sigma}^{-1})$. One can also consider the morphism
\[
\begin{tikzcd}
    &\ul \omega_{\mu, r, 0}^p \cong \ul \omega_{\mu, r, 0} \otimes \ul \omega_{\mu, r, 0}^{p-1} \arrow[r, hook] & \ul \omega_{\mu, r} \otimes \ul \omega_{\mu, r, 0}^{p-1} \arrow[r, "A_{r+1}\cdot"] & \ul \omega_{\mu, r} \otimes \det(\ul \omega_{\mu, r})^{p-1}
\end{tikzcd}
\]
which, by slight abuse of notation, we denote again $(A_{r+1} \cdot)$. We set $B_r \coloneqq A_{r+1} \cdot B_{r, 0}$, which we identify with a section of $\ul \omega_{\mu, r} \otimes \det(\ul \omega_{\mu, r})^{p-1} \otimes \det \ul \omega_\sigma \otimes \delta_{\sigma}^{-1}$ over $S_{K, w_{r+1}}$. Notice that $B_r$ is nowhere vanishing on $S_{K, w_{r+1}}$. We now extend $B_r$ to a section over the closure $\overline{S}_{K, w_{r+1}}$, which vanishes on the complement of $S_{K, w_{r+1}}$.
\par First, consider the restriction of the partial unit-root projection $A_r \cdot p_{\mr{ur}, r} \colon H_{\mu, r} \to \ul \omega_{\mu, r} \otimes \det(\ul \omega_{\mu, r})^{p-1}$ to $\overline{S}_{K, w_{r+1}}$. Since it coincides with multiplication by $A_r \equiv 0$ when restricted to $\ul \omega_{\mu, r}$, we see that it factors through $H_{\mu, r} / \ul \omega_{\mu, r}$. On that quotient, when further restricted to ${S}_{K, w_{r+1}}$, one can see that $A_r \cdot p_{\mr{ur}, r}$ coincides with multiplication by $B_r$. In fact, one can write $A_r \cdot p_{\mr{ur}, r} \colon H_{\mu, r}/ \ul \omega_{\mu, r} \to \ul \omega_{\mu, r} \otimes \det(\ul \omega_{\mu, r})^{p-1}$ as $B_{r, 0} = \pi_{r, 0}V\colon H_{\mu,r}/\ul\omega_{\mu, r} \to \ul \omega_{\mu, r, 0}^p$ composed with $V|_{\ul \omega_{\mu, r}}^{\mr{adj}}\colon \ul \omega^p_{\mu, r, 0} \hookrightarrow \ul \omega_{\mu, r}^{(p)} \to \ul \omega_{\mu, r} \otimes \det(\ul \omega_{\mu, r})^{p-1}$. The map $V|_{\ul \omega_{\mu, r}}^{\mr{adj}}$ coincides with multiplication by $A_{r+1}$ when restricted to $\ul \omega_0^{p}$. This follows from the commutative diagram
\[
\begin{tikzcd}
    & \ul \omega_{\mu, r}^{(p)} \cong \hom(\overset{n-r-1}{\wedge} \ul \omega_{\mu, r}^{(p)}, \det(\ul \omega_{\mu, r})^{p}) \arrow[r, "(\wedge V)^\ast"] & \hom(\overset{n-r-1}{\wedge} \ul \omega_{\mu, r}, \det(\ul \omega_{\mu, r})^{p}) \cong \ul \omega_{\mu, r} \otimes \det(\omega_{\mu, r})^{p-1}\\
    & \ul \omega_{\mu, r, 0}^{(p)} \cong \hom(\det(\ul \omega_{\mu, r+1})^{p},\det(\ul \omega_{\mu, r})^{p}) \arrow[u, hook] \arrow[r, "A_{r+1} \cdot"] & \hom (\det \ul \omega_{\mu, r+1},\det(\ul \omega_{\mu, r})^{p}) \cong \ul \omega_{\mu, r, 0} \otimes \det(\ul \omega_{\mu, r})^{p-1} \arrow[u, hook],
\end{tikzcd}
\]
where the first row is the definition of $V|_{\ul \omega_{\mu, r}}^{\mr{adj}}$ and in the second row we are using the surjection $\wedge^{n-r-1}\ul \omega_{\mu, r} \to \det(\ul \omega_{\mu, r+1})$ and the isomorphism $\ul \omega_{\mu, r, 0} \cong \det(\ul \omega_{\mu, r+1})^{-1} \otimes \det(\ul \omega_{\mu, r})$ deduced from \ref{onemorefilt}. This, in particular, shows what we have claimed above, that $B_r$ extends to a section defined over $\overline{S}_{K, w_{r+1}}$. This is the section corresponding to the morphism $A_r \cdot p_{\mr{ur}, r} \colon H_{\mu, r}/ \ul \omega_{\mu, r} \to \ul \omega_{\mu, r} \otimes \det(\ul \omega_{\mu, r})^{p-1}$. Moreover, over  deeper strata $S_{K, w_s}, s > r + 1$ (when there are any), we have that the image of $V \colon H_{\mu, r}/\ul \omega_{\mu, r} \to \ul \omega_{\mu, r}^{(p)}$ intersects $\ul \omega_{\mu, r, 0}^{p}$ trivially, by Proposition \ref{propEOunit}. Therefore, $V|_{\ul \omega_{\mu, r}}^{\mr{adj}}$ vanishes on the image of $V \colon H_{\mu, r}/\ul \omega_{\mu, r} \to \ul \omega_{\mu, r}^{(p)}$, as one can see working through the definition of the adjugate morphism again. Let us summarise what we have proved so far.
\begin{lemma}
Over the closed stratum $\overline{S}_{K, w_{r+1}}$ the morphism $A_r \cdot p_{\mr{ur}, r} \colon H_{\mu, r} \to \ul \omega_{\mu, r} \otimes \det(\ul \omega_{\mu, r})^{p-1}$ factors through $H_{\mu, r}/ \ul \omega_{\mu, r} \cong (\det \ul \omega_{\sigma})^{-1} \otimes \delta_{\sigma}$ and thus gives a section
\[
    B_r \in H^0(\overline{S}_{K, w_{r+1}}, \ul \omega_{\mu, r} \otimes \det(\ul \omega_{\mu, r})^{p-1} \otimes \det(\ul \omega_{\sigma}) \otimes \delta_\sigma^{-1})
\]
which is nowhere vanishing over ${S}_{K, w_{r+1}}$ and identically zero on $\overline{S}_{K, w_{r+1}} \setminus {S}_{K, w_{r+1}}$.
\end{lemma}
\par The idea is to compare the action of $\theta_r$ over $\overline{S}_{K, w_{r+1}}$ and that of multiplication by $B_r$. Proposition \ref{propwtfilt} shows that they are essentially the same, up to a scalar depending on the weight and the decomposition of $\mr{gr}^{\bullet, r}(\ul \omega^{\ul k, w})$ in its graded terms. Indeed, one can consider the restriction to $\overline{S}_{K, w_{r+1}}$ of 
\[
\begin{tikzcd}
    & \psi \colon \ul \omega_{\mu, r} \arrow[r, "\nabla_r"] & H_{\mu, r} \otimes \Omega_{\overline{S}_{K, w_r}/\mb{F}}^1 \arrow[rr, "A_r \cdot p_{\mr{ur}, r} \otimes \mr{id}"] &&\ul \omega_{\mu, r} \otimes \det(\ul \omega_{\mu, r})^{p-1} \otimes \Omega_{\overline{S}_{K, w_r}/\mb{F}}^1.
\end{tikzcd}
\]
This is an $\mc{O}_{\overline{S}_{K, w_{r+1}}}$-linear morphism that factors through the restriction to $\overline{S}_{K, w_{r+1}}$ of $\widetilde{\ul{\mr{ks}}}_{\mu, r} \colon \ul \omega_{\mu, r} \to H_{\mu, r} / \ul \omega_{\mu, r} \otimes \Omega^1_{\overline{S}_{K, w_r}/\mb{F}}$. By construction of the Kodaira--Spencer map, we have the following commutative diagram
\[
    \begin{tikzcd}
        & \ul \omega_{\mu, r} \otimes (H_{\mu, r}/ \ul \omega_{\mu, r})^\vee \arrow[r, "{\widetilde{\ul{\mr{ks}}}_{\mu, r} \otimes \mr{id}}"] \arrow[d, "{\ul{\mr{ks}}_{\mu, r}}"] & (H_{\mu, r}/ \ul \omega_{\mu, r}) \otimes (H_{\mu, r}/ \ul \omega_{\mu, r})^\vee \otimes \Omega^1_{\overline{S}_{K, w_{r}}/\mb{F}} \arrow[d] \\
        & \Omega^1_{\overline{S}_{K, w_{r}}/\mb{F}} \arrow[r, equal] & \Omega^1_{\overline{S}_{K, w_{r}}/\mb{F}}.
    \end{tikzcd}
\]
In particular, over $\overline{S}_{K, w_{r+1}}$, the operator $\theta_r = (\mr{id} \otimes \ul{\mr{ks}}_{\mu, r}^{-1}) \psi \colon \ul \omega_{\mu, r} \to \mr{Sym}^2(\ul \omega_{\mu, r}) \otimes \det(\ul \omega_{\mu, r})^{p-1} \otimes \det(\ul \omega_{\sigma}) \otimes \delta_{\sigma}^{-1}$ coincides with multiplication by $B_r$. Recall that the graded sheaf $\mr{gr}^{\bullet, r}(\ul \omega^{\ul k, w})$ decomposes as a direct sum of sheaves $\mr{gr}^{\ul a, r}(\ul \omega^{\ul k, w})$. From the Leibniz rule and the fact that $\theta_r$ is zero on $\ul \omega_{0, r}$ (again, over $\overline{S}_{K, w_{r+1}}$), we deduce that for $s \in H^0(\overline{S}_{K, w_r}, \mr{gr}^{\ul a, r}(\ul \omega^{\ul k, w}))$
\[
    \theta_r(s)|_{\overline{S}_{K, w_{r+1}}} = \lvert \ul a \rvert \cdot B_r \cdot s|_{\overline{S}_{K, w_{r+1}}}.
\]
Now consider the following lemma.
\begin{lemma}
The generalised Hasse invariant $A_r$ vanishes with simple zeros along $\overline{S}_{K, w_{r+1}}$. In particular, $f \in H^0(\overline{S}_{K, w_r}, \mathrm{gr}^{\bullet, r}(\ul \omega^{\ul k, w}))$ vanishes along $\overline{S}_{K, w_{r+1}}$ if and only if there is some $g \in H^0(\overline{S}_{K, w_r}, \mathrm{gr}^{\bullet, r}(\ul \omega^{\ul k - \ul w_r, w}))$
such that $f = A_r \cdot g$.\end{lemma}
\begin{proof}
This follows from Grothendieck--Messing theory. We will use the same notations from the proof of Proposition \ref{propmoreonKScharp}. Let $s \in {S}_{r+1} \subseteq \overline{S}_r$ be a point and take $k = k(s)$. Consider $\ul{A}_v \in \overline{S}_r(k[\epsilon]/(\epsilon^2))$ any lift of $\ul{A}_s$ corresponding to $s$. This will correspond to a linear morphism of $k$-vector spaces $h_{\mu, r} \colon \ul \omega_{\mu, r} \to H/\ul \omega = H_{\mu, r}/\ul \omega_{\mu, r}$. Furthermore, $v \in \overline{S}_{r+1}(k[\epsilon]/(\epsilon^2)) \subseteq \overline{S}_r(k[\epsilon]/(\epsilon^2))$ if and only if $h_{\mu, r}$ is zero on the line $\ul \omega_{\mu, r, 0} \subseteq \ul \omega_{\mu, r}$. We want to compute $A_r(v)$. We can choose a basis $e_1, \dots, e_{n-r+1}$ of $H_{\mu, r}$ such that:
\begin{itemize}
    \item $e_1, \dots, e_{n-r}$ span $\ul \omega_{\mu, r}$,
    \item $V(e_1) = 0$ (so that $e_1$ spans $\ul \omega_{\mu, r, 0}$), $V(e_i) = e_i^{(p)}, 2\leq i \leq n-r$, $V(e_{n-r+1}) = e_1^{(p)}$ and
    \item $h_{\mu, r}(e_i) = a_i e_{n-r+1} + \ul \omega_{\mu, r}$, $a_i \in k$ for $1 \leq i \leq n-r$.
\end{itemize}
Notice that $v \in \overline{S}_{r+1}(k[\epsilon]/(\epsilon^2))$ if and only if $a_1 = 0$. We can lift each $e_i$ to $\tilde{e}_i \coloneqq e_i + \epsilon a_i e_{n-r+1}$ to obtain a basis of $\widetilde{\ul \omega}_{\mu, r}$ with respect to which $A_r(v) = \epsilon a_1 (\tilde{e}_1 \wedge \cdots \wedge \tilde{e}_{n-r})^{p-1}$. This shows that $A_r$ has simple zeroes at any point of $S_{r+1}$.
\end{proof}
In conclusion, we have proved what we wanted. 
\begin{proposition}
\label{propwtfilt}
Let $f \in H^0(\overline{S}_{K, w_r}, \mathrm{gr}^{\bullet, r}(\ul \omega^{\ul k, w}))$ and write
\[
    f = \sum_{\ul a} f_{\ul a}, \quad f_{\ul a} \in H^0(\overline{S}_{K, w_r}, \mr{gr}^{\ul a, r}(\ul \omega^{\ul k, w})).
\]
Then $\theta_r(f)$ is divisible by the Hasse invariant $A_r$ if and only if for each component $f_{\ul a}$ either $A_r \mid f_{\ul a}$ or $p \mid \lvert \ul a \rvert$. 
\end{proposition}


\bibliographystyle{amsplain}
\bibliography{bib.bib}

\end{document}